\documentclass[11pt]{article}%
\usepackage{graphicx}
\usepackage{amsmath}
\usepackage{fullpage}
\usepackage[notref,notcite]{showkeys}
\usepackage{amsfonts}
\usepackage{amssymb}%
\providecommand{\U}[1]{\protect\rule{.1in}{.1in}}
\newtheorem{theorem}{Theorem}

\newtheorem{corollary}[theorem]{Corollary}

\newtheorem{definition}[theorem]{Definition}
\newtheorem{example}[theorem]{Example}

\newtheorem{lemma}[theorem]{Lemma}

\newtheorem{problem}[theorem]{Problem}
\newtheorem{proposition}[theorem]{Proposition}
\newtheorem{remark}[theorem]{Remark}

\newenvironment{proof}[1][Proof]{\textbf{#1.} }{\ \rule{0.5em}{0.5em}}
\numberwithin{theorem}{section} \numberwithin{equation}{section}
\begin{document}

\title{Relations and radicals in abstract lattices and in lattices of subspaces of
Banach spaces and ideals of Banach algebras. Amitsur's theory revisited. }
\author{Edward Kissin, Victor Shulman and Yuri Turovskii}
\maketitle
\date{}

\begin{abstract}
We refine Amitsur's theory of radicals in
complete lattices and apply the
obtained results to the theory of radicals in the lattices of subspaces of
Banach spaces and in the lattices of ideals of Banach and C*-algebras and of
Banach Lie algebras.
\end{abstract}

\tableofcontents

\section{Introduction}

In his research of the radical theory of algebras and rings, Amitsur
\cite{Am2, Am3} discovered that a significant part of the results of this
theory can be formulated and proved in terms of the general theory of
lattices. In \cite{Am} he developed the theory of radicals for relations in
lattices that was used in various areas of algebra: group theory,
non-associative rings, Lie algebras, universal algebras, etc.

Later Dixon \cite{Di1} initiated the radical approach to some problems of
functional analysis and laid the basis of the theory of topological radicals
of Banach algebras. This theory was further developed and applied to the
theory of invariant subspaces of operator algebras and to classification of
Banach and operator Lie algebras in \cite{ST}, \cite{KST3}, \cite{KST5}.

In this paper we investigate both aspects of the theory of radicals. In part
one: Sections 2 to 7 we revise and refine Amitsur's theory of radicals in
complete abstract lattices. In part two: Sections 8 and 9 we apply the
obtained results to the theory of radicals in the lattices of subspaces of
Banach spaces and in the lattices of ideals of Banach and C*-algebras and of
Banach Lie algebras.

Recall that a partially ordered set $(Q,\leq)$ with a reflexive,
anti-symmetric, transitive relation $\leq$ is a lattice if all $a,b\in Q$ have
a least upper bound $a\vee b$ and a greatest lower bound $a\wedge b.$ It is
\textit{complete }if each $G\subseteq Q$ has a least upper bound $\vee G$ and
a greatest lower bound $\wedge G$. Set $\mathbf{0}=\wedge Q$ and
$\mathbf{1}=\vee Q.$ A transitive relation in $Q$ is an \textit{order}; a
relation $\prec_{_{1}}$ in $Q$ is \textit{stronger }than a relation $\prec$
if
\begin{equation}
a\prec_{_{1}}b\text{ implies }a\prec b\text{ for }a,b\in Q\text{ }(\text{we
write }\prec_{_{1}}\text{ }\subseteq\text{ }\prec). \label{0}%
\end{equation}
We will only consider reflexive relations $\ll$ in ($Q,\leq)$ stronger than
$\leq$ and denote by Ref$\left(  Q\right)  $ the family of all such relations.
For $x\leq y$ in $Q,$ set $[x,y]=\{z\in Q$: $x\leq z\leq y\}.$

Amitsur \cite{Am} studied $\mathbf{H}$- and dual $\mathbf{H}$-relations $\ll$
in complete lattices $Q$ and, using a special procedure, constructed the
$\mathbf{R}$-order $\ll^{\triangleright}$ from an $\mathbf{H}$-relation and
the dual $\mathbf{R}$-order $\ll^{\triangleleft}$ from a dual $\mathbf{H}%
$-relation (Definition \ref{D2.2}). He proved that $\ll^{\triangleright}$ has
a unique $\ll^{\triangleright}$-\textit{radical} $\mathfrak{r}_{_{\left[
a,b\right]  }}$ in each interval $\left[  a,b\right]  $ in $Q$ and
$\ll^{\triangleleft}$ has a unique \textit{dual} $\ll^{\triangleleft}%
$-\textit{radical} $\mathfrak{p}_{_{\left[  a,b\right]  }}$.

In our paper (Theorems \ref{eur} and \ref{eur1}) we establish that even weaker
relations ($\mathbf{T}$- and dual $\mathbf{T}$-orders) have unique radicals in
each $[a,b].$ Moreover, for a $\mathbf{T}$-order, the map $\mathfrak{r}$:
$a\in Q\rightarrow\mathfrak{r}_{_{\left[  a,\mathbf{1}\right]  }}$ is
pre-radical map; for a dual $\mathbf{T}$-order, the map $\mathfrak{p}$: $a\in
Q\rightarrow\mathfrak{p}_{_{\left[  \mathbf{0},a\right]  }}$ is a dual
pre-radical maps ((\ref{3.41}) and (\ref{2.0})) on $Q.$ For $\mathbf{R}%
$-orders and dual $\mathbf{R}$-orders they are, respectively, radical and dual
radical maps ((\ref{f3.1}) and (\ref{t3.2})) on $Q$ (Theorems \ref{T2.2} and
\ref{T2.2'}). If a relation $\ll$ is both a $\mathbf{T}$- and a dual
$\mathbf{T}$-order, it coincides "locally" with $\leq$, i.e., $Q$ is a union
of mutually disjoint intervals $[a_{\lambda},b_{\lambda}],$ $\lambda\in
\Lambda,$ such that $x\ll y$ if and only if $x,y\in\lbrack a_{\lambda
},b_{\lambda}]$ for some $\lambda\in\Lambda,$ and $x\leq y$ (Theorem
\ref{T2.12}).

In Section 4 we investigate the relations $\ll^{\triangleright}$ and
$\ll^{\triangleleft}$ constructed from arbitrary relations $\ll$. We compare
them to other naturally constructed relations $\ll^{\text{lo}}$ and
$\ll^{\text{up}}$ (Definition \ref{D2}) and prove that each interval $[a,b]$
in $Q$ has $\ll^{\triangleright}$- and dual $\ll^{\triangleleft}$-radicals
which are not, however, unique.

For an $\mathbf{H}$-relation $\ll,$ we show in Section 5 that $\ll
^{\triangleright}$ is an $\mathbf{R}$-order, it coincides with $\ll
^{\text{up}}$ and each $[a,b]$ has a unique $\ll^{\triangleright}$-radical.
Similarly, for a dual $\mathbf{H}$-relation $\ll,$ $\ll^{\triangleleft}$ is a
dual $\mathbf{R}$-order coinciding with $\ll^{\text{lo}}$ and each $[a,b]$ has
a unique dual $\ll^{\triangleleft}$-radical. We prove that $\ll$ is an
$\mathbf{R}$-order (a dual $\mathbf{R}$-order) if and only if $\ll$ $=$
$\ll^{\triangleright}$ ($\ll$ $=$ $\ll^{\triangleleft}).$

In Section 6 we investigate two incompatible relations: the gap
($<_{\mathfrak{g}})$ and continuous ($<_{\mathfrak{c}})$ relations in lattices
$Q$. We show that if $Q$ is modular then $<_{\mathfrak{g}}$ is an $\mathbf{H}%
$- and a dual $\mathbf{H}$-relation. If $Q$ is modular and has (JID) and (MID)
(see (\ref{6.9})) then $<_{\mathfrak{c}}$ is an $\mathbf{R}$- and a dual
$\mathbf{R}$-order and $Q$ is a union of disjoint sets without gaps.

Section 7 is devoted to the study of enveloping and inscribing sets in
lattices $Q.$ We prove that there is a bijection between all enveloping sets
and all radical maps in $Q,$ and between all inscribing sets and all dual
radical maps in $Q$. We consider some enveloping sets in the lattice Ref($Q).$
In particular, we show that the sets of all $\mathbf{T}$-orders, of all dual
$\mathbf{T}$-orders, of all $\mathbf{R}$-orders, of all dual $\mathbf{R}%
$-orders are enveloping in Ref($Q).$ We also study transfinite extensions of
relations and, in particular, of the relations $<_{\mathfrak{g}}$ and
$<_{\mathfrak{c}}.$

An important place in the theory of operator algebras is occupied by the study
of lattices of subspaces of Banach spaces $X$ and, especially, of the lattices
Lat $\mathcal{A}$ of invariant subspaces of operator algebras $\mathcal{A}$ on
$X$. In Section 8 the above results are used to study such lattices.

Denote by Ln($X)$ the lattice of all \textit{linear} subspaces and by Cl($X)$
the lattice of all \textit{closed} subspaces of $X.$ In Theorem \ref{T4.6} we
describe ascending and descending $\ll$-series of subspaces in sublattices of
Ln($X)$ and Cl($X)$ with respect to $\mathbf{H}$- and dual $\mathbf{H}%
$-relations $\ll,$ and their $\ll^{\triangleright}$- and $\ll^{\triangleleft}%
$-radicals $\mathfrak{r}$ and $\mathfrak{p}$, respectively. In the sublattices
Lat $\mathcal{A}$ of Cl($X)$ these radicals are superinvariant (Proposition
\ref{P9.4}), i.e., they are invariant for the operator Lie algebra
\[
\text{Nor }\mathcal{A}=\{S\in B(X):SA-AS\in\mathcal{A}\text{ for }%
A\in\mathcal{A}\}.
\]
We concentrate on the study of the gap relation $<_{\mathfrak{g}}$ and the
relations $\{\ll_{n}\}_{1\leq n\leq\infty}$ defined in Ln$(X)$ and Cl$(X)$ by
the condition $L\ll_{n}M$ if $\dim(M/L)<n.$ As Ln($X)$ is a modular lattice,
$<_{\mathfrak{g}}$ and all $\ll_{n}$ are both $\mathbf{H}$- and dual
$\mathbf{H}$-relations ($\mathbf{HH}$-relations) in it (Proposition \ref{P9.1}
and Corollary \ref{C9.6}).

The lattice Cl($X)$ is not modular and, although all $\ll_{n}$ are still
$\mathbf{HH}$-relations in it$,$ there are sublattices where $<_{\mathfrak{g}%
}$ is neither $\mathbf{H}$-, nor dual $\mathbf{H}$-relation (Corollary
\ref{C9.1}). The main obstacle is the fact that the sum of subspaces from
Cl($X)$ is not necessarily closed$.$ We construct in Cl($X)$ an $\mathbf{H}%
$-relation $\sqsubset_{\mathfrak{g}}$ and a dual $\mathbf{H}$-relation
$\prec_{\mathfrak{g}}$ stronger than $<_{\mathfrak{g}}$ that give us the
$\sqsubset_{\mathfrak{g}}^{\triangleright}$-radical and the dual
$\prec_{\mathfrak{g}}^{\triangleleft}$-radical in each sublattice of Cl($X).$

For a Hilbert space $X,$ we define in (\ref{5,7}) another class of the
$\mathbf{HH}$-relations $\ll_{_{n}}^{\bot},$ $0\leq n\leq\infty,$ in Cl($X)$
(Theorem \ref{T9.3}). Moreover, each $\mathbf{HH}$-relation in Cl$(X)$ is
either $\ll_{n},$ or $\ll_{_{n}}^{\bot}$ for some $n$ (see \cite{K3} and
Theorem \ref{T6.3}). Each commutative subspace lattice (CSL) $Q$ in Cl($X)$
has properties (JID) and (MID) (see (\ref{6.4})). So it is modular,
$<_{\mathfrak{g}}=$ $\sqsubset_{\mathfrak{g}}=$ $\prec_{\mathfrak{g}}$ is an
$\mathbf{HH}$-relation in $Q,$ and $Q$ is a union of disjoint intervals each
of which has no gaps (Theorem \ref{T9.1}).

The study of ideals of Banach and Banach Lie algebras constitutes one of the
main tools of research. In \cite{KST3} the authors used the structure of
chains of ideals in Banach Lie algebras to develop the radical theory of these
algebras. In this paper we study chains of ideals generated by $\mathbf{H}$-
and dual $\mathbf{H}$-relations in the lattices Id$_{A}$ of all closed ideals
of Banach algebras $A.$

As $\ll_{\infty}$ is $\mathbf{HH}$-order in Id$_{A},$ each sublattice of
Id$_{A}$ has the unique $\ll_{_{\infty}}^{\triangleright}$-radical and the
dual $\ll_{_{\infty}}^{\triangleleft}$-radical$.$ Denoting by $\Sigma_{A}$ the
set of all subalgebras of finite codimension in a Banach algebra $A$, we show
in Proposition \ref{C-Laf} that $\cap_{_{S\in\Sigma_{A}}}S=\cap\{I\in$
Id$_{A}$: $\dim(A/I)<\infty\}.$

The relation $<_{\mathfrak{g}}$ in Id$_{A}$ is not always an $\mathbf{H}$-, or
a dual $\mathbf{H}$-relation, since the sum of closed ideals of $A$ is not
necessarily closed. If, however, each ideal of $A$ has a bounded left or right
approximate identity, then $I+J$ is closed for all $I,J\in$ Id$_{A},$ so that
$<_{\mathfrak{g}}$ is an $\mathbf{HH}$-relation in Id$_{A}$ (Corollary
\ref{C3.4}).

Let $LR\mathcal{(}A)$ be the operator algebra generated by all operators of
left and right multiplication by elements from $A.$ Then Id$_{A}=$ Lat
$LR\mathcal{(}A)$ is a sublattice of Cl($A)$, the Lie algebra Nor
$LR\mathcal{(}A)$ contains all derivations of $A$ and operators from Nor
$LR\mathcal{(}A)$ map $\overline{I^{2}}$ in $I,$ $I\in$ Id$_{A}$ (Theorem
\ref{T9.6}).

If $\cap_{_{S\in\Sigma_{A}}}S=\{0\},$ for a Banach Lie algebra $A$, then
(Theorem \ref{comdiff}) the dual $\ll_{_{\infty}}^{\triangleleft}$-radical
$\mathfrak{p}$ in the lattice of all characteristic Lie ideals of $A$
(invariant for all derivations of $A)$ is $\{0\}.$ So there is a descending
series $(I_{\lambda})_{1\leq\lambda\leq\gamma}$ of characteristic Lie ideals
of $A$ such that $\dim(I_{\lambda}/I_{\lambda+1})<\infty$ for $\lambda
\neq\gamma,$ $I_{1}=A$ and $\mathfrak{p}=I_{\gamma}=\{0\}.$

Section 9.2 is devoted to the study of $\mathbf{H}$- and dual $\mathbf{H}%
$-relations in the lattices Id$_{A}$ of all ideals of C*-algebras $A.$ In this
case Id$_{A}$ is a modular lattice and $<_{\mathfrak{g}}$ is an $\mathbf{HH}%
$-relation in Id$_{A}.$ The lattice Id$_{A}$ has many $\mathbf{H}$-relations
that can be obtained in the following way. Let $\mathfrak{A}$ be the set of
all C*-algebras. A subclass $P$ of $\mathfrak{A}$ is a \textit{property}, if
$\{0\}\in P$ and $A\in P$ implies $B\in P$ for all $B\approx A$. Each property
$P$ generates the relation $\ll_{_{P}}$ in Id$_{A}$ by $I\ll_{_{P}}J$ if
$I\subseteq J$ in Id$_{A}$ and $J/I\in P.$

A property $P$ is lower stable if $A\in P$ implies Id$_{A}\subset P;$\emph{
}$P$ is upper stable\textbf{ }if $A\in P$ implies that the quotients $A/I\in
P$ for all $I\in$ Id$_{A}.$ In Theorem \ref{T9.4} we prove that $P$ is upper
(lower) stable if and only if $\ll_{_{P}}$ is an $\mathbf{H}$-relation (a dual
$\mathbf{H}$-relation). Some characteristics of the $\ll_{_{P}}%
^{\triangleright}$-radicals $\mathfrak{r}_{_{P}}$ and the dual $\ll_{_{P}%
}^{\triangleleft}$-radicals $\mathfrak{p}_{_{P}}$ in Id$_{A}$ are discussed in
Theorem \ref{T9.5}. These radicals are invariant for all automorphisms of $A$
(Corollary \ref{C9.8}). Since many properties in $\mathfrak{A}$ are upper, or
lower stable, we have a large variety of $\mathbf{H}$- and dual $\mathbf{H}%
$-relations in Id$_{A}.$ These relations were investigated in \cite{KST4}; in
this paper we briefly consider some of them.

For example, the classes $CCR$ and $GCR$ of all CCR- and GCR-algebras are
lower and upper stable properties, while the class of all NGCR-algebras is a
lower, but not upper stable property (see \cite{D}). So $\ll_{_{CCR}}$ and
$\ll_{_{GCR}}$ are $\mathbf{HH}$-relations. This gives a well-known result
that the radical $\mathfrak{r}_{_{CCR}}$ of a C*-algebra $A$ is the largest
GCR-ideal and $A/\mathfrak{r}_{_{CCR}}$ has no CCR-ideals. Moreover, if
$\mathfrak{r}_{_{CCR}}\nsubseteq I\neq A$ then $J/I$ is a CCR-algebra for some
$I\subsetneqq J\in$ Id$_{A}.$

We also consider the classes $RZ$ of all real rank zero, $AF$ of all
approximately finite-dimensional and $NU$ of all nuclear C*-algebras. They are
lower and upper stable properties, so that $\ll_{_{RZ}},$ $\ll_{_{AF}},$
$\ll_{_{NU}}$ are $\mathbf{HH}$-relations in Id$_{A}$ for all C*-algebras $A$
(Corollary \ref{C9.7}). While the relations $\ll_{_{CCR}}$ and $\ll_{_{RZ}}$
are not transitive, the relations $\ll_{_{AF}}$ and $\ll_{_{NU}}$ are
transitive. Moreover, they are $\mathbf{R}$-orders in all Id$_{A}.$ So the
corresponding radicals are, respectively, the largest $AF$-algebra and the
largest nuclear algebra in $A$ (Corollary \ref{C8.4}) (this result was
initially obtained in \cite{ST}).

\section{Radicals and $\mathbf{T}$-orders in complete lattices}

Let $(Q,\leq)$ be a complete lattice. Set $\mathbf{0}=\wedge Q$ and
$\mathbf{1}=\vee Q.$ For each subset $G\subseteq Q,$ its $\wedge
$\textit{-completion} $G^{\wedge}$ and $\vee$\textit{-completion}\textbf{
}$G^{\vee}$ are defined by%

\begin{equation}
G^{\wedge}=\left\{  \wedge N:\varnothing\neq N\subseteq G\right\}  \text{ and
}G^{\vee}=\left\{  \vee N:\varnothing\neq N\subseteq G\right\}  . \label{1.10}%
\end{equation}
A set $G$ is $\wedge$-\textit{complete} if $G=G^{\wedge},$ $\vee
$-\textit{complete} if $G=G^{\vee},$ and \textit{complete} if $G^{\vee
}=G^{\wedge}=G$.

Let $(Q,\leq)$ be a complete lattice, let $x<y$ in $Q$ and let $\ll$ be a
relation in Ref($Q)$. Set
\begin{align}
(x,y]  &  =[x,y]\diagdown\{x\},\text{ }[x,y)=[x,y]\diagdown\{y\},\nonumber\\
\left[  a,\ll\right]   &  =\left\{  x\in Q\text{\textbf{: }}a\ll x\right\}
,\text{ }\left[  \ll,a\right]  =\left\{  x\in Q\text{: }x\ll a\right\}  \text{
for each }a\in Q. \label{1}%
\end{align}

\begin{definition}
\label{D1.1}A relation $\ll$ from \emph{Ref(}$Q)$ is called
\textbf{up-contiguous} if $a\ll b$ implies $[a,b]\subseteq\lbrack\ll
,b]$\emph{;}\smallskip

\textbf{down-contiguous} if $a\ll b$ implies $[a,b]\subseteq\lbrack
a,\ll];\smallskip$

\textbf{up-expanded} if $\left[  a,\ll\right]  $ is $\vee$-complete$,$
\textbf{down-expanded} if $\left[  \ll,a\right]  $ is $\wedge$-complete for
all $a\in Q$\emph{;}\smallskip

\textbf{contiguous}$,$\textbf{ expanded} if $\ll$ satisfies the corresponding
up- and down-condition.
\end{definition}

Note that if a relation has one of the properties defined above, its
restriction to any interval also has it. We consider now some examples of
these relations.

\begin{example}
\label{E1}\emph{1) Let }$X=[0,1]\subset\mathbb{R}$\emph{ and }$Q=P(X)$\emph{
be its power set -- the lattice of all subsets of }$X$\emph{ with }$\leq
$\emph{ }$=$\emph{ }$\subseteq.$\emph{ For }$A,B\in Q,$\emph{ we write }$A\ll
B$\emph{ if }$A\subseteq B\subseteq\overline{A.}$\emph{ Then }$\ll$\emph{ is
contiguous: if }$A\ll B$\emph{ then }$[A,B]\subseteq\lbrack A,\ll]\cap
\lbrack\ll,B],$\emph{ since }$\overline{C}=\overline{A}$\emph{ for each }%
$C\in\lbrack A,B].$\emph{ It is also up-expanded, as }$\vee\lbrack
A,\ll]=\overline{A}\in\lbrack A,\ll].$\emph{ However, }$\ll$\emph{ is not
down-expanded, as }$\wedge\lbrack\ll,X]=\varnothing\notin\lbrack
\ll,X].\medskip$

\emph{2) \ Let }$X$\emph{ be the set of all Lebesgue measurable subsets of
}$\mathbb{R}$\emph{ and }$Q$\emph{ be the set of all equivalence classes in
}$X.$\emph{ For }$A,B\in Q,$\emph{ we write }$A\leq B$\emph{ if }%
$\mu(A\diagdown B)=0,$\emph{ and we write }$A\ll B$\emph{ if }$A\leq B$\emph{
and }$\mu(B\diagdown A)<\infty.$\emph{ Then }$\ll$\emph{ is contiguous, but
neither up-, nor down-expanded.\medskip}

\emph{3) Let }$Q_{_{N}}=\{1,...,N\}$\emph{ for }$N\in N,$\emph{ the order
}$\leq$\emph{ is defined as usual and }$n\ll m$\emph{ if }$n$\emph{ is a
divisor of }$m.$\emph{ Then }$\ll$\emph{ is expanded, but not contiguous.
\ \ }$\blacksquare$
\end{example}

\begin{definition}
\label{D3.2}\emph{(i) }A transitive relation is a $\mathbf{T}$\textbf{-order}
if it is up-contiguous and up-expanded;\smallskip

\emph{(ii) } A transitive relation is a dual $\mathbf{T}$-\textbf{order} if it
is down-contiguous and down-expanded.\smallskip

\emph{(iii)} $\ll$ is a $\mathbf{TT}$-\textbf{order}, if it is a $T$- and a
dual $T$-order, i.e., $\ll$ is contiguous and expanded\emph{.}
\end{definition}

By the Duality Principle \cite[Theorem $1.3^{\prime}$]{Sk}, the results for
down-conditions follow from the corresponding results for up-conditions and
vice versa. The results for dual $\mathbf{T}$-orders follow from the
corresponding results for $\mathbf{T}$-orders and vice versa.

Following \cite{Am}, for $\ll$ $\in$ Ref($Q),$\emph{ }define its \textit{lower
}and \textit{upper complement }relations $\overrightarrow{\ll}$ and
$\overleftarrow{\ll}$%
\begin{equation}
a\text{ }\overleftarrow{\ll}\text{ }b\text{ if }\left[  a,\ll\right]
\cap\,[a,b]=\left\{  a\right\}  ;\text{ and }a\text{ }\overrightarrow{\ll
}\text{ }b\text{ if }\left[  \ll,b\right]  \,\cap\,[a,b]=\left\{  b\right\}
\text{ for }a\leq b\text{ in }Q. \label{A1}%
\end{equation}
An element $\mathfrak{r}\in Q$ is called a $\ll$\textit{-radical} if
$\mathbf{0}\ll\mathfrak{r}$ $\overleftarrow{\ll}$ $\mathbf{1;}$ an element
$\mathfrak{p}$ is called a dual $\ll$\textit{-radical} if $\mathbf{0}%
$\ $\overrightarrow{\ll}$ $\mathfrak{p}\ll\mathbf{1}.$ In particular, for an
interval $[a,b]\subseteq Q,$%
\begin{align}
\mathfrak{r}  &  \in\left[  a,b\right]  \text{ is a }\ll
\text{\textit{-radical} in }\left[  a,b\right]  \text{ if }a\ll\mathfrak{r}%
\text{ }\overleftarrow{\ll}\text{ }b,\label{2.10}\\
\mathfrak{p}  &  \in\left[  a,b\right]  \text{ is a \textit{dual}}%
\ll\text{-\textit{radical} in }\left[  a,b\right]  \text{ if }a\text{\ }%
\overrightarrow{\ll}\text{ }\mathfrak{p}\ll b. \label{2.11}%
\end{align}
The set of radicals may be empty or have many elements. For example,\emph{
}let $Q=[0,1]\subseteq\mathbb{R}$ and $\ll$ $\in$ Ref($Q)$ be such that $x\ll
y$ only if $0\neq x\leq y\neq1.$ Then the set $[a,1],$ $a\neq0,$ has no $\ll
$-radicals and the dual $\ll$-radical $\mathfrak{p}=1;$ the set $[0,b],$
$b\neq1,$ has no dual $\ll$-radicals and the $\ll$-radical $\mathfrak{r}=0$.

\begin{theorem}
\label{eur}\emph{(i)\ }If $\ll$ is an up-expanded order\emph{,} $\mathfrak{r}%
=\vee(\left[  a,b\right]  \cap\left[  a,\ll\right]  )$ is a $\ll$-radical in
$\left[  a,b\right]  \subseteq Q.$\smallskip

\emph{(ii)\ \ }If $\ll$ is a $\mathbf{T}$-order then $\mathfrak{r}$ in
\emph{(i) }is a unique radical and $[a,\mathfrak{r}]=[a,b]\cap\lbrack
\ll,\mathfrak{r}].\smallskip$

\emph{(iii)}\ If each $[a,b]\subseteq Q$ has a unique $\ll$-radical then $\ll$
is up-expanded.\smallskip

\emph{(iv) }An up-contiguous order $\ll$ is a $\mathbf{T}$-order if and only
if each $[a,b]\subseteq Q$ has a unique $\ll$-radical.
\end{theorem}

\begin{proof}
(i) As $\ll$ is up-expanded, $\left[  a,\ll\right]  \cap\left[  a,b\right]  $
is $\vee$-complete. So it contains $\mathfrak{r}.$ Hence $a\ll\mathfrak{r}$.
If $\mathfrak{r}\ll y$ for some $y\in(\mathfrak{r},b],$ then $a\ll y$ by
transitivity of $\ll$ -- a contradiction, since $\mathfrak{r}$ is the largest
element in $\left[  a,\ll\right]  \cap\left[  a,b\right]  .$ Thus
$\mathfrak{r}$ $\overleftarrow{\ll}$ $b$. So $\mathfrak{r}$ is a $\ll$-radical
in $[a,b].$

(ii) Let $\ll$ be also up-contiguous and $a\ll z$ $\overleftarrow{\ll}$ $b$
for some $z.$ Then $z\leq\mathfrak{r}$, since $\mathfrak{r}$ is the largest
element in $\left[  a,\ll\right]  \cap\left[  a,b\right]  .$ As $a\ll
\mathfrak{r}$ and $\ll$ is up-contiguous, $[a,\mathfrak{r}]\subseteq\lbrack
\ll,\mathfrak{r}].$ So $z\ll\mathfrak{r}$. As $z$ $\overleftarrow{\ll}$ $b,$
we have $z=\mathfrak{r}$. Thus $\mathfrak{r}$ is a unique $\ll$-radical.

(iii) Let $a\in Q,$ $G\subseteq\lbrack a,\ll]$ and $b=\vee G.$ Let
$\mathfrak{r}$ be the $\ll$-radical in $[a,b]$: $a\ll\mathfrak{r}$
$\overleftarrow{\ll}$ $b$. If $x\in G$ then $a\ll x\leq b.$ Let $\mathfrak{r}%
_{x}$ be the $\ll$-radical in $[x,b]$: $x\ll\mathfrak{r}_{x}$
$\overleftarrow{\ll}$ $b$. As $\ll$\ is an order, $a\ll\mathfrak{r}_{x}$
$\overleftarrow{\ll}$ $b.$ So $\mathfrak{r}_{x}$ is a $\ll$-radical in
$[a,b].$ As the $\ll$-radical in $[a,b]$ is unique, $\mathfrak{r}%
=\mathfrak{r}_{x}.$ Thus $a\ll x\ll\mathfrak{r}$ $\overleftarrow{\ll}$ $b$ for
all $x\in G.$ Hence $b=\vee G\leq\mathfrak{r}.$ Thus $b=\mathfrak{r}.$ So
$a\ll b.$ Therefore $b\in\lbrack a,\ll],$ so that $\ll$ is up-expanded.

Part (iv) follows from (ii) and (iii).\bigskip
\end{proof}

By duality we get the following result.

\begin{theorem}
\label{eur1}\emph{(i)\ }If $\ll$ is a down-expanded order\emph{,}
$\mathfrak{p}=\wedge(\left[  a,b\right]  \cap\left[  \ll,b\right]  )$ is a
dual $\ll$-radical in $\left[  a,b\right]  $\emph{.}\smallskip

\emph{(ii) \ }If $\ll$ is a dual $\mathbf{T}$-order then $\mathfrak{p}$ in
\emph{(i) }is unique and $[\mathfrak{p},b]=\left[  a,b\right]  \cap\left[
\mathfrak{p},\ll\right]  .$

\emph{(iii) }If each $[a,b]\subseteq Q$ has a unique dual $\ll$-radical then
$\ll$ is down-expanded.\smallskip

\emph{(iv) }A down-contiguous order $\ll$ is a dual $\mathbf{T}$-order if and
only if each $[a,b]\subseteq Q$ has a unique dual $\ll$-radical.
\end{theorem}

As the following examples show, the results of Theorems \ref{eur} and
\ref{eur1} can not be strengthened.

\begin{example}
\label{E2.1}\emph{(i) }The existence of a unique $\ll$-radical in each
$[a,b]\subseteq Q$ only guarantees that $\ll$ is up-expanded, but not that it
is up-contiguous. \emph{Indeed, let }$\ll$\emph{ be the reflexive relation in
}$Q=[0,1]\subseteq\mathbb{R}$\emph{ with only one non-trivial pair }$0\ll
1.$\emph{ Each }$[a,b]\subseteq Q$\emph{ has a unique }$\ll$\emph{-radical:
}$\mathfrak{r}_{[a,b]}=a$\emph{ if }$[a,b]\neq\lbrack0,1],$\emph{ and
}$\mathfrak{r}_{[0,1]}=1.$\emph{ The relation }$\ll$\emph{ is up-expanded, but
not up-contiguous (not a }$\mathbf{T}$\emph{-order): }$0\ll1$\emph{ but}
$Q=[0,1]\nsubseteq\lbrack\ll,1]=\{0,1\}.\medskip$

\emph{(ii) }The existence of $\ll$-radicals in each $[a,b]\subseteq Q$ does
not even guarantee that $\ll$ is up-expanded. \emph{Indeed, let }%
$Q=\{\mathbf{0},a,b,\mathbf{1}\},$ $\mathbf{0}<a<\mathbf{1}$ \emph{and}
$\mathbf{0}<b<\mathbf{1}.$ \emph{Let }$\ll$\emph{ be the reflexive relation in
}$Q$\emph{ such that }$\mathbf{0}\ll a$\emph{ and }$\mathbf{0}\ll b.$
\emph{Then }$[\mathbf{0},\mathbf{1}]$\emph{ has two }$\ll$\emph{-radicals:
}$\mathfrak{r}_{_{1}}=a$\emph{ and }$\mathfrak{r}_{_{2}}=b.$\emph{ The
relation }$\ll$\emph{ is not up-expanded, since }$[\mathbf{0},\ll
]=\{\mathbf{0},a,b\}$ \emph{and }$\vee\lbrack\mathbf{0},\ll]=a\vee
b=\mathbf{1\notin}[\mathbf{0},\ll].$\medskip

\emph{(iii) }If $\ll$ is up-expanded but not up-contiguous then all
$[a,b]\subseteq Q$ have radicals but not necessarily unique.\emph{ Indeed, let
}$\ll$\emph{ be the reflexive relation in }$Q=[0,1]\subseteq\mathbb{R}$\emph{
such that }$0\ll1$\emph{ and }$0\ll\frac{1}{2}.$\emph{ It is up-expanded, but
not up-contiguous (not a }$\mathbf{T}$\emph{-order): }$0\ll1$\emph{ but
}$Q=[0,1]\nsubseteq\lbrack\ll,1]=\{0,1\}.$\emph{ Each }$[a,b]\neq\lbrack
0,1]$\emph{ has a unique }$\ll$\emph{-radical, while }$[0,1]$\emph{ has two
}$\ll$\emph{-radicals: }$\frac{1}{2}$\emph{ and }$1.$ \ \ \ \ $\blacksquare$
\end{example}

For a dual $\mathbf{T}$-order $\ll$ in $Q,$ denote by $\mathfrak{p}(b)$ the
unique dual $\ll$-radical in\ $[\mathbf{0},b]\subseteq Q$:%
\begin{equation}
\mathfrak{p}(b)=\mathfrak{p}_{\ll}(b)=\wedge\lbrack\ll,b],\text{ }%
\mathbf{0}\text{ }\overrightarrow{\ll}\text{ }\mathfrak{p}(b)\text{ }\ll\text{
}b\text{ and }\left[  \mathbf{0},b\right]  \cap\left[  \mathfrak{p}%
,\ll\right]  =[\mathfrak{p},b]. \label{2.13}%
\end{equation}
For a $\mathbf{T}$-order $\ll$ in $Q,$ denote by $\mathfrak{r}(a)$ the unique
$\ll$-radical in\ $[a,\mathbf{1}]$:%
\begin{equation}
\mathfrak{r}(a)=\mathfrak{r}_{\ll}(a)=\vee\lbrack a,\ll],\text{ }a\text{ }%
\ll\text{ }\mathfrak{r}(a)\text{ }\overleftarrow{\ll}\text{ }\mathbf{1}\text{
and }[a,\mathbf{1}]\cap\lbrack\ll,\mathfrak{r}]=[a,\mathfrak{r}]. \label{2.14}%
\end{equation}

\begin{lemma}
\label{L2.1}\emph{(i) }Let $\ll$ be a dual $\mathbf{T}$-order in $Q.$ Then
\begin{align}
\mathfrak{p}(\mathfrak{p}(b))  &  =\mathfrak{p}(b)\text{ for all }b\in
Q;\text{ \ \ }a\ll b\text{ implies }\mathfrak{p}(b)=\mathfrak{p}(a)\ll
a;\label{2,2}\\
\mathfrak{p}(b)  &  \leq a\leq b\text{ implies }\mathfrak{p}(b)=\mathfrak{p}%
(a)\ll a. \label{2,5}%
\end{align}

If $\ll$ is also up-contiguous then $\mathfrak{p}(b)\leq a\leq b$ implies
$a\ll b.\smallskip$

\emph{(ii) }Let $\ll$ be a $\mathbf{T}$-order in $Q.$ Then
\begin{align}
\mathfrak{r}(\mathfrak{r}(a))  &  =\mathfrak{r}(a)\text{ for all }a\in
Q;\text{ \ \ }a\ll b\text{ implies }b\ll\mathfrak{r}(b)=\mathfrak{r}%
(a);\nonumber\\
a  &  \leq b\leq\mathfrak{r}(a)\text{ implies }b\ll\mathfrak{r}%
(b)=\mathfrak{r}(a). \label{2,3}%
\end{align}

If $\ll$ is also down-contiguous then $a\leq b\leq\mathfrak{r}(a)$ implies
$a\ll b.$
\end{lemma}

\begin{proof}
By (\ref{2.13}), $\mathfrak{p}(\mathfrak{p}(b))\ll\mathfrak{p}(b)\ll b.$ As
$\ll$ is transitive, $\mathfrak{p}(\mathfrak{p}(b))\ll b.$ By (\ref{2.13}),
$\mathfrak{p}(b)$ is the smallest element in $[\ll,b]$. Hence $\mathfrak{p}%
(b)\leq\mathfrak{p}(\mathfrak{p}(b))$. Thus $\mathfrak{p}(\mathfrak{p}%
(b))=\mathfrak{p}(b).$

Let $a\ll b.$ As $\mathfrak{p}(b)$ is the smallest element in $[\ll,b]$, we
have $\mathfrak{p}\left(  b\right)  \leq a\leq b$. So, by (\ref{2.13}),
$\mathfrak{p}\left(  b\right)  \ll a$. Thus $a\ll b$ implies $\mathfrak{p}%
\left(  b\right)  \ll a$. Hence $\mathfrak{p}\left(  b\right)  \ll a$ implies
$\mathfrak{p}(a)\ll\mathfrak{p}(b).$ So $\mathfrak{p}(a)\ll\mathfrak{p}(b)\ll
b.$ As $\ll$ is transitive, $\mathfrak{p}(a)\ll b$. By (\ref{2.13}),
$\mathfrak{p}\left(  b\right)  \leq\mathfrak{p}(a)$. Thus $\mathfrak{p}%
(b)=\mathfrak{p}(a)$. As $\mathfrak{p}(a)\ll a,$ (\ref{2,2}) is proved.

Let $\mathfrak{p}(b)\leq a\leq b$. By (\ref{2.13}), $\mathfrak{p}(b)\ll a$. By
(\ref{2,2}), $\mathfrak{p}(a)=\mathfrak{p}(\mathfrak{p}(b))=\mathfrak{p}(b).$

If, in addition, $\ll$ is up-contiguous then $\mathfrak{p}(b)\ll b$ and
$a\in\lbrack\mathfrak{p}(b),b]$ imply $a\ll b$.

Part (ii) can be proved similarly.\bigskip
\end{proof}

Let $g$: $Q\rightarrow Q$ be a map on $Q.$ We say that $g$ is
\begin{align}
\text{a \textit{pre-radical map }if }x  &  \leq g(x)=g(g(x))\text{ and
}x<y<g(x)\text{ implies }g(y)=g(x);\label{3.41}\\
\text{a \textit{dual pre-radical map }if\textit{ }}g(g(x))  &  =g(x)\leq
x\text{ and }g(x)<y<x\text{ implies }g(y)=g(x);\label{2.0}\\
\text{a \textit{radical map }if }x  &  \leq g\left(  x\right)  =g\left(
g\left(  x\right)  \right)  \text{ and }x\leq y\text{ implies }g\left(
x\right)  \leq g\left(  y\right)  ;\label{f3.1}\\
\text{a \textit{dual radical map }if }g\left(  g\left(  x\right)  \right)   &
=g\left(  x\right)  \leq x\text{ and }x\leq y\text{ implies }g\left(
x\right)  \leq g\left(  y\right)  . \label{t3.2}%
\end{align}
Radical maps were considered in \cite[Definition I.3.26]{G}, where they are
called closure operators.

\begin{remark}
\label{R3.1}Radical maps are pre-radical, while pre-radical maps are not
always radical. \emph{Indeed, }$x<y<f(x)$\emph{ implies }$f(x)\leq f(y)\leq
f(f(x))=f(x).$\emph{ So }$f(y)=f(x)$\emph{ and }$f$\emph{ is pre-radical.}

\emph{On the other hand, let }$Q=\{\mathbf{0},a,b,c,d,\mathbf{1}\},$\emph{
}$a<b,$\emph{ }$a<c<d.$\emph{ Set }$f(\mathbf{0})=\mathbf{0},$\emph{
}$f(a)=f(b)=b,$\emph{ }$f(c)=f(d)=d,$\emph{ }$f(\mathbf{1})=\mathbf{1}.$\emph{
Then }$f$\emph{ is a }pre-radical map\emph{, but }not a radical map\emph{, as
}$a<c$\emph{ does not imply} $f(a)=b\leq f(c)=d.$ \emph{Similar} \emph{results
hold for }dual radical and dual pre-radical maps\emph{. \ \ }$\blacksquare$
\end{remark}

For a map $g$: $Q\rightarrow Q,$ define the following relation $\ll^{g}$ in
$Q$:%
\begin{equation}
x\ll^{g}y\text{ if }x\leq y\text{ and }g(x)=g(y)\text{ for }x,y\in Q,
\label{2,1}%
\end{equation}

\begin{theorem}
\label{T3.6}\emph{(i) }For a dual $\mathbf{T}$-order $\ll,$ the map
\[
\mathfrak{p}_{_{\ll}}\emph{:\ }b\in Q\mapsto\mathfrak{p}_{_{\ll}}(b)\text{
}\emph{(}\text{see }\emph{(\ref{2.13}))\ }\text{is a dual pre-radical map}%
\]
and$\ \ll$ $\subseteq$ $\ll^{\mathfrak{p}_{_{\ll}}}$ \emph{(}see
\emph{(\ref{0})). }If $\ll$ is a contiguous dual $\mathbf{T}$-order then $\ll$
$=$ $\ll^{\mathfrak{p}_{_{\ll}}}.\smallskip$

\emph{(ii) }The map $g\mapsto$ $\ll^{g}$ is a bijection from the set of all
dual pre-radical maps onto the set of all contiguous dual $\mathbf{T}$-orders.
The map $\ll$ $\mapsto$ $\mathfrak{p}_{_{\ll}}$ is its inverse\emph{,}
i.e.\emph{, }$\mathfrak{p}_{_{\ll^{g}}}=g.$
\end{theorem}

\begin{proof}
(i) From (\ref{2,2}), (\ref{2,5}) and (\ref{2.0}) it follows that
$\mathfrak{p}$ is a dual pre-radical map.

Let $a\ll b.$ By (\ref{2,2}), $\mathfrak{p}(a)=\mathfrak{p}(b).$ Thus
$a\ll^{\mathfrak{p}}b$ by (\ref{2,1}). So $\ll$ $\subseteq$ $\ll
^{\mathfrak{p}}$.

Let $a\ll^{\mathfrak{p}}b.$ Then, by (\ref{2,1}), $\mathfrak{p}%
(b)=\mathfrak{p}(a)\leq a\leq b.$ If $\ll$ is also up-contiguous then, by
Lemma \ref{L2.1}(i), $a\ll b,$ so that $\ll^{\mathfrak{p}}$ is stronger than
$\ll.$ Thus $\ll$ $=$ $\ll^{\mathfrak{p}}.$

(ii) If $x\ll^{g}y\ll^{g}z$ then $g(x)=g(y)=g(z)$ by (\ref{2,1}). So $x\ll
^{g}z.$ Thus $\ll^{g}$ is an order.

Let $x\ll^{g}y$ and $z\in\left[  x,y\right]  .$ Then $g(y)=g(x)\leq x\leq
z\leq y$ by (\ref{2.0}) and $g(z)=g(y)=g(x).$ Thus $x\ll^{g}z\ll^{g}y.$ So
$\left[  x,y\right]  \subseteq\lbrack x,\ll^{g}]$ and $\left[  x,y\right]
\subseteq\lbrack\ll^{g},y],$ i.e., $\ll^{g}$ is contiguous.

Fix $y\in Q.$ By (\ref{2,1}) and (\ref{2.0}), $\left[  \ll^{g},y\right]
=\{x\in Q$: $g(y)=g(x)\leq x\leq y\}\subseteq\lbrack g(y),y].$ On the other
hand, for each $g(y)\leq x\leq y,$ it follows from (\ref{2.0}) that $x\ll
^{g}y.$ Thus $[g(y),y]\subseteq\left[  \ll^{g},y\right]  .$ So $\left[
\ll^{g},y\right]  =[g(y),y].$ Hence
\[
\mathfrak{p}(y)=\mathfrak{p}_{\ll^{g}}(y)=\wedge\lbrack\ll^{g},y]=g(y)\in
\left[  \ll^{g},y\right]  .
\]
Thus $\left[  \ll^{g},y\right]  $ is $\wedge$-complete. Hence $\ll^{g}$ is
down-expanded. So $\ll^{g}$ is a dual $\mathbf{T}$-order which is also
up-contiguous and $\mathfrak{p}_{_{\ll^{g}}}=g.$ Applying (i), we complete the
proof.$\bigskip$
\end{proof}

By duality we have the following result.

\begin{theorem}
\label{T3.6'}\emph{(i) }For a $\mathbf{T}$-order $\ll,$ the map $\mathfrak{r}%
\emph{:\ }a\mapsto\mathfrak{r}_{_{\ll}}(a)$ $\emph{(}$see $\emph{(\ref{2.14}%
))\ }$is a pre-radical map\emph{,} and$\ \ll$ $\subseteq$ $\ll^{\mathfrak{r}%
_{\ll}}$ \emph{(}see\emph{(\ref{0}))}$\emph{.}$ If $\ll$ is a contiguous
$\mathbf{T}$-order then $\ll$ $=$ $\ll^{\mathfrak{r}_{\ll}}.\smallskip$

\emph{(ii) }The map $g\mapsto$ $\ll^{g}$ is a bijection from the set of all
pre-radical maps onto the set of all contiguous $\mathbf{T}$-orders. The map
$\ll$ $\mapsto$ $\mathfrak{r}_{_{\ll}}$ is its inverse\emph{,} i.e.\emph{,
}$\mathfrak{r}_{_{\ll^{g}}}=g.$
\end{theorem}

Each $\mathbf{TT}$-order $\ll$ defines the maps $\mathfrak{p},$ $\mathfrak{r}$
on $Q$ by (\ref{2.13}) and (\ref{2.14}). We will show that $Q$ decomposes into
a union of disjoint intervals and the restriction of $\ll$ to each of them
coincides with $\leq.$ Set%
\begin{equation}
Q_{\mathfrak{r}}=\{x\in Q:x=\mathfrak{r}(x)\}\text{ and }Q_{\mathfrak{p}%
}=\{y\in Q:y=\mathfrak{p}(y)\}. \label{2.21}%
\end{equation}

\begin{proposition}
\label{C2.13}Let $\ll$ be a $\mathbf{TT}$-order from \emph{Ref(}$Q).$
Then\smallskip

\emph{(i) \ \ }$\mathfrak{r}(a)=\mathfrak{r}(\mathfrak{p}(a))$ and
$\mathfrak{p}(a)=\mathfrak{p}(\mathfrak{r}(a))$ for all $a\in Q.\smallskip$

\emph{(ii) \ }The map $\mathfrak{p}$ maps $Q$ onto $Q_{\mathfrak{p}}$\emph{
}and it maps isomorphically $Q_{\mathfrak{r}}$ onto $Q_{\mathfrak{p}}.$

$\qquad$The map $\mathfrak{r}$ maps $Q$ onto $Q_{\mathfrak{r}}$ and it maps
isomorphically $Q_{\mathfrak{p}}$ on $Q_{\mathfrak{r}}$ and $\mathfrak{r}%
|_{Q_{\mathfrak{r}}}=(\mathfrak{p}|_{Q_{\mathfrak{p}}})^{-1}.\smallskip$

\emph{(iii) }Let $\mathfrak{p}(c)\leq a\leq b\leq\mathfrak{r}\left(  c\right)
$ for some $c\in Q.$ Then $a\ll b,$ $\mathfrak{p}(a)=\mathfrak{p}%
(b)=\mathfrak{p}(c)$

\qquad and $\mathfrak{r}\left(  a\right)  =\mathfrak{r}\left(  b\right)
=\mathfrak{r}\left(  c\right)  $. Thus the relations $\leq$ and $\ll$ coincide
on $\left[  \mathfrak{p}(c),\mathfrak{r}\left(  c\right)  \right]  .$
\end{proposition}

\begin{proof}
(i) By (\ref{2.13}) and (\ref{2.14}), $\mathfrak{p}(a)$ $\ll$ $a$ $\ll$
$\mathfrak{r}(a)$ for $a\in Q.$ So (i) follows from (\ref{2,2}) and (\ref{2,3}).

(ii) By (\ref{3.41}) and (\ref{2.0}), $\mathfrak{p}^{2}=\mathfrak{p}$ and
$\mathfrak{r}^{2}=\mathfrak{r}.$ Hence $y=\mathfrak{p}(x)\in Q_{\mathfrak{p}}$
for $x\in Q,$ as $\mathfrak{p}(y)=\mathfrak{p}^{2}(x)=\mathfrak{p}(x)=y.$
Similarly, $\mathfrak{r}(x)\in Q_{\mathfrak{r}}$ for $x\in Q.$ For $y\in
Q_{\mathfrak{p}},$ we have $\mathfrak{r}(y)\in Q_{\mathfrak{r}}$ and, by (i),
$\mathfrak{p}(\mathfrak{r}(y))=\mathfrak{p}(y)=y.$ Similarly, $\mathfrak{p}%
(x)\in Q_{\mathfrak{p}}$ and $\mathfrak{r}(\mathfrak{p}(x))=x$ for $x\in
Q_{\mathfrak{r}}.$ From this (ii) follows immediately.

(iii) By (\ref{2.13}) and (\ref{2.14}), $\mathfrak{p}(c)\ll\mathfrak{r}\left(
c\right)  .$ As $\ll$ is contiguous, $\mathfrak{p}(c)\ll a\ll b\ll
\mathfrak{r}\left(  c\right)  $. So, by (\ref{2,2}) and (\ref{2,3}),
$\mathfrak{p}(a)=\mathfrak{p}(b)=\mathfrak{p}(c)$ and $\mathfrak{r}\left(
a\right)  =\mathfrak{r}\left(  b\right)  =\mathfrak{r}\left(  c\right)  $.
Thus $\leq|_{\left[  \mathfrak{p}(c),\mathfrak{r}\left(  c\right)  \right]  }$
is stronger than $\ll|_{\left[  \mathfrak{p}(c),\mathfrak{r}\left(  c\right)
\right]  }.$ As $\ll$ $\subseteq$ $\leq$ in $Q,$ we have $\leq|_{\left[
\mathfrak{p}(c),\mathfrak{r}\left(  c\right)  \right]  }=$ $\ll|_{\left[
\mathfrak{p}(c),\mathfrak{r}\left(  c\right)  \right]  }.\bigskip$
\end{proof}

Recall that subsets $G_{1}$ and $G_{2}$ of $Q$ are disjoint if $G_{1}\cap
G_{2}=\varnothing.$

\begin{theorem}
\label{T2.12}Let $\ll$ belong to \emph{Ref(}$Q).$ The following conditions are
equivalent.\smallskip

\emph{(i) \ \ }$\ll$ is a $\mathbf{TT}$-order$.\smallskip$

\emph{(ii) \ }$Q=\cup_{\lambda\in Q_{\mathfrak{r}}}[\mathfrak{p}%
(\lambda),\lambda]$ and all intervals $[\mathfrak{p}(\lambda),\lambda],$
$\lambda\in Q_{\mathfrak{r}},$ are mutually disjoint\emph{.}

\qquad Moreover\emph{, }$x\ll y$ if and only if $x,y\in\lbrack\mathfrak{p}%
(\lambda),\lambda]$ and $x\leq y$ for some $\lambda\in Q_{\mathfrak{r}}%
.$\smallskip

\emph{(iii) }$Q=\cup_{\lambda\in Q_{\mathfrak{p}}}[\lambda,\mathfrak{r}%
(\lambda)]$ and all intervals $[\lambda,\mathfrak{r}(\lambda)],$ $\lambda\in
Q_{\mathfrak{p}},$ are mutually disjoint\emph{.}

\qquad Moreover\emph{, }$x\ll y$ if and only if $x,y\in\lbrack\lambda
,\mathfrak{r}(\lambda)]$ and $x\leq y$ for some $\lambda\in Q_{\mathfrak{p}}%
$\emph{.}
\end{theorem}

\begin{proof}
(ii) $\Rightarrow$ (i). Let $Q$ satisfy (ii). If $x\ll y$ then $\left[
x,y\right]  \subseteq\lbrack\mathfrak{p}(\lambda),\lambda]$ for some
$\lambda\in Q_{\mathfrak{r}}.$ As $\ll$ on $[\mathfrak{p}(\lambda),\lambda]$
coincides with $\leq,$ we have $\left[  x,y\right]  \subseteq\lbrack\ll,y]$
and $\left[  x,y\right]  \subseteq\lbrack x,\ll].$ Thus $\ll$ is contiguous.

We also have $\left[  x,\ll\right]  =\left[  x,\lambda\right]  $ and $\left[
\ll,y\right]  =\left[  \mathfrak{p}(\lambda),y\right]  .$ Hence $\left[
x,\ll\right]  $ is $\vee$-complete and $\left[  \ll,y\right]  $ is $\wedge
$-complete. Thus $\ll$ is expanded. Finally, if $x\ll y\ll z$ then
$x,y,z\in\left[  \mathfrak{p}(\lambda),\lambda\right]  $ for some $\lambda$
and $x\leq y\leq z.$ Hence $x\leq z,$ so that $x\ll z.$ Thus $\ll$ is
transitive, so that $\ll$ is a $\mathbf{TT}$-order.

(i) $\Rightarrow$ (ii). For a $\mathbf{TT}$ -order $\ll,$ let $z\in
\lbrack\mathfrak{p}(\lambda),\lambda]\cap\lbrack\mathfrak{p}(\nu),\nu],$ where
$\lambda\neq\nu$ in $Q_{\mathfrak{r}}.$ As $\mathfrak{p}(\lambda)\leq
z\leq\lambda,$ we have $\mathfrak{p}(z)=\mathfrak{p}(\lambda)$ by (\ref{2.0}).
Similarly, $\mathfrak{p}(z)=\mathfrak{p}(\nu),$ so $\mathfrak{p}%
(z)=\mathfrak{p}(\lambda)=\mathfrak{p}(\nu).$ By Proposition \ref{C2.13},
$\mathfrak{p}$ maps isomorphically $Q_{\mathfrak{r}}$ on $Q_{\mathfrak{p}}.$
Thus $\lambda=\nu,$ a contradiction. So all $[\mathfrak{p}(\lambda),\lambda]$
are disjoint$.$

Let $x\ll y.$ Set $\lambda=\mathfrak{r}(y).$ By (\ref{2,3}), $y\ll
\mathfrak{r}(x)=\mathfrak{r}(y)=\lambda$. By (\ref{3.41}) and (\ref{2.21}),
$\lambda\in Q_{\mathfrak{r}}$. So $\mathfrak{p}(\lambda)=\mathfrak{p}%
(\mathfrak{r}(x))=\mathfrak{p}(x)\ll x\ll y\ll\lambda$ by Proposition
\ref{C2.13} and (\ref{2,5}). Thus $x,y\in\lbrack\mathfrak{p}(\lambda
),\lambda].$

Conversely, let $x\leq y$ in $[\mathfrak{p}(\lambda),\lambda].$ By Proposition
\ref{C2.13}, $\ll$ and $\leq$ coincide in $[\mathfrak{p}(\lambda),\lambda].$
So $x\ll y.$

Finally, let $x\in Q.$ Set $\lambda=\mathfrak{r}(x).$ By (\ref{2,3}) and
Proposition \ref{C2.13}, $x\ll\lambda\in Q_{\mathfrak{r}}.$ Hence, by
Proposition \ref{C2.13} and (\ref{2,5}), $\mathfrak{p}(\lambda)=\mathfrak{p}%
(x)\ll x\ll\lambda.$ So $x\in\lbrack\mathfrak{p}(\lambda),\lambda].$ Thus
$Q=\cup_{\lambda\in Q_{\mathfrak{r}}}[\mathfrak{p}(\lambda),\lambda].$

Similarly, one can prove (i) $\Leftrightarrow$ (iii).\bigskip
\end{proof}

It should be noted that if $Q=\cup_{\lambda\in\Lambda}[a_{\lambda},b_{\lambda
}]$ and all intervals $[a_{\lambda},b_{\lambda}],$ $\lambda\in\Lambda,$ are
mutually disjoint\emph{,} then the relation in $Q$ defined by $x\ll y$ if and
only if $x,y\in\lbrack a_{\lambda},b_{\lambda}]$ and $x\leq y$ for some
$\lambda\in\Lambda,$ is a $\mathbf{TT}$-order.

\section{Radicals, $\mathbf{H}$-relations and $\mathbf{R}$-orders in complete
lattices}

In this section we consider $\mathbf{H}$-relations, and $\mathbf{R}$-orders.
We show that many results obtained for $\mathbf{T}$-orders in the previous
section can be strengthened for $\mathbf{R}$-orders.

\begin{lemma}
\label{le1}\emph{(i) }For a relation $\ll$ in $Q,$ the following conditions
are equivalent$:\smallskip$

$\qquad1)$\emph{\ }$\ll$ is up-contiguous and $a\wedge b\ll b$ implies $a\ll
a\vee b;\smallskip$

$\qquad2)$\emph{\ }$a\ll b$ and $a\leq c$ imply $c\ll b\vee c$ for $a,b,c\in
Q;\smallskip$

$\qquad3)$\emph{ }$a\ll b$ implies $a\vee x\ll b\vee x$ for every $x\in Q$.

\emph{(ii) }For a relation $\ll$ in $Q,$ the following conditions are
equivalent$:\smallskip$

$\qquad1)$ $\ll$ is down-contiguous and $a\ll a\vee b$ implies $a\wedge b\ll
b;\smallskip$

\qquad$2)$ $a\ll b$ and $c\leq b$ in $Q$ imply $a\wedge c\ll c;\smallskip$

\qquad$3)$ $a\ll b$ implies $a\wedge x\ll b\wedge x$ for every $x\in Q.$
\end{lemma}

\begin{proof}
(i) $1)$ $\Rightarrow$ $2).$ Let $a\ll b$ and $a\leq c$. Then $b\wedge
c\in\left[  a,b\right]  $. As $\ll$ is up-contiguous, $b\wedge c\ll b$, so
that $c\ll b\vee c$. \ $2)$ $\Longleftrightarrow$ $3)$ is proved in \cite[Page
776]{Am}.

$3)$ $\Rightarrow$ $1).$ For $c\in\left[  a,b\right]  $ and $a\ll b,$ one has
$c=a\vee c\ll b\vee c=b$. Thus $\ll$ is up-contiguous. If $a\wedge b\ll b$
then $a=a\vee\left(  a\wedge b\right)  \ll a\vee b$. Part (ii) is proved similarly.
\end{proof}

\begin{definition}
\label{D2.2}\emph{(\cite{Am}) (i) }$\ll$ is an $\mathbf{H}$-\textbf{relation}
if it satisfies conditions of Lemma $\ref{le1}(i).\smallskip$

\emph{(ii)} \ $\ll$ is a \textbf{dual }$\mathbf{H}$\textbf{-relation} if it
satisfies equivalent conditions of Lemma $\ref{le1}(ii).\smallskip$

\emph{(iii) }$\ll$ is an $\mathbf{HH}$-relation if it is an $\mathbf{H}$- and
a dual $\mathbf{H}$-relation.

\emph{(iv)} $\ll$ is an $\mathbf{R}$-\textbf{order} if it is an up-expanded
$\mathbf{H}$-order$.$\smallskip

\emph{(v) } $\ll$ is a dual $\mathbf{R}$-\textbf{order} if it is a
down-expanded dual $\mathbf{H}$-order.\smallskip

\emph{(vi)\ }$\ll$ is an$\emph{\ }\mathbf{RR}$-\textbf{order} if it is an
$\mathbf{R}$-order and a dual $\mathbf{R}$-order$\emph{.}$
\end{definition}

By the Duality Principle \cite[Theorem $1.3^{\prime}$]{Sk}, the results for
$\mathbf{H}$-relations and $\mathbf{R}$-orders follow from the corresponding
results for dual $\mathbf{H}$-relations and dual $\mathbf{R}$-orders and vice versa.

Each $\mathbf{H}$-order $\ll$ is \textit{finitely up-expanded}: $b,c\in\lbrack
a,\ll]$\emph{ }implies $b\vee c\in\lbrack a,\ll]$ for all $a\in Q.$ Indeed, by
Lemma \ref{le1}, $c=a\vee c\ll b\vee c$. By transitivity, $a\ll b\vee c$, as
$a\ll c.$ Similarly, each dual $\mathbf{H}$-order $\ll$ is \textit{finitely
down-expanded}, i.e., $a,c\in\lbrack\ll,b]$\emph{ }implies $a\wedge
c\in\lbrack\ll,b]$ for each $a\in Q.$

The notions of (dual) $\mathbf{R}$-orders are stronger than the notions of
(dual) $\mathbf{H}$-orders, respectively, as they require up- and
down-expandedness and not only its "finite" version. Amitsur \cite{Am} defined
them in a different, but equivalent way (he called them $\mathbf{R}$- and dual
$\mathbf{R}$-relations).

Comparing Definitions \ref{D3.2} and \ref{D2.2}, we see that (dual)
$\mathbf{R}$-orders are (dual) $\mathbf{T}$-orders$.$ Thus the results for
(dual) $\mathbf{T}$-orders hold also for (dual) $\mathbf{R}$-orders. In particular,

1) for a dual $\mathbf{R}$-order $\ll,$ $\mathfrak{p}(b)=\wedge\lbrack\ll,b]$
is the unique dual $\ll$-radical in$\ [\mathbf{0},b]$ and (\ref{2.13}) holds;

2) for an $\mathbf{R}$-order $\ll,$ $\mathfrak{r}(a)=\vee\lbrack a,\ll]$ is
the unique $\ll$-radical in$\ [a,\mathbf{1}]$ and (\ref{2.14}) holds.

However, as the following example shows, (dual) $\mathbf{T}$-orders are not
necessarily (dual) $\mathbf{R}$-orders.

\begin{example}
\label{E4.1}\emph{(i) Let} $Q=\{\mathbf{0},a,b,\mathbf{1}\},$ $a\wedge
b=\mathbf{0}$ \emph{and} $a\vee b=\mathbf{1}.$ \emph{Let }$\ll$\emph{ be a
reflexive relation in }$Q.\smallskip$

\emph{1) If only} $\mathbf{0}\ll a$ \emph{then }$\ll$\emph{ is a}
$\mathbf{TT}$-\emph{order but not an} $\mathbf{R}$-\emph{order as}
$b=\mathbf{0}\vee b\not \ll a\vee b=\mathbf{1}.\smallskip$

\emph{2) If} \emph{only }$a\ll\mathbf{1}$ \emph{then }$\ll$\emph{ is a}
$\mathbf{TT}$-\emph{order but not a dual} $\mathbf{R}$-\emph{order}.\smallskip

\emph{(ii) Let }$Q=[0,1]\subset\mathbb{R}$\emph{,} $\ll|_{[0,1)}=$
$\leq|_{[0,1)}$ \emph{and }$1\ll1.$\emph{ Then }$\ll$\emph{ is an }%
$\mathbf{H}$\emph{-order but not an }$\mathbf{R}$\emph{-order}.
\end{example}

We consider now some additional properties of the maps $a\mapsto
\mathfrak{p}(a),$ $a\mapsto\mathfrak{r}(a)$ (see Lemma \ref{L2.1}).

\begin{lemma}
\label{L2.2}\emph{(i) }Let $\ll$ be a dual $\mathbf{R}$-order$.$ Then
\emph{(\ref{2,2}) }and \emph{(\ref{2,5})} hold and\emph{, }for $a,b\in Q,$%
\[
\mathfrak{p}(a\wedge b)\ll\mathfrak{p}(a)\wedge\mathfrak{p}(b)\ll a\wedge
b\text{ and \ }\mathfrak{p}(a)\leq\mathfrak{p}(b)\text{ if }a\leq b.
\]

\emph{(ii) }Let $\ll$ be an $\mathbf{R}$-order$.$ Then \emph{(\ref{2,3})}
holds and\emph{, }for $a,b\in Q,$%
\[
a\vee b\ll\mathfrak{r}(a)\vee\mathfrak{r}(b)\ll\mathfrak{r}(a\vee b)\text{ and
}\mathfrak{r}(a)\leq\mathfrak{r}(b)\mathit{\ }\text{if }a\leq b.
\]

\end{lemma}

\begin{proof}
As a dual $\mathbf{R}$-order is a dual $\mathbf{T}$-order, the results of
Lemma \ref{L2.1}(i) hold for them.

As $\mathfrak{p}(b)\ll b$ and $\mathfrak{p}(a)\ll a$ by (\ref{2.13}), we have
$\mathfrak{p}(a)\wedge\mathfrak{p}\left(  b\right)  \ll\mathfrak{p}(a)\wedge
b\ll a\wedge b$ from Lemma \ref{le1}(ii). By transitivity, $\mathfrak{p}%
(a)\wedge\mathfrak{p}\left(  b\right)  \ll a\wedge b$. So, by (\ref{2,2}),
$\mathfrak{p}(a\wedge b)\ll\mathfrak{p}(a)\wedge\mathfrak{p}(b)\ll a\wedge b.$

If $a\leq b$ then $\mathfrak{p}(a)=\mathfrak{p}(a\wedge b)\leq\mathfrak{p}%
(a)\wedge\mathfrak{p}(b).$ So $\mathfrak{p}(a)\leq\mathfrak{p}(b).$ The proof
of (ii) is similar.\bigskip
\end{proof}

Recall that, for a map $g$: $Q\rightarrow Q,$ the relation $\ll^{g}$ in $Q$ is
defined in (\ref{2,1}).

\begin{theorem}
\label{T2.2}\emph{(i) }For a dual $\mathbf{R}$-order $\ll,$ the map
$\mathfrak{p}_{\ll}\emph{:\ }b\mapsto\mathfrak{p}_{\ll}(b)$ in
$\emph{(\ref{2.13})\ }$is a dual radical map and$\ \ll$ $\subseteq$
$\ll^{\mathfrak{p}_{\ll}}$ \emph{(}see \emph{(\ref{0})). }If $\ll$ is a
contiguous dual $\mathbf{R}$-order then $\ll$ $=$ $\ll^{\mathfrak{p}_{\ll}%
}.\smallskip$

\emph{(ii) }The map $g\mapsto$ $\ll^{g}$ is a bijection from the set of all
dual radical maps onto the set of all contiguous dual $\mathbf{R}$-orders. The
map $\ll$ $\mapsto$ $\mathfrak{p}_{_{\ll}}$ is its inverse\emph{,} i.e.\emph{,
}$\mathfrak{p}_{_{\ll^{g}}}=g.$
\end{theorem}

\begin{proof}
(i) By (\ref{2,2}) and (\ref{2.13}), $\mathfrak{p}(\mathfrak{p}%
(b))=\mathfrak{p}(b)\leq b$ for $b\in Q.$ By Lemma \ref{L2.2}, $\mathfrak{p}%
(a)\leq\mathfrak{p}(b)$ if $a\leq b.$ Thus $\mathfrak{p}$ satisfies
(\ref{t3.2}), so that it is a dual radical map.

If $\ll$ is a dual $\mathbf{R}$-order, it is a dual $\mathbf{T}$-order. So it
follows from Theorem \ref{T3.6} that $\ll$ $\subseteq$ $\ll^{\mathfrak{p}}$
and that $\ll$ $=$ $\ll^{\mathfrak{p}}$ if $\ll$ is also up-contiguous$.$

(ii) If $g$ is a dual radical map then it is pre-dual by Remark \ref{R3.1}.
Hence, by Theorem \ref{T3.6}, $\ll^{g}$ is a contiguous dual $\mathbf{T}%
$-order and $\mathfrak{p}_{_{\ll^{g}}}=g.$

If $a\ll^{g}b$ then $g\left(  a\right)  =g\left(  b\right)  $. Let $c\leq b$.
By (\ref{t3.2}), $g(c)\leq g(b)$ and $g(x)\leq x$ for all $x\in Q.$ Thus%
\[
g(c)=g(b)\wedge g(c)=g(a)\wedge g(c)\leq a\wedge c\leq c.
\]
So $g(c)\leq a\wedge c\leq c.$ As $g$ is also pre-dual, $g(c)=g(a\wedge c)$ by
(\ref{2.0}). Hence $a\wedge c\ll^{g}c.$ Thus we have from Lemma \ref{le1}(ii)
that $\ll^{g}$ is a dual $\mathbf{H}$-relation. So $\ll^{g}$ is a contiguous
dual $\mathbf{R}$-order. Applying (i), we complete the proof.\bigskip
\end{proof}

By duality we get the following result.

\begin{theorem}
\label{T2.2'}\emph{(i) }For a $\mathbf{R}$-order $\ll,$ the map $\mathfrak{r}%
\emph{:\ }a\mapsto\mathfrak{r}(a)$ in $\emph{(\ref{2.14})\ }$is a radical map
and$\ \ll$ $\subseteq$ $\ll^{\mathfrak{r}}$ \emph{(}see\emph{(\ref{0}%
))}$\emph{.}$ If $\ll$ is a contiguous $\mathbf{R}$-order then $\ll$ $=$
$\ll^{\mathfrak{r}}.\smallskip$

\emph{(ii) }The map $g\mapsto$ $\ll^{g}$ is a bijection from the set of all
radical maps onto the set of all contiguous $\mathbf{R}$-orders. The map $\ll$
$\mapsto$ $\mathfrak{r}_{_{\ll}}$ is its inverse\emph{,} i.e.\emph{,
}$\mathfrak{r}_{_{\ll^{g}}}=g.$
\end{theorem}

We will now study properties of the maps $\mathfrak{p},$ $\mathfrak{r}$ for
relations satisfying combined conditions.

\begin{lemma}
\label{L2.3}\emph{(i) }If $\ll$ is an $\mathbf{H}$-relation and a dual
$\mathbf{T}$-order then $\mathfrak{p}(a\vee b)\ll\mathfrak{p}(a)\vee
\mathfrak{p}(b)$ for $a,b\in Q.$\smallskip

\emph{(ii) \ }If $\ll$ is an $\mathbf{H}$-relation and a dual $\mathbf{R}%
$-order\emph{,} then $\mathfrak{p}(a\vee b)=\mathfrak{p}(a)\vee\mathfrak{p}%
(b)$ for $a,b\in Q.\smallskip$

\emph{(iii) }If $\ll$ is a dual $\mathbf{H}$-relation and a $\mathbf{T}$-order
then $\mathfrak{r}(a)\wedge\mathfrak{r}(b)\ll\mathfrak{r}(a\wedge b)$ for
$a,b\in Q.\smallskip$

\emph{(iv)} If $\ll$ is a dual $\mathbf{H}$-relation and an $\mathbf{R}$-order
then $\mathfrak{r}(a)\wedge\mathfrak{r}(b)=\mathfrak{r}(a\wedge b)$ for
$a,b\in Q.$
\end{lemma}

\begin{proof}
(i) By (\ref{2.13}), $\mathfrak{p}(a)\ll a.$ As $\ll$ is an $\mathbf{H}%
$-relation, $\mathfrak{p}(a)\vee\mathfrak{p}(b)\ll a\vee\mathfrak{p}(b)$ by
Lemma \ref{le1}. Similarly, $\mathfrak{p}(b)\ll b$ implies $a\vee
\mathfrak{p}(b)\ll a\vee b$. By transitivity of $\ll$, $\mathfrak{p}%
(a)\vee\mathfrak{p}(b)\ll a\vee b$. Since $\ll$ is also a dual $\mathbf{T}%
$-order, $\mathfrak{p}(a\vee b)\ll\mathfrak{p}(a)\vee\mathfrak{p}(b)$ by
(\ref{2,2}).

(ii) If $\ll$ is a dual $\mathbf{R}$-order, $\mathfrak{p}(a)\vee
\mathfrak{p}(b)\leq\mathfrak{p}(a\vee b)$ by Lemma \ref{L2.2}. So the proof
follows from (i).

Parts (iii) and (iv) follow from duality.\bigskip
\end{proof}

Recall that $\mathbf{TT}$-orders are $\mathbf{T}$- and dual $\mathbf{T}%
$-orders, and $\mathbf{RR}$-orders are $\mathbf{R}$- and dual $\mathbf{R}%
$-orders$\emph{.}$ Each $\mathbf{TT}$-order defines both maps $\mathfrak{r}%
$\emph{: }$a\mapsto\mathfrak{r}(a)$ and $\mathfrak{p}$\emph{: }$a\mapsto
\mathfrak{p}(a)$ on $Q.$ In terms of them we will find necessary and
sufficient conditions for a $\mathbf{TT}$-order to be an $\mathbf{R}$-, or a
dual $\mathbf{R}$-, or an $\mathbf{RR}$-order.

\begin{corollary}
\label{C2.14}\emph{(i) }Let $\ll$ be a $\mathbf{TT}$-order in $Q.$ The
following conditions are equivalent\emph{.\smallskip}

$\qquad1)\ \ll$ is a dual $\mathbf{R}$-order$;$

$\qquad2)$ $\mathfrak{p}(a)\leq\mathfrak{p}(b)$ if $a\leq b$ in $Q;$

$\qquad3)\ \mathfrak{p}(a\wedge b)\leq\mathfrak{p}(a)\wedge\mathfrak{p}(b)$
for $a,b\in Q;$

\qquad$4)$ $\mathfrak{p}(a\wedge b)=\mathfrak{p}\left(  \mathfrak{p}%
(a)\wedge\mathfrak{p}(b)\right)  $ for $a,b\in Q$.\smallskip

\emph{(ii) }Let $\ll$ be a $\mathbf{TT}$-order in $Q.$ The following
conditions are equivalent\emph{:\smallskip}

$\qquad1)$ $\ll$ is an $\mathbf{R}$-order$;$

$\qquad2)$ $\mathfrak{r}(a)\leq\mathfrak{r}(b)$ if $a\leq b$ in $Q;$

$\qquad3)$ $\mathfrak{r}(a)\vee\mathfrak{r}(b)\leq\mathfrak{r}(a\vee b)$ for
$a,b\in Q;$

$\qquad4)$ $\mathfrak{r}\left(  \mathfrak{r}(a)\vee\mathfrak{r}(b)\right)
=\mathfrak{r}(a\vee b)$ for $a,b\in Q$.\smallskip

\emph{(iii) }A $\mathbf{TT}$-order $\ll$\textit{ is an }$\mathbf{RR}%
$\textit{-order if and only if }$\mathfrak{p}(a)\leq\mathfrak{p}(b)$\textit{
and }$\mathfrak{r}(a)\leq\mathfrak{r}(b)$\textit{ for all }$a\leq b.$
\end{corollary}

\begin{proof}
(i) $1)\Rightarrow2)$ follows from Lemma \ref{L2.2}(i).

$2)\Rightarrow3)$ As $a\wedge b\leq a,$ we have $\mathfrak{p}(a\wedge
b)\leq\mathfrak{p}(a)$ from 2). Similarly, $\mathfrak{p}(a\wedge
b)\leq\mathfrak{p}(b)$ and $3)$ follows$.$

$3)\Rightarrow4)$ We have $\mathfrak{p}(a\wedge b)\leq\mathfrak{p}%
(a)\wedge\mathfrak{p}(b)\leq a\wedge b$. By (\ref{2,5}), $\mathfrak{p}(a\wedge
b)=\mathfrak{p}\left(  \mathfrak{p}(a)\wedge\mathfrak{p}(b)\right)  $.

$4)\Rightarrow3).$ As $\mathfrak{p}(x)\leq x$ for $x\in Q,$ we have
$\mathfrak{p}\left(  \mathfrak{p}(a)\wedge\mathfrak{p}(b)\right)
\leq\mathfrak{p}(a)\wedge\mathfrak{p}(b).$

$3)\Rightarrow1).$ Let $x\ll y$ and $z\in Q.$ By Theorem \ref{T2.12},
$\mathfrak{p}(\lambda)\leq x\leq y\leq\lambda$ and $\mathfrak{p}(\mu)\leq
z\leq\mu$ for some $\lambda,\mu\in Q_{\mathfrak{r}}$. Hence, by 3),%
\[
\mathfrak{p}(\lambda\wedge\mu)\leq\mathfrak{p}(\lambda)\wedge\mathfrak{p}%
(\mu)\leq x\wedge z\leq y\wedge z\leq\lambda\wedge\mu\leq\mathfrak{r}\left(
\lambda\wedge\mu\right)  .
\]
By Proposition \ref{C2.13}(iii), $x\wedge z\ll y\wedge z.$ Thus $\ll$ is a
dual $\mathbf{H}$-relation. So it is a dual $\mathbf{R}$-order.

Part (ii) can be proved similarly. Part (iii) follows from (i) and
(ii).\bigskip
\end{proof}

Maps $f,g$ on $Q$ are \textit{conjugate }if $g\left(  x\right)  =g\left(
f\left(  x\right)  \right)  $ and $f\left(  x\right)  =f\left(  g\left(
x\right)  \right)  $ for all $x\in Q.$

\begin{corollary}
\emph{(i) }For an $\mathbf{RR}$-order $\ll,$ the maps $\mathfrak{r}_{_{\ll}}$
and $\mathfrak{p}_{_{\ll}}$ are conjugate.\smallskip

\emph{(ii) }Let maps $f,g$ on $Q$ be conjugate\emph{ }and let
\begin{equation}
g(x)\leq x\leq f(x)\text{ for }x\in Q,\text{ and }f(x)\leq f(y)\text{ and
}g(x)\leq g(y)\text{ for }x\leq y\text{ in }Q. \label{3;1}%
\end{equation}
Then there is an $\mathbf{RR}$-order $\ll$ such that $f=\mathfrak{r}_{_{\ll}}$
and $g=\mathfrak{p}_{_{\ll}}$. Moreover\emph{, }$\ll$ $=$ $\ll^{f}$ $=$
$\ll^{g}.$
\end{corollary}

\begin{proof}
(i) Let $\ll$ be an $\mathbf{RR}$-order$.$ By Theorems \ref{T2.2} and
\ref{T2.2'}, $\mathfrak{r}_{_{\ll}}$ is a radical map and $\mathfrak{p}%
_{_{\ll}}$ is a dual radical map$.$ By (\ref{2.13}) and (\ref{2.14}),
$\mathfrak{p}_{_{\ll}}(x)\ll x\ll\mathfrak{r_{_{\ll}}(}x)$ for all $x\in Q,$
and%
\[
\mathfrak{p}_{_{\ll}}(x)\overset{(\ref{2,2})}{=}\mathfrak{p}_{_{\ll}%
}(\mathfrak{r}_{_{\ll}}(x))\text{ and }\mathfrak{r}_{_{\ll}}%
(x)\overset{(\ref{2,3})}{=}\mathfrak{r}_{_{\ll}}(\mathfrak{p}_{_{\ll}}(x)).
\]

(ii) As $f$ and $g$ are conjugate, $f(f(x))=f(g(f(x)))=f(g(x))=f(x).$
Similarly, $g(g(x))=g(x).$ From this and from (\ref{3;1}) we get that $f$ is a
radical and $g$ is a dual radical map ((\ref{f3.1}), (\ref{t3.2})).

Let $x\ll^{f}y.$ By (\ref{2,1}), $f\left(  x\right)  =f\left(  y\right)  .$ As
$f$ and $g$ are conjugate, $g\left(  x\right)  =g\left(  f\left(  x\right)
\right)  =g\left(  f\left(  y\right)  \right)  =g\left(  y\right)  .$ So
$x\ll^{g}y.$ Similarly, $x\ll^{g}y\Rightarrow x\ll^{f}y.$ Thus $\ll^{f}=$
$\ll^{g}.$ By Theorems \ref{T2.2} and \ref{T2.2'}, $\ll$ $=$ $\ll^{f}=$
$\ll^{g}$ is an $\mathbf{RR}$-order, $f=\mathfrak{r}_{_{\ll}}$ and
$g=\mathfrak{p}_{_{\ll}}$.
\end{proof}

\section{Refinement of relations}

We proved in the previous section that $\mathbf{T}$-orders have unique
radicals. However, majority of relations in applications are neither
contiguous, nor expanded. In this section we consider two ways to construct
from a relation some new relations that are contiguous, or expanded, or both,
so that the new relations have radicals. We start with the complements of
$\ll$: $\overleftarrow{\ll}$ and $\overrightarrow{\ll}$ defined in (\ref{A1}).

\begin{proposition}
\label{Lat1}\emph{(i)\ }$\overleftarrow{\ll}$ is a down-contiguous
relation\emph{;} the set $[a,\overleftarrow{\ll}]$ is $\wedge$-complete for
each $a\in Q.\smallskip$

\emph{(ii) \ }$\overrightarrow{\ll}$ is an up-contiguous relation$;$ the set
$[\overrightarrow{\ll},b]$ is $\vee$-complete for each $b\in Q.\smallskip$

\emph{(iii) }If $\ll$ is stronger than $\prec,$ then $\overleftarrow{\prec}$
is stronger than $\overleftarrow{\ll}$ and $\overrightarrow{\prec}$ is
stronger than $\overrightarrow{\ll}$\emph{.}
\end{proposition}

\begin{proof}
(i) Let $a$ $\overleftarrow{\ll}$ $b.$ By (\ref{A1}), $[a,\ll]\cap\lbrack
a,b]=\left\{  a\right\}  .$ Then $[a,\ll]\cap\lbrack a,x]=\left\{  a\right\}
$ for each $x\in\lbrack a,b].$ Thus $a$ $\overleftarrow{\ll}$ $x,$ so that
$\overleftarrow{\ll}$ is down-contiguous (see Definition \ref{D1.1}).

Let $G\subseteq\lbrack a,\overleftarrow{\ll}]$. By (\ref{A1}), $[a,\ll
]\cap\lbrack a,b]=\left\{  a\right\}  $ for all $b\in G$. Thus $[a,\ll
]\cap\lbrack a,\wedge G]=\left\{  a\right\}  $. Hence $a\;\overleftarrow{\ll
}\!\wedge G$, so that $[a,\overleftarrow{\ll}]$ is $\wedge$-complete. Part
(ii) is proved similarly.

(iii) Let $a$ $\overleftarrow{\prec}$ $b.$ By (\ref{A1}), $[a,\prec
]\cap\lbrack a,b]=\left\{  a\right\}  .$ As $\ll$ $\subseteq$ $\prec$ (see
(\ref{0})), we have $[a,\ll]\subseteq\lbrack a,\prec].$ Hence $[a,\ll
]\cap\lbrack a,b]=\left\{  a\right\}  .$ By (\ref{A1}), $a$
$\overleftarrow{\ll}$ $b.$ The proof of the second statement is
similar.\bigskip
\end{proof}

We introduce now two constructions that play important part in this paper.

\begin{definition}
\label{D2}\thinspace Let $\ll$ be a relation from \emph{Ref(}$Q).$ We say that
$G\subseteq Q$ is\smallskip

$1)$ a \textbf{lower }$\ll$\textbf{-set} if\emph{,} for each $x\in
G\diagdown\{\wedge G\},$ there is $y\in G$ such that $x\neq y$ and $y\ll
x;\smallskip$

$2)$ an \textbf{upper }$\ll$\textbf{-set} if\emph{, }for each $x\in
G\diagdown\{\vee G\},$ there is $y\in G$ such that $x\neq y$ and $x\ll
y.$\smallskip

$3)$ We write $a\ll^{\text{\emph{lo}}}b\,$if$\ [a,b]\ $is a lower$\ \ll
$-set\emph{;} and\ $a\ll^{\text{\emph{up}}}b$ if $[a,b]\ $is an upper $\ll
$-set$.$
\end{definition}

\begin{proposition}
\label{p11}\emph{(i)\ }$\ll^{\emph{lo}}=$ $\overleftarrow{\left(
\overrightarrow{\ll}\right)  }$ is a down-contiguous order$.$ If $\ll$ is
down-contiguous\emph{,} $\ll$ $\subseteq$ $\ll^{\text{\emph{lo}}}.\smallskip$

\emph{(ii) \ }$\overrightarrow{\ll}$ $\subseteq$ $\overrightarrow{\ll
^{\text{\emph{lo}}}}$ and $(\ll^{\emph{lo}})^{\text{\emph{lo}}}=$
$\ll^{\emph{lo}}.$\smallskip

\emph{(iii) }$\ll^{\emph{up}}=$ $\overrightarrow{\left(  \overleftarrow{\ll
}\right)  }$ is an up-contiguous order. If $\ll$ is up-contiguous\emph{,}
$\ \ll$ $\subseteq$ $\ll^{\text{\emph{up}}}.\smallskip$

\emph{(iv) }$\overleftarrow{\ll}$ $\subseteq$ $\overleftarrow{\ll
^{\,\emph{up}}}$ and $(\ll^{\emph{up}})^{\emph{up}}=$ $\ll^{\emph{up}}.$
\end{proposition}

\begin{proof}
(i) If $a$ $\overleftarrow{\left(  \overrightarrow{\ll}\right)  }$ $b$ then
$a$ is not $\overrightarrow{\ll}$-related to all $x\in(a,b]$ by (\ref{A1}).
So, by (\ref{A1}),
\begin{equation}
a\text{ }\overleftarrow{\left(  \overrightarrow{\ll}\right)  }\text{ }b\text{
if and only if, for each }x\in(a,b],\text{ there is }y\in\lbrack a,x)\text{
such that }y\ll x. \label{2.1}%
\end{equation}
That is, if and only if $[a,b]$ is a lower $\ll$-set: $a\ll^{\text{lo}}b.$
Hence $\ll^{\text{lo}}=$ $\overleftarrow{\left(  \overrightarrow{\ll}\right)
}.$ So, by Proposition \ref{Lat1}(i), $\ll^{\text{lo}}$ is down-contiguous.
Clearly, it is an order.

If $\ll$ is down-contiguous and $a\ll b$ then $\left[  a,b\right]
\subseteq\left[  a,\ll\right]  $. So $a\ll x$ for all $x\in(a,b]$. Thus
$\left[  a,b\right]  $ is a lower $\ll$-set: $a\ll^{\text{lo}}b$. Thus $\ll$
$\subseteq$ $\ll^{\text{lo}}.$ The proof of (iii) is similar.

(ii) By Proposition \ref{Lat1}(ii), $\overrightarrow{\ll}$ is up-contiguous.
So, by (iii) and (i), $\overrightarrow{\ll}$ $\subseteq$ $\left(
\overrightarrow{\ll}\right)  ^{\text{up}}%
=\overrightarrow{\overleftarrow{\left(  \overrightarrow{\ll}\right)  }%
}=\overrightarrow{\ll^{\text{lo}}}.$

Hence, by (i) and by Proposition \ref{Lat1}(iii), $(\ll^{\text{lo}%
})^{\text{lo}}=\overleftarrow{\overrightarrow{\ll^{\text{lo}}}}\ \subseteq
\ \overleftarrow{\left(  \overrightarrow{\ll}\right)  }=$ $\ll^{\text{lo}}.$

Let $a$ be not $\overleftarrow{\overrightarrow{\ll^{\text{lo}}}}$--related to
$b$. By (\ref{2.1}), there is $x\in(a,b]$ such that $y$ is not $\ll
^{\text{lo}}$--related to $x$ for each $y\in\lbrack a,x)$. In particular, $a$
is not $\ll^{\text{lo}}$-related to $x.$ Then there is $z\in(a,x]$ such that
$u$ $\not \ll $ $z$ for each $u\in\lbrack a,z)$. Hence $a$ is not
$\ll^{\text{lo}}$--related to $b$. Thus $\ll^{\text{lo}}$ $\subseteq$
$\overleftarrow{\overrightarrow{\ll^{\text{lo}}}}=(\ll^{\text{lo}}%
)^{\text{lo}}$. So $\ll^{\text{lo}}\ =\ (\ll^{\text{lo}})^{\text{lo}}.$

The proof of part (iv) is similar.\bigskip
\end{proof}

To describe the second way of constructing new relations from $\ll$, we study
chains in lattices. Recall that a chain $C$ in $Q$ is a linearly ordered
subset: either $x\leq y$ or $y\leq x$ for $x,y\in C.$ A chain $C$ is
\textit{from} $a$ \textit{to} $b$ if $\wedge C=a$ and $\vee C=b$. A chain $C$
is \textit{maximal}\textbf{ }in $G\subseteq Q$ if $C\subseteq G$ and $G$ has
no larger chains.

\begin{definition}
\label{D3}Let\emph{ }$\ll$ be a relation in $Q$ and $C$ be a chain from $a$ to
$b.$ Then\smallskip

$1)\ C$ is a \textbf{lower \emph{(}upper\emph{)} }$\ll$\textbf{-chain}, if it
is a lower \emph{(}upper\emph{)} $\ll$-set.\smallskip

$2)$ $C$ is a \textbf{lower }$\ll$\textbf{-gap chain} if each $x\in
C\diagdown\{a\}$ has an immediate $\ll$-predecessor $p_{x}\in C$\emph{:}

$\qquad\qquad\qquad p_{x}\neq x,$ $p_{x}\ll x$ and $[p_{x},x]\cap
C=\{p_{x},x\}.\smallskip$

$3)$\emph{ }$C$ is an \textbf{upper }$\ll$\textbf{-gap chain} if each $x\in
C\diagdown\{b\}$ has an immediate $\ll$-successor $s_{x}\in C$\emph{:}

$\qquad\qquad\qquad x\neq s_{x},$ $x\ll s_{x}$ and $[x,s_{x}]\cap
C=\{x,s_{x}\}.$
\end{definition}

\begin{lemma}
\label{L2.1p}Let $G$ be a $\wedge$-complete set in $Q,$ $C$ be a chain in $G$
and $b:=\vee C\in G.$ Set $a=\wedge C.$\smallskip

\emph{(i) \ \ }There is a maximal chain in $G\cap\lbrack a,b]$ containing
$C.\smallskip$

\emph{(ii) \ }If $C$ is a maximal chain in $G\cap\lbrack a,b],$ it is $\wedge
$-complete. If $G$ is complete\emph{, }$C$ is complete.\smallskip

\emph{(iii) (cf. }Lemma \emph{I.34 \cite{G}) }If $C$ is a $\wedge$-complete
lower $\ll$-gap chain and $b\in C,$ then $C$ is complete$.$
\end{lemma}

\begin{proof}
(i) Apply Zorn's Lemma to the set of all chains in $G\cap\lbrack a,b]$ that
contain $C.$

(ii) As $C^{\wedge}$ is a chain in $G\cap\lbrack a,b]$ containing $C$ and $C$
is maximal, $C=C^{\wedge}$. If $G$ is complete then $C^{\vee}\subseteq
G\cap\lbrack a,b].$ As $C^{\vee}$ is a chain containing $C$ and $C$ is
maximal, $C=C^{\vee}$. So $C$ is complete.

(iii) For $\Gamma\subseteq C,$ $\Gamma\neq\{a\},$ let $K=\left\{  x\in
C\text{: }y\leq x\text{ for all }y\in\Gamma\right\}  .$ As $b\in K,$
$K\neq\varnothing.$ As $C$ is $\wedge$-complete, $d:=\wedge K\in C.$ Then
$a<d$ and $y\leq d$ for all $y\in\Gamma.$ So $d\in K.$ As $C$ is a lower $\ll
$-gap chain, there is an immediate $\ll$-predecessor $p$ of $d$ in $C$:
$C\cap\lbrack p,d]=\{p,d\}.$ If $y\leq p$ for all $y\in\Gamma,$ then $p\in K$
and $d=\wedge K\leq p,$ a contradiction. Hence there is $y\in\Gamma$ such that
$p<y\leq d.$ So $y=d$, i.e., $\vee\Gamma=d\in C$. Thus $C$ is $\vee
$-complete.\bigskip
\end{proof}

We shall now consider a particular type of extension of lower $\ll$-gap chains.

\begin{definition}
Let $D$ be a chain, $a=\wedge D$ and $b=\vee D\in D.$ A chain $C$
\textbf{down-extends}\textit{ }$D$ if%
\begin{equation}
C\subseteq\lbrack\mathbf{0},b]\text{ and }D=C\cap\lbrack a,b],\text{ i.e.,
}C=D\cup E,\text{ where }E\text{ is a chain in }[\mathbf{0},a]\text{.}
\label{2.7}%
\end{equation}
A chain $D$ is a \textbf{maximal down-extended} lower $\ll$-gap chain
\emph{(}resp. lower $\ll$-chain\emph{)}, if there does not exist another lower
$\ll$-gap chain \emph{(}resp. lower $\ll$-chain\emph{) }that down-extends $D$.
\end{definition}

Let $Q=[0,1]\subset\mathbb{R}.$ Then ($Q,\leq)$ is a maximal chain, a maximal
down-extended lower $\leq$-chain, but not a lower $\leq$-gap chain. The set
$\mathfrak{Q}_{[0,1]}$ of all rational numbers in $[0,1]$ is a maximal
down-extended lower $\leq$-chain, but is neither a maximal lower $\leq$-chain,
nor a lower $\leq$-gap chain. The set $\{0\}\cup\mathfrak{Q}_{[\frac{1}{2}%
,1]}$ is a maximal down-extended lower $\leq$-chain but not a lower $\leq$-gap
chain; $\mathfrak{Q}_{[\frac{1}{2},1]}$ is a lower $\leq$-chain but not
maximal down-extended. The sets $\{0,1\}$ and $\{0,\frac{1}{2},1\}$ are
maximal down-extended lower $\leq$-gap chains and $\{\frac{1}{2},1\}$ is a
lower $\ll$-gap chain but not maximal down-extended.

\begin{proposition}
\label{P2.6}Let $G$ be a $\wedge$-complete set in $Q,$ let $S$ be a
complete$,$ lower $\ll$-gap chain in $G$ and $b:=\vee G\in S.$ Then there is a
maximal down-extended complete\emph{,} lower $\ll$-gap chain $C$ in $G$ that%
\begin{equation}
1)\text{ down-extends }G\text{ and }G\cap\lbrack\ll,\mathfrak{p}%
]=\{\mathfrak{p}\},\text{ where }\mathfrak{p}=\wedge C\in C. \label{2.16}%
\end{equation}
If $G$ is a lower $\ll$-set\emph{,} there is a complete\textit{ }lower $\ll
$-gap chain $C$ in $G$ such that $\wedge G=\wedge C=\mathfrak{p}$ and $\vee
C=b.$
\end{proposition}

\begin{proof}
Let $\Gamma_{S}$ be the set of all complete,\textit{ }lower $\ll$-gap chains
in $G$ that down-extend $S.$ It is partially ordered by the inclusion
"down-extends". Let $\Omega$ be a linearly ordered subset of $\Gamma_{S}.$ Set
$K=\cup_{_{C\in\Omega}}C$ and $T=K\cup\{\wedge K\}.$ Then $T$ is a chain in
$G$ and it down-extends each $D\in\Omega$, as%
\[
T\cap\lbrack\wedge D,b]=(\cup_{_{C\in\Omega}}C)\cap\lbrack\wedge
D,b]=\cup_{_{C\in\Omega}}(C\cap\lbrack\wedge D,b])\overset{(\ref{2.7})}{=}D.
\]

Let $x\in K$ and $\wedge K<x.$ Then $x\in D$ and $\wedge D<x$ for some
$D\in\Omega.$ As $D$ is a lower $\ll$-gap chain, there is $p\in D$ such that
$p\ll x\neq p$ and $D\cap\lbrack p,x]=\{p,x\}$. By the above,%
\[
T\cap\lbrack p,x]=T\cap([\wedge D,b]\cap\lbrack p,x])=D\cap\lbrack
p,x]=\{p,x\},
\]
so that $T$ is a lower $\ll$-gap chain in $G$ that down-extends $S.$

To prove that $T$ is $\wedge$-complete, let $E\subseteq T$. If $\wedge K\in E$
then $\wedge E=\wedge K\in T\mathfrak{.}$ If $\wedge K\notin E$, set $E_{_{C}%
}=C\cap E$ for $C\in\Omega.$ As each $C$ is complete, $e_{_{C}}$:$=\wedge
E_{_{C}}\in C$. If all $e_{_{C}}$ lie in a chain $D\in\Omega$ then, as $D$ is
complete,%
\[
\wedge E\overset{(\ref{2.5})}{=}\wedge_{_{C\in\Omega}}\left(  \wedge E_{_{C}%
}\right)  =\wedge_{_{C\in\Omega}}e_{_{C}}\in D\subseteq T\mathfrak{.}%
\]
If not then, for each $D\in\Omega$, there is $C\in\Omega$ such that
$e_{C}\notin D.$ Thus $\wedge C\leq e_{_{C}}<\wedge D$. Hence
\[
\wedge K\overset{(\ref{2.5})}{=}\wedge_{_{C\in\Omega}}\left(  \wedge C\right)
\leq\wedge_{_{C\in\Omega}}e_{_{C}}\leq\wedge_{_{D\in\Omega}}\wedge D=\wedge
K.
\]
Therefore $\wedge E=\wedge_{_{C\in\Omega}}e_{_{C}}=\wedge K\in T\mathfrak{,}$
so that $T$ is $\wedge$-complete. By Lemma \ref{L2.1p}(iii), $T$ is complete.
Thus $T\in\Gamma_{S}$ and $T$ is a supremum of $\Omega.$ By Zorn's lemma,
$\Gamma_{S}$ has a maximal element -- a complete,\textit{ }lower $\ll$-gap
chain $C$ in $G$ that down-extends $S.$

Set $\mathfrak{p}=\wedge C$. As $C$ is complete, $\mathfrak{p}\in C\subseteq
G.$ If $x\ll\mathfrak{p}$ for some $x\in G,$ $\mathfrak{p}\neq x,$ then
$C\cup\{x\}\in\Gamma_{S}$ and larger than $C.$ As $C$ is maximal, such $x$
does not exist. So $G\cap\lbrack\ll,\mathfrak{p}]=\{\mathfrak{p}\}.$

\ Let $G$ be a lower $\ll$-set. If $\mathfrak{p}\neq\wedge G$ then there is
$x\in G$ such that $x\neq\mathfrak{p}$ and $x\ll\mathfrak{p}$ which
contradicts (\ref{2.16})$.$ Hence $\mathfrak{p}=\wedge G.$ The chain $S=\{b\}$
is a complete$,$ lower $\ll$-gap chain in $G$ containing $b.$ By (i) and the
above argument, there is a complete\emph{,} lower $\ll$-gap chain $C$ in $G$
that down-extends $S$ and $\mathfrak{p}=\wedge C=\wedge G.$\bigskip
\end{proof}

Let $\gamma$ be an ordinal number. A set $\left(  x_{\lambda}\right)
_{1\leq\lambda\leq\gamma}$ in $Q$ is a\emph{ }\textit{descending
}(respectively, an \textit{ascending}) $\ll$\textit{-series} from $a$ to $b$
if%
\begin{align}
x_{\lambda+1}  &  \ll x_{\lambda}\text{ for }\lambda<\gamma,\text{ }x_{\beta
}=\wedge_{\lambda<\beta}(x_{\lambda})\text{ for limit ordinals }\beta,\text{
}x_{1}=b,\text{ }x_{\gamma}=a;\nonumber\\
\text{respectively, }x_{\lambda}  &  \ll x_{\lambda+1}\text{ for }%
\lambda<\gamma,\text{ }x_{\beta}=\vee_{\lambda<\beta}(x_{\lambda})\text{ for
limit ordinals }\beta,\text{ }x_{1}=a,\text{ }x_{\gamma}=b. \label{2.12}%
\end{align}

\begin{proposition}
\label{p10}Let $\ll\mathbf{\in}$ \emph{Rel}$\left(  Q,\leq\right)  $ and $C$
be a complete chain in $Q$ from $a$ to $b.$ Then \smallskip

\emph{(i)} $C$ is an upper $\ll$-gap chain if and only if it is an ascending
$\ll$-series.\smallskip

\emph{(ii) }$C$ is a lower $\ll$-gap chain if and only if it is a descending
$\ll$-series.
\end{proposition}

\begin{proof}
(i) Any ascending $\ll$-series $\left(  x_{\lambda}\right)  _{1\leq\lambda
\leq\gamma}$ is, clearly, a $\vee$-complete, upper $\ll$-gap chain. By dual to
Lemma \ref{L2.1p}(iii), it is complete. Conversely, let $C$ be a complete,
upper $\ll$-gap chain from $a$ to $b$. Let $\varnothing\neq G\neq\left\{
b\right\}  $ be any subset of $C.$ Then $\wedge G\in C,$ as $C$ is complete,
and $\wedge G\neq b$. Hence there is $d\in C$ such that $[\wedge G,d]\cap C$
is a gap. So $\wedge G\in G$.

Suppose that, for some ordinal $\beta,$
\begin{equation}
C=C_{\beta}\cup G_{\beta},\text{ }G_{\beta}\cap C_{\beta}=\varnothing\text{
and }\vee C_{\beta}\leq\wedge G_{\beta}, \label{2.17}%
\end{equation}
where $C_{\beta}=\left(  c_{\alpha}\right)  _{\alpha<\beta}$ is an ascending
$\ll$-series. Let $G_{\beta}\neq\varnothing.$ If there is $\alpha_{0}$ such
that $\beta=\alpha_{0}+1$ then $c_{\alpha_{0}}$ is the largest element in
$C_{\beta}.$ As $C$ is an upper $\ll$-gap chain, take $c_{\beta}=c_{\alpha
_{0}+1}$ equal to an immediate $\ll$-successor of $c_{\alpha_{0}}.$ Clearly,
$c_{\beta}$ belongs to $G_{\beta}$. If such $\alpha_{0}$ does not exist then
$\beta$ is a limit ordinal. Set $c_{\beta}=\wedge G_{\beta}$. By the above,
$c_{\beta}\in G_{\beta}$ and $c_{\alpha}<c_{\beta}$ for all $\alpha<\beta$. In
both cases, extend $C_{\beta}$ to the ascending $\ll$-series $C_{\beta
+1}=\left(  c_{\alpha}\right)  _{\alpha<\beta+1}$ and restrict $G_{\beta}$ to
$G_{\beta+1}=G_{\beta}\diagdown\left\{  c_{\beta}\right\}  $. Then
$C=C_{\beta+1}\cup G_{\beta+1}.$

Let $\gamma$ be a limit ordinal. Suppose that there are $\{C_{\beta},G_{\beta
}\},$ $\beta<\gamma,$ satisfying (\ref{2.17}). Set $C_{\gamma}=\cup
_{\beta<\gamma}C_{\beta}$ and $G_{\gamma}=\cap_{\beta<\gamma}G_{\beta}.$ It is
easy to see that $C=C_{\gamma}\cup G_{\gamma},$ $C_{\gamma}\cap G_{\gamma
}=\varnothing$ and $c\leq g$ for all $c\in C_{\gamma}$ and $g\in G_{\gamma},$
so that $C_{\gamma}\leq\wedge G_{\gamma}.$ Thus $\{C_{\gamma},G_{\gamma}\}$
satisfies (\ref{2.17}). Moreover, $C_{\gamma}$ is an ascending $\ll$-series.
Thus, by transfinite induction, (\ref{2.17}) holds for all ordinals. As $C$ is
a set, find an ordinal $\gamma$ such that $G_{\gamma}=\varnothing.$ Then
$C=C_{\gamma}$ is an ascending $\ll$-series and $b=\vee C=a_{\gamma}$. The
proof of (ii) is similar.\bigskip
\end{proof}

For each $\ll$ $\in$ Ref($Q)$, define the reflexive relations $\ll
^{\triangleleft}$ and $\ll^{\triangleright}$ from Ref($Q)$:%
\begin{align}
a  &  \ll^{\triangleleft}b\text{ if }a<b\text{ and there is a complete lower
}\ll\text{-gap chain from }a\text{ to }b;\nonumber\\
a  &  \ll^{\triangleright}b\text{ if }a<b\text{ and there is a complete upper
}\ll\text{-gap chain from }a\text{ to }b. \label{2.18}%
\end{align}
If $a\ll b$ then $\{a,b\}$ is a complete$,$ lower and upper $\ll$-gap chain,
i.e., $a\ll^{\triangleleft}b$ and $a\ll^{\triangleright}b.$ Thus%
\begin{equation}
\ll\text{ }\subseteq\text{ }\ll^{\triangleleft}\text{ and }\ll\text{
}\subseteq\text{ }\ll^{\triangleright}\text{ }(\text{see (\ref{0})}).
\label{2,8}%
\end{equation}

\begin{theorem}
\label{P2.1}\emph{(i) }Let $\ll$ be a relation from \emph{Ref(}$Q).$ Then
$\ll^{\triangleleft}$ and $\ll^{\triangleright}$ are orders\emph{,}%
\begin{equation}
\ll^{\text{\emph{lo}}}\text{ }\subseteq\text{ }\ll^{\triangleleft}=\left(
\ll^{\triangleleft}\right)  ^{\triangleleft}\text{ and }\overrightarrow{\ll
}=\overrightarrow{\ll^{\triangleleft}};\text{ \ }\ll^{\text{\emph{up}}}\text{
}\subseteq\text{ }\ll^{\triangleright}=\left(  \ll^{\triangleright}\right)
^{\triangleright}\text{ and }\overleftarrow{\ll}=\overleftarrow{\ll
^{\triangleright}}. \label{2.19}%
\end{equation}

\emph{(ii) \ }If $\ll$ is a down-expanded order\emph{,} $\ll$ $=$
$\ll^{\triangleleft}$. If $\ll$ is a dual $\mathbf{T}$-order then $\ll$ $=$
$\ll^{\triangleleft}$ $=$ $\ll^{\text{\emph{lo}}}$.\smallskip

\emph{(iii) }If $\ll$ is an up-expanded order then $\ll$ $=$ $\ll
^{\triangleright}$. If $\ll$ is a $\mathbf{T}$-order then $\ll$ $=$
$\ll^{\triangleright}$ $=$ $\ll^{\text{\emph{up}}}$.
\end{theorem}

\begin{proof}
(i) Clearly, $\ll^{\triangleleft}$ and $\ll^{\triangleright}$ are transitive,
so they are orders.

If $a$ $\ll^{\text{lo}}$ $b$ then $\left[  a,b\right]  $ is a complete lower
$\ll$-set. Hence, by Proposition \ref{P2.6}, there is a complete, lower $\ll
$-gap chain from $a$ to $b.$ Thus $a$ $\ll^{\triangleleft}$ $b,$ so that
$\ll^{\text{lo}}$ $\subseteq$ $\ll^{\triangleleft}$.

By (\ref{2,8}), $\ll^{\triangleleft}$ $\subseteq$ $\left(  \ll^{\triangleleft
}\right)  ^{\triangleleft}.$ Conversely, let $a$ $\left(  \ll^{\triangleleft
}\right)  ^{\triangleleft}$ $b,$ $a\neq b.$ Then there is a descending
$\ll^{\triangleleft}$-series $\left(  x_{\alpha}\right)  _{1\leq\alpha
\leq\gamma}$ such that $x_{1}=b$ and $x_{\gamma}=a$. For each ordinal
$\alpha,$ $x_{\alpha+1}\ll^{\triangleleft}x_{\alpha},$ so that there is a
descending $\ll$-series $\left(  u_{\alpha\beta}\right)  _{1\leq\beta
\leq\gamma\left(  \alpha\right)  }$ such that $u_{\alpha1}=x_{\alpha}$ and
$u_{\alpha\gamma\left(  \alpha\right)  }=x_{\alpha+1}$. Renumbering
$\cup_{1\leq\alpha\leq\gamma}\left(  u_{\alpha\beta}\right)  _{1\leq\beta
\leq\gamma\left(  \alpha\right)  },$ we get a descending $\ll$-series $\left(
v_{\lambda}\right)  _{1\leq\lambda\leq\delta}$ such that $v_{1}=b$ and
$v_{\delta}=a$. This shows that $(\ll^{\triangleleft})^{\triangleleft}$
$\subseteq$ $\ll^{\triangleleft}$. So $(\ll^{\triangleleft})^{\triangleleft}$
$=$ $\ll^{\triangleleft}$.

As $\ll\/\subseteq$ $\ll^{\triangleleft}$ by (\ref{2,8}), we have
$\overrightarrow{\ll^{\triangleleft}}\subseteq\overrightarrow{\ll}$ by
Proposition \ref{Lat1}(iii).

Let $a$ $\overrightarrow{\ll}$ $b$. Suppose that $x\ll^{\triangleleft}b$ for
some $x\in\lbrack a,b).$ By (\ref{2.18}), there is a complete lower $\ll$-gap
chain $C$ from $x$ to $b.$ Hence, by Definition \ref{D3}, $b$ has an immediate
$\ll$-predecessor\emph{ }$p\in C$:\emph{ }$a\leq x\leq p<b$ and $p\ll b.$ This
contradicts $a$ $\overrightarrow{\ll}$ $b$. So there is no $x\in\lbrack a,b)$
such that $x\ll^{\triangleleft}b,$ i.e., $a$ $\overrightarrow{\ll
^{\triangleleft}}$ $b.$ Hence $\overrightarrow{\ll}\subseteq
\overrightarrow{\ll^{\triangleleft}}$ whence $\overrightarrow{\ll
}=\overrightarrow{\ll^{\triangleleft}}$ which completes the proof of the first
part of (\ref{2.19}).

The proof of the first part of (\ref{2.19}) is similar.

(ii) By (\ref{2,8}), $\ll$ $\subseteq$ $\ll^{\triangleleft}.$ Let $a$
$\ll^{\triangleleft}$ $b.$ By (\ref{2.18}) and Proposition \ref{p10}, there is
a descending $\ll$-series $\left(  x_{\lambda}\right)  _{1\leq\lambda
\leq\gamma}$ from $a$ to $b$: $x_{\lambda+1}<x_{\lambda},$ $x_{\lambda+1}\ll
x_{\lambda}$ for $\lambda<\gamma,$ $x_{\beta}=\wedge_{\lambda<\beta
}(x_{\lambda})$ for limit ordinals $\beta,$ $x_{1}=b,$ $x_{\gamma}=a.$ Suppose
that $x_{\lambda}\ll b$ for some $\lambda.$ Then $x_{\lambda+1}\ll x_{\lambda
}\ll b.$ As $\ll$ is an order, $x_{\lambda+1}\ll b.$ Let $\beta$ be a limit
ordinal and $x_{\lambda}\ll b$ for all $\lambda<\beta.$ As $\ll$ is
down-expanded, $x_{\beta}=\wedge_{\lambda<\beta}(x_{\lambda})\ll b.$ By
transfinite induction, $a=x_{\gamma}\ll b.$ Thus $\ll^{\triangleleft}$
$\subseteq$ $\ll$ whence $\ll$ $=$ $\ll^{\triangleleft}$.

If $\ll$ is a dual $\mathbf{T}$-order, it is down-contiguous. Hence $\ll$
$\subseteq$ $\ll^{\text{lo}}$ by Proposition \ref{p11}$.$ By (\ref{2.19}),
$\ll^{\text{lo}}$ $\subseteq$ $\ll^{\triangleleft}.$ As $\ll$ $=$
$\ll^{\triangleleft},$ we have $\ll$ $=$ $\ll^{\text{lo}}$ $=$ $\ll
^{\triangleleft}$.

Part (iii) is proved by duality.\bigskip
\end{proof}

Even if $\ll$ is contiguous, the relations $\ll^{\triangleright}$ and
$\ll^{\triangleleft}$ are not necessarily contiguous.

\begin{example}
\label{E3.1}\emph{Let} $Q=\{\mathbf{0},a,b,\mathbf{1\},}$ $\mathbf{0}%
<a<\mathbf{1,}$ $\mathbf{0}<b<\mathbf{1.}$ \emph{Let }$\ll$ \emph{be a
reflexive relation in} $Q,$ $\mathbf{0}\ll a\ll\mathbf{1}$\emph{ and}
$\mathbf{0}\not \ll \mathbf{1.}$ \emph{Then }$\ll$\emph{ }is contiguous, while
$\ll^{\triangleright}$ and $\ll^{\triangleleft}$ are not contiguous\emph{.}

\emph{Indeed, }$\mathbf{0}\ll^{\triangleright}\mathbf{1}.$ \emph{However,}
$[\mathbf{0,1]\nsubseteq\lbrack\ll^{\triangleright},1]}$ \emph{and
}$[\mathbf{0,1]\nsubseteq\lbrack0,\ll^{\triangleright}],}$ \emph{since}
$b\in\lbrack\mathbf{0,1],}$ $\mathbf{0\not \ll ^{\triangleright}}b$ \emph{and}
$b\not \ll ^{\triangleright}\mathbf{1.}$ \emph{Thus }$\ll^{\triangleright}%
$\emph{ is not contiguous. Similarly, }$\ll^{\triangleleft}$ \emph{is not
contiguous. \ \ }$\blacksquare$
\end{example}

If, however, $Q$ is a complete \textit{chain} and $\ll$ is contiguous then
$\ll^{\triangleright}$ and $\ll^{\triangleleft}$ are contiguous.

\begin{lemma}
\label{L3.4}Let $Q$ be a complete chain. If $\ll$ from \emph{Ref}$(Q)$ is
up-contiguous then the relations $\ll^{\triangleright},$ $\ll^{\triangleleft}$
are up-contiguous. If $\ll$ is down-contiguous\emph{,} $\ll^{\triangleright},$
$\ll^{\triangleleft}$ are down-contiguous.
\end{lemma}

\begin{proof}
Let $a\ll^{\triangleright}b$ and let $T$ be a complete upper $\ll$-gap chain
from $a$ to $b.$ Then $a\in\lbrack\ll^{\triangleright},b].$ Let $x\in(a,b].$
Set $c=\vee\{t\in T$: $t\leq x\}$ and $c^{\prime}=\wedge\{t\in T$: $x\leq
t\}.$ As $T$ is complete, $c,c^{\prime}\in T$ and $c\leq x\leq c^{\prime}.$ If
there is $r\in T$ such that $c<r<c^{\prime},$ then either $x\leq r,$ or $r<x,$
as $Q$ is a chain. If $x\leq r$ then $c^{\prime}\leq r,$ a contradiction.
Similarly, the condition $r<x$ gives a contradiction. Thus $[c,c^{\prime
}]_{_{T}}$ is a gap, so that $c^{\prime}=c_{s}$ is the immediate $\ll
$-successor of $c$ in $T.$

Let $\ll$ be up-contiguous. Then $c\ll c_{s}$ and $c\leq x\leq c_{s}$ imply
$x\ll c_{s}.$ Hence $\{x\}\cup\{t\in T$: $c_{s}\leq t\}$ is a complete upper
$\ll$-gap chain from $x$ to $b.$ Thus $x\ll^{\triangleright}b$ and
$\ll^{\triangleright}$ is up-contiguous.

If $\ll$ is down-contiguous then $c\ll c_{s}$ and $c\leq x\leq c_{s}$ imply
$c\ll x.$ Hence $\{t\in T$: $t\leq c\}\cup\{x\}$ is a complete upper $\ll$-gap
chain from $a$ to $x.$ Thus $a\ll^{\triangleright}x$ and $\ll^{\triangleright
}$ is down-contiguous.

Similarly, if $\ll$ be down-contiguous then $\ll^{\triangleright},$
$\ll^{\triangleleft}$ are down-contiguous.\bigskip
\end{proof}

By Theorem \ref{P2.1}, if $\ll$ is a down-expanded order then $\ll
^{\triangleleft}=$ $\ll$ is down-expanded. If, however, $\ll$ is \textbf{not}
down-expanded, $\ll^{\triangleleft}$ is not necessarily down-expanded even if
$Q$ is a complete chain.

\begin{example}
\label{E3.2}Let $Q$ be a complete chain. If $\ll$ is \textbf{not}
down-expanded then $\ll^{\triangleleft}$ is not necessarily down-expanded. If
$\ll$ is \textbf{not} up-expanded then $\ll^{\triangleright}$ is not
necessarily up-expanded.\smallskip

\emph{Indeed, let }$Q=0\cup\{\frac{1}{n}\}_{n=1}^{\infty}$\emph{ be a subset
of }$[0,1]$\emph{ with usual order }$\leq.$\emph{ Then }$(Q,\leq)$ \emph{is a
complete chain. Let }$\ll$\emph{ be a reflexive relation in }$Q$\emph{ such
that only }$\frac{1}{n}\ll1$\emph{ for all }$n=1,2,...$\emph{ Then }%
\[
\ll\emph{\ }\text{\emph{is an order}},\emph{\ }\ll\emph{\ }=\emph{\ }%
\ll^{\triangleleft}\emph{\ }=\text{ }\ll^{\triangleright}\text{\emph{and}%
}\emph{\ }[\ll^{\triangleleft},1]=[\ll,1]=\left\{  1/n\right\}  _{n=1}%
^{\infty}.
\]
\emph{ Hence }$\wedge\lbrack\ll^{\triangleleft},1]=0\notin\lbrack
\ll^{\triangleleft},1].$\emph{ Thus }$[\ll^{\triangleleft},1]$\emph{ is not
}$\wedge$\emph{-complete. So }$\ll,\ll^{\triangleleft}$\emph{ are not
down-expanded.}

\emph{Similar example shows that if }$\ll$\emph{ is not up-expanded, }%
$\ll^{\triangleleft}$\emph{ is not necessarily up-expanded. \ \ }%
$\blacksquare$
\end{example}

While a relation $\ll$ may have neither $\ll$-radicals, nor dual $\ll
$-radicals, the relation $\ll^{\triangleright}$ always has $\ll
^{\triangleright}$-radicals and the relation $\ll^{\triangleleft}$ always has
dual $\ll^{\triangleleft}$-radicals in all $[a,b]\subseteq Q.$

\begin{corollary}
\label{C2.2}For $\ll$ $\in$ \emph{Ref(}$Q),$ each interval $[a,b]\subseteq Q$
has dual $\ll^{\triangleleft}$-radicals and $\ll^{\triangleright}$-radicals.
\end{corollary}

\begin{proof}
The chain $S=\{b\}$ is a complete$,$ lower $\ll$-gap chain in $[a,b]$
containing $b.$ It follows from Proposition \ref{P2.6} that there is a
complete, lower $\ll$-gap chain $C$ in $[a,b]$ that down-extends $S$ and
$[a,b]\cap\lbrack\ll,\mathfrak{p}]=\{\mathfrak{p}\},$ where $\mathfrak{p}%
=\wedge C.$ Hence $a$ $\overrightarrow{\ll}$ $\mathfrak{p}\ll^{\triangleleft
}b$ and $[a,\mathfrak{p}]\cap\lbrack\ll,\mathfrak{p}]=\{\mathfrak{p}\}$ by
(\ref{A1}). By (\ref{2.19}), $\overrightarrow{\ll}=\overrightarrow{\ll
^{\triangleleft}}.$ So $a$ $\overrightarrow{\ll^{\triangleleft}}$
$\mathfrak{p}$ and $\mathfrak{p}$ is a dual $\ll^{\triangleleft}$-radical in
$[a,b]$. Similarly, $\ll^{\triangleright}$-radicals exist.$\bigskip$
\end{proof}

For the uniqueness of the radicals in Corollary \ref{C2.2} we need some extra conditions.

\section{$\mathbf{H}$-relations $\ll$ and the corresponding $\ll
^{\triangleleft}$ and $\ll^{\triangleright}$ relations$.$}

We constructed above the relations $\ll^{\text{lo}},$ $\ll^{\text{up}},$
$\ll^{\triangleleft}$ and $\ll^{\triangleright}.$ In this section we show
that, if $\ll$ is an $\mathbf{H}$-relation then $\ll^{\text{up}}=$
$\ll^{\triangleright}$ is an $\mathbf{R}$-order; if $\ll$ is a dual
$\mathbf{H}$-relation, $\ll^{\text{lo}}=$ $\ll^{\triangleleft}$ is a dual
$\mathbf{R}$-order.

The\thinspace presence\thinspace of\thinspace upper\thinspace(lower)\thinspace
$\ll$-chains\thinspace effects\thinspace the\thinspace structure\thinspace
of\thinspace lattices.

\begin{proposition}
\label{P2.2n}Let $\ll$ be a dual $\mathbf{H}$-relation in $Q.$ Let $C$ be a
$\wedge$-complete\emph{,} lower $\ll$-chain in a $\wedge$-complete set
$G\subseteq Q.$ Let $b=\vee C\in C$ and $a=\wedge C.$ Then\smallskip

\emph{(i) \ \ }For each $z\in\lbrack\mathbf{0},b]\diagdown\lbrack
\mathbf{0},a],$ there is $c\in C$ such that $z\neq c\wedge z$ and $c\wedge
z\ll z$.

\qquad In particular$,$ the sets $[\mathbf{0},b]\diagdown\lbrack\mathbf{0},a]$
and $[a,b]$ are lower $\ll$-sets.\smallskip

\emph{(ii) \ }Each chain in $[a,b]$ larger than $C$ is a lower $\ll
$-chain\emph{.\smallskip}

\emph{(iii)\ }There is a maximal chain $S$ in $G\cap\lbrack a,b]$ containing
$C;$ it is a $\wedge$-complete\emph{,} lower $\ll$-chain$.$

$\qquad$If $G$ is complete\emph{, }$S$ is complete.
\end{proposition}

\begin{proof}
(i) Let $M=\left\{  y\in C:\text{ }z\leq y\right\}  .$ Then $M\neq
\varnothing,$ as $b\in M$. Let $d=\wedge M.$ Then $z\leq d\in C,$ since $C$ is
$\wedge$-complete, and $a<d,$ as $z\notin\lbrack\mathbf{0},a].$ Hence there is
$c\in C$ such that $d\neq c\ll d$. Then $c\notin M$. By Lemma \ref{le1},
$c\wedge z\ll d\wedge z=z.$ Moreover, $c\wedge z\neq z,$ since otherwise
$z\leq c$ and $c\in M$. Thus $[\mathbf{0},b]\diagdown\lbrack\mathbf{0},a]$ is
a lower $\ll$-set.

If $z\in\lbrack a,b]$ then $c\wedge z\in\lbrack a,b].$ So $[a,b]$ is a lower
$\ll$-set.

(ii) Let $D$ be a chain in $[a,b],$ $C\subset D$ and $a\neq z\in D\diagdown
C.$ By (i), there is $c\in C$ such that $z\neq c\wedge z\ll z$. As $D$ is a
chain, $c=c\wedge z.$ Thus $z\neq c\ll z,$ so that $D$ is a lower $\ll
$-chain\emph{.}

(iii) By Lemma \ref{L2.1p}, $S$ exists and it is $\wedge$-complete. If $G$ is
complete, $S$ is complete. By (ii), $S$ is a lower $\ll$-chain.\bigskip
\end{proof}

Let $G$ be a $\wedge$-complete set in $Q$ and $b:=\vee G\in G.$\emph{ }It
follows from Proposition \ref{P2.6} that, for different complete$,$ lower
$\ll$-gap chains $S$ in $G$ containing $b$ (for example, $S=\{b\}),$ there are
maximal down-extended complete, lower $\ll$-gap chains $C_{S}$ in $G$ that
down-extend $S$ and the elements $\mathfrak{p}_{_{C_{S}}}=\wedge C_{S}$
satisfy (\ref{2.16}). We will show now that if $\ll$ is a dual $\mathbf{H}%
$-relation then all $\mathfrak{p}_{_{C_{S}}}$ coincide.

\begin{theorem}
\label{T2.3p}Let $G$ be a $\wedge$-complete set in $Q,$ $b:=\vee G\in G$\emph{
}and $\ll$ be a dual $\mathbf{H}$-relation. Then\smallskip

\emph{(i) \ \ }There is a unique $\mathfrak{p}\in G$ satisfying $G\cap
\lbrack\ll,\mathfrak{p}]=\{\mathfrak{p}\}$\emph{.\smallskip}

\emph{(ii)} \ If $S$ is a complete$,$ lower $\ll$-gap chain in $G$ and $b\in
S$ then $\mathfrak{p}\leq\wedge S.$ If $\mathfrak{p}<\wedge S,$ there is a

\qquad maximal down-extended complete lower $\ll$-gap chain $C\subseteq G$
down-extending $S$\emph{,} $\mathfrak{p}=\wedge C$.\smallskip

\emph{(iii) }The following conditions are equivalent.\smallskip

\qquad$1)$ $G$ is a lower $\ll$-set\emph{;\smallskip}

\qquad$2)$ There is a $\wedge$-complete\emph{,} lower $\ll$-chain $C$ in $G$
from $a$ to $b$ and $b\in C;\smallskip$

\qquad$3)$ There is a complete\emph{,} lower $\ll$-gap chain $S$ in $G$ from
$a$ to $b$ and $b\in S.$
\end{theorem}

\begin{proof}
(i) and (ii). Let $S$ and $T$ be complete$,$ lower $\ll$-gap chains in $G$
containing $b$ (for example, $S=\{b\}).$ By Proposition \ref{P2.6}, there is a
maximal down-extended complete, lower $\ll$-gap chain $C$ in $G$ that
down-extends $S$ and $\mathfrak{p}_{_{C}}=\wedge C$ satisfies
\begin{equation}
G\cap\lbrack\ll,\mathfrak{p}_{_{C}}]=\{\mathfrak{p}_{_{C}}\}\emph{.}
\label{4.9}%
\end{equation}
Let $D$ be another maximal down-extended complete, lower $\ll$-gap chain in
$G$ that either down-extends $S,$ or $T.$ Set $\mathfrak{p}_{_{D}}=\wedge D.$
If $\mathfrak{p}_{_{C}}\notin\lbrack\mathbf{0},\mathfrak{p}_{_{D}}]$ then, as
$\ll$ is a dual $\mathbf{H}$-relation, it follows from Proposition
\ref{P2.2n}(i) that there is $z\in D\subseteq G$ such that $\mathfrak{p}%
_{_{C}}\neq z\wedge\mathfrak{p}_{_{C}}\ll\mathfrak{p}_{_{C}}$ which
contradicts (\ref{4.9})$.$ Thus $\mathfrak{p}_{_{C}}\in\lbrack\mathbf{0}%
,\mathfrak{p}_{_{D}}]$. Similarly, $\mathfrak{p}_{_{D}}\in\lbrack
\mathbf{0},\mathfrak{p}_{_{C}}]$ whence $\mathfrak{p}_{_{D}}=\mathfrak{p}%
_{_{C}}$. Thus $\mathfrak{p}_{_{C}}$ is uniquely defined.

(iii) As $\ll$ is a dual $\mathbf{H}$-relation, 2) $\Rightarrow$ 1) follows
from Proposition \ref{P2.2n}(i).

3) $\Rightarrow$ 2) is evident. As $G$ is a lower $\ll$-set, 1) $\Rightarrow$
3) follows from Proposition \ref{P2.6}.$\medskip$
\end{proof}

The following corollary gives an analogue of Theorem \ref{T2.3p}(i) and (ii)
for lower $\ll$-chains.

\begin{corollary}
\label{C2.5p}Let $G,$ $b,$ $\ll$ and $\mathfrak{p}$ be as in Theorem
\emph{\ref{T2.3p}. }If $T\subseteq G$ is a $\wedge$-complete$,$\textit{ }lower
$\ll$-chain containing $b$ then $\mathfrak{p}\leq\wedge T.$ If $\mathfrak{p}%
<\wedge T,$ there is a maximal down-extended $\wedge$-complete\textit{ }lower
$\ll$-chain $\mathcal{T}$ in $G$ from $\mathfrak{p}$ to $b$ that down-extends
$T$ and $\wedge\mathcal{T}=\mathfrak{p}.$
\end{corollary}

\begin{proof}
Set $t=\wedge T\in T\subseteq G.$ As $\ll$ is a dual $\mathbf{H}$-relation, it
follows from Theorem \ref{T2.3p} (iii) that there is a complete$,$ lower $\ll
$-gap chain $S$ in $G$ from $t$ to $b$. By Theorem \ref{T2.3p}(ii),
$\mathfrak{p}\leq t.$ If $\mathfrak{p}<t$\ then there is a complete$,$ lower
$\ll$-gap chain $C$ in $G$ from $\mathfrak{p}$ to $b$ that down-extends $S.$
Then $T^{\prime}=C\cap\lbrack\mathfrak{p},t]$ is a complete$,$ lower $\ll$-gap
chain from $\mathfrak{p}$ to $t.$ Hence $\mathcal{T}=T^{\prime}\cup T$ is a
$\wedge$-complete$,$ lower $\ll$-chain in $G$ from $\mathfrak{p}$ to $b$ that
down-extends $T.$ As $\wedge\mathcal{T}=\mathfrak{p},$ $\mathcal{T}$ can not
be down-extended. So $\mathcal{T}$ is maximal down-extended.\bigskip
\end{proof}

The results analogous to Theorem \ref{T2.3p} and Corollary \ref{C2.5p} hold by
duality for upper $\ll$-chains and upper $\ll$-gap chains if $\ll$ is an
$\mathbf{H}$-relation.

To refine the results of Theorem \ref{P2.1} for $\mathbf{H}$-relations, we
define for each $a\in Q,$ the maps $\curlyvee_{a}$,$\curlywedge_{a}$ on $Q$ by
setting%
\begin{equation}
\curlyvee_{a}(x)=a\vee x\text{ and }\curlywedge_{a}(x)=a\wedge x\text{ for
}x\in Q. \label{5.3}%
\end{equation}
For $G\subseteq Q$, let $\curlyvee_{a}\left(  G\right)  =\{\curlyvee
_{a}\left(  x\right)  $: $x\in G\}$ and $\curlywedge_{a}\left(  G\right)
=\{\curlywedge_{a}\left(  x\right)  $: $x\in G\}$. Then
\begin{equation}
\curlywedge_{a}(\wedge G)=\wedge\curlywedge_{a}\left(  G\right)  \text{ and
}\curlyvee_{a}(\vee G)=\vee\curlyvee_{a}\left(  G\right)  . \label{e.1}%
\end{equation}
For example, $\curlywedge_{a}(\wedge G)\leq\curlywedge_{a}(g)$ for $g\in G.$
So $\curlywedge_{a}(\wedge G)\leq\wedge\curlywedge_{a}(G).$ Conversely,
$\wedge\left(  \curlywedge_{a}(G)\right)  \leq g$ for $g\in G,$ so that
$\wedge\left(  \curlywedge_{a}(G)\right)  \leq\wedge G.$ As $\wedge\left(
\curlywedge_{a}(G)\right)  \leq a,$ we have $\wedge\left(  \curlywedge
_{a}(G)\right)  \leq\curlywedge_{a}(\wedge G).$ Thus $\curlywedge_{a}(\wedge
G)=\wedge\curlywedge_{a}\left(  G\right)  .$

We also have (see \cite[Theorem 1.4]{Sk}) that%
\begin{equation}
\text{if }G=\cup_{\alpha\in A}G_{\alpha}\text{ then}\wedge G=\wedge\{\wedge
G_{\alpha}\text{: }\alpha\in A\}\text{ and }\vee G=\vee\{\vee G_{\alpha
}\text{: }\alpha\in A\}. \label{2.5}%
\end{equation}

\begin{proposition}
\label{P2.4}Let $\ll$ be a dual $\mathbf{H}$-relation and let $G,$
$\{G_{\alpha}\}_{\alpha\in A}$ be $\wedge$-complete lower $\ll$\textbf{-}%
sets.\smallskip

\emph{(i) \ }For each $a\in Q,$ the set $\curlywedge_{a}(G)$ is a $\wedge
$-complete$,$ lower $\ll$\textbf{-}set$.\smallskip$

\emph{(ii) }Let $b=\vee G_{\alpha}\in G_{\alpha}$ for all $\alpha\in A.$ Then
the $\wedge$-completion $(\cup_{\alpha\in A}G_{\alpha})^{\wedge}$ of
$\cup_{\alpha\in A}G_{\alpha}$ $($see $(\ref{1.10}))$ is a lower $\ll
$\textbf{-}set.
\end{proposition}

\begin{proof}
(i) Let $\wedge\curlywedge_{a}(G)<\curlywedge_{a}(g)$ for some $g\in G.$ Set
$E_{g}=\left\{  z\in G\text{: }\curlywedge_{a}(g)\leq\curlywedge
_{a}(z)\right\}  .$ Then $g\in E_{g}$ and $e_{g}$:$=\wedge E_{g}\in G$, as $G$
is $\wedge$-complete$.$ Moreover, $e_{g}\in E_{g}$ and $\curlywedge
_{a}(g)=\curlywedge_{a}(e_{g}),$ since
\[
\curlywedge_{a}(e_{g})=\curlywedge_{a}(\wedge E_{g})\overset{(\ref{e.1}%
)}{=}\wedge(\curlywedge_{a}(E_{g}))=\curlywedge_{a}(g),\text{ as }g\in
E_{g}\text{ and }\curlywedge_{a}(g)\leq\curlywedge_{a}(z)\text{ for }z\in
E_{g}.
\]

We have $\curlywedge_{a}(\wedge G)\overset{(\ref{e.1})}{=}\wedge
\curlywedge_{a}(G)<\curlywedge_{a}(g)=\curlywedge_{a}(e_{g})$ by above. As
$\wedge G\leq e_{g},$ it follows that $\wedge G<e_{g}.$ As $G$ is a lower
$\ll$\textbf{-}set, there is $h\in G$ such that $h\ll e_{g}$ and $h<e_{g}$. As
$e_{g}$ is minimal in $E_{g}$, $\curlywedge_{a}(h)\neq\curlywedge_{a}(g).$ So
$\curlywedge_{a}(h)<\curlywedge_{a}(e_{g})=\curlywedge_{a}(g).$ As $\ll$ is a
dual $\mathbf{H}$-relation, $\curlywedge_{a}(h)=a\wedge h\ll a\wedge
e_{g}=\curlywedge_{a}(e_{g})=\curlywedge_{a}(g)$, i.e., $\curlywedge_{a}(h)$
is a $\ll$-predecessor of $\curlywedge_{a}(g).$ Thus $\curlywedge_{a}(G)$ is a
lower $\ll$\textbf{-}set.

For $N\subseteq\curlywedge_{a}(G)$, there is $M\subseteq G$ with
$N=\curlywedge_{a}(M)$. As $G$ is $\wedge$-complete, $\wedge M\in G.$ Hence
$\wedge N=\wedge\curlywedge_{a}(M)\overset{(\ref{e.1})}{=}\curlywedge
_{a}(\wedge M)\in\curlywedge_{a}(G)$. Thus $\curlywedge_{a}(G)$ is $\wedge$-complete.

(ii) Set $K=\cup_{\alpha\in A}G_{\alpha}.$ Let $x\in K^{\wedge}$ be such that
$\wedge K<x.$ Then $x=\wedge M$ for some $M\subseteq K.$ Let $F_{\alpha}=M\cap
G_{\alpha},$ $N_{\alpha}=G_{\alpha}\diagdown F_{\alpha}$ and $n_{\alpha
}=\wedge N_{\alpha}$ for all $\alpha\in A$. Then $K=(\cup_{\alpha}N_{\alpha
})\cup M$ and%
\[
\wedge K\overset{\left(  \ref{2.5}\right)  }{=}(\wedge\{n_{\alpha}\text{:
}\alpha\in A\})\wedge x=\wedge\{n_{\alpha}\wedge x\text{: }\alpha\in A\}.
\]
If $n_{\alpha}\wedge x=x$ for all $\alpha,$ then $\wedge K=x$ -- a
contradiction. Thus $n_{\beta}\wedge x<x$ for some $\beta.$ As $N_{\beta
}\subseteq G_{\beta},$ we have $\wedge\curlywedge_{x}(G_{\beta})=\curlywedge
_{x}(\wedge G_{\beta})\leq n_{\beta}\wedge x<x$. As $x\leq b\in G_{\beta},$ we
have $x=\curlywedge_{x}(b)\in\curlywedge_{x}(G_{\beta})$. By (i),
$\curlywedge_{x}(G_{\beta})$ is a $\wedge$-complete, lower $\ll$\textbf{-}set.
Hence there is $u\in G_{\beta}$ such that $\curlywedge_{x}(u)\ll
\curlywedge_{x}(b)$ and $\curlywedge_{x}(u)\neq\curlywedge_{x}(b),$ i.e.,
$x\neq x\wedge u\ll x.$ As $x\wedge u\in K^{\wedge}$, $K^{\wedge}$ is a lower
$\ll$\textbf{-}set.\bigskip
\end{proof}

We$\,$will use Proposition$\,$\ref{P2.4}$\,$to prove$\,$the main results$\,$of
this$\,$section.

\begin{theorem}
\label{inf}\emph{(i)\ }Let $\ll$ be a dual $\mathbf{H}$-relation$.$
Then\smallskip

$\qquad1)$ $\ll^{\triangleleft}$ $=$ $\ll^{\text{\emph{lo}}}$ is a dual
$\mathbf{R}$-order and $\overrightarrow{\ll}=\overrightarrow{\ll
^{\triangleleft}}=\left(  \overrightarrow{\ll}\right)  ^{\triangleright
}=\left(  \overrightarrow{\ll}\right)  ^{^{\text{\emph{up}}}}$ is an
$\mathbf{R}$-order$;\smallskip$

$\qquad2)$ for each $[a,b]\subseteq Q,$ there is a unique dual $\ll
^{\triangleleft}$-radical $\mathfrak{p}=\mathfrak{p}_{_{[a,b]}}$ such that $a$
$\overrightarrow{\ll}$ $\mathfrak{p}$ $\ll^{\triangleleft}b.$

\qquad\ \ \ Moreover$,$ $a\leq x$ $\ll^{\triangleleft}$ $b$ implies
$x\in\lbrack\mathfrak{p},b].\smallskip$

\emph{(ii) }Let $\ll$ be an $\mathbf{H}$-relation and $[a,b]\subseteq Q.$
Then\smallskip

$\qquad1)$ $\ll^{\triangleright}$ $=$ $\ll^{\text{\emph{up}}}$ \textit{is an
}$\mathbf{R}$\textit{-order}$,$\textit{ }and $\overleftarrow{\ll
}=\overleftarrow{\ll^{\triangleright}}=\left(  \overleftarrow{\ll}\right)
^{\triangleleft}=\left(  \overleftarrow{\ll}\right)  ^{\text{\emph{lo}}}$ is a
dual $\mathbf{R}$-order$;\smallskip$

$\qquad2)$ for each $[a,b]\subseteq Q,$ there is a unique $\ll^{\triangleright
}$-radical $\mathfrak{r}=\mathfrak{r}_{_{[a,b]}}$ such that $a$ $\ll
^{\triangleright}$ $\mathfrak{r}$ $\overleftarrow{\ll}$ $b.$

\qquad\ \ \ Moreover$,$ $a\ll^{\triangleright}x\leq b$ implies $x\in\lbrack
a,\mathfrak{r}].$
\end{theorem}

\begin{proof}
(i) 1) Let $\ll$ be a dual $\mathbf{H}$-relation. Firstly, let us show that
$\ll^{\triangleleft}$ is down-expanded. Let $G=[\ll^{\triangleleft},c]$ for
some $c\in Q.$ For each $g\in G\diagdown\{c\},$ $G_{g}=\{g,c\}$ is a complete,
lower $\ll^{\triangleleft}$-set and $G=\cup_{g\in G\diagdown\{c\}}G_{g}.$
Hence $G^{\wedge}$ is a $\wedge$-complete, lower $\ll^{\triangleleft}$-set by
Proposition \ref{P2.4}. As $\wedge(G^{\wedge})=\wedge G,$ it follows from
Theorem \ref{T2.3p}(iii) that there is a complete\textit{ }lower
$\ll^{\triangleleft}$-gap chain $C$ in $G^{\wedge}$ such that $\wedge C=\wedge
G$ and $\vee C=c.$ Hence $\wedge G$ $\left(  \ll^{\triangleleft}\right)
^{\triangleleft}$ $c.$ As $\ll^{\triangleleft}$ $=$ $\left(  \ll
^{\triangleleft}\right)  ^{\triangleleft}$ by Theorem \ref{P2.1}, we have
$\wedge G$ $\ll^{\triangleleft}$ $c.$ Thus $[\ll^{\triangleleft},c]$ is
$\wedge$-complete whence $\ll^{\triangleleft}$ is down-expanded.

It follows from Theorem \ref{T2.3p}(iii) that $[a,b]$ is a lower $\ll$-set
($a$ $\ll^{\text{lo}}$ $b$ (see Definition \ref{D2})) if and only if $a$
$\ll^{\triangleleft}$ $b.$ Thus $\ll^{\triangleleft}$ $=$ $\ll^{\text{lo}}$ is down-expanded.

Let us now prove that $\overrightarrow{\ll}$ is an $\mathbf{H}$-order. Suppose
that $a\ \overrightarrow{\ll}\ b$ and $a\leq c$. Assume that $c\leq x\ll b\vee
c$ for some $x.$ As $\ll$ is a dual $\mathbf{H}$-relation, $a\leq x\wedge
b\ll(b\vee c)\wedge b=b.$ As $a\ \overrightarrow{\ll}\ b,$ we have $x\wedge
b=b$ by (\ref{A1}). So $b\leq x.$ Thus $x=b\vee c.$ From this we conclude that
$c\ \overrightarrow{\ll}\ b\vee c.$ It follows from Lemma \ref{le1} that
$\overrightarrow{\ll}$ is an $\mathbf{H}$-relation.

To show that $\overrightarrow{\ll}$ is an order, let $a$\ $\overrightarrow{\ll
}\ b$\ $\overrightarrow{\ll}$ $c$ and $a\leq x\ll c$. By Lemma \ref{le1},
$a\leq x\wedge b\ll c\wedge b=b.$ As $a$\ $\overrightarrow{\ll}\ b,$ we have
from (\ref{A1}) that $x\wedge b=b.$ Hence $b\leq x\ll c.$ As $b$%
\ $\overrightarrow{\ll}$ $c,$ we have from (\ref{A1}) that $x=c.$ Thus
$a$\ $\overrightarrow{\ll}\ c$ by (\ref{A1}). So $\overrightarrow{\ll}$ is
transitive. Thus $\overrightarrow{\ll}$ is an $\mathbf{H}$-order.

Similarly, if $\ll$ is an $\mathbf{H}$-relation, $\overleftarrow{\ll}$ is a
dual $\mathbf{H}$-order. Hence if $\ll$ is a dual $\mathbf{H}$-relation then
$\overleftarrow{\left(  \overrightarrow{\ll}\right)  }$ is a dual $\mathbf{H}%
$-order.\ By Proposition \ref{p11},\ $\ll^{\text{lo}}=\overleftarrow{\left(
\overrightarrow{\ll}\right)  }$ whence $\ll^{\text{lo}}$ is a dual
$\mathbf{H}$-order. Combining this with the fact that $\ll^{\triangleleft}$
$=$ $\ll^{\text{lo}}$ is down-expanded, we get that $\ll^{\triangleleft}$ $=$
$\ll^{\text{lo}}$ is a dual $\mathbf{R}$-order.

By duality, $\ll^{\triangleright}$\ $=$ $\ll^{\text{up}}$ is\textit{ }an
$\mathbf{R}$-order, if $\ll$ is an $\mathbf{H}$-relation. Hence, as
$\overrightarrow{\ll}$ is an $\mathbf{H}$-order by above, $\left(
\overrightarrow{\ll}\right)  ^{\triangleright}$ $=$ $\left(
\overrightarrow{\ll}\right)  ^{\text{up}}.$ Therefore, since $\ll^{\text{lo}%
}=$ $\overleftarrow{\left(  \overrightarrow{\ll}\right)  }$ and $\ll
^{\text{up}}=$ $\overrightarrow{\left(  \overleftarrow{\ll}\right)  }$ by
Proposition \ref{p11},%
\begin{equation}
\left(  \overrightarrow{\ll}\right)  ^{\triangleright}=\text{ }\left(
\overrightarrow{\ll}\right)  ^{\text{up}}=\overrightarrow{\left(
\overleftarrow{\left(  \overrightarrow{\ll}\right)  }\right)  }%
=\overrightarrow{\ll^{\text{lo}}}=\overrightarrow{\ll^{\triangleleft}}.
\label{4.2}%
\end{equation}
By (\ref{2,8}) and Proposition \ref{Lat1}, $\overrightarrow{\ll^{\triangleleft
}}$ $\subseteq$ $\overrightarrow{\ll}.$ By Proposition \ref{p11},
$\overrightarrow{\ll}$ $\subseteq$ $\overrightarrow{\ll^{\text{lo}}}$. As
$\ll^{\triangleleft}$ $=$ $\ll^{\text{lo}},$ we have $\overrightarrow{\ll
^{\triangleleft}}$ $\subseteq$ $\overrightarrow{\ll}$ $\subseteq$
$\overrightarrow{\ll^{\triangleleft}}$. So $\overrightarrow{\ll}$ $=$
$\overrightarrow{\ll^{\triangleleft}}$. By (\ref{4.2}), $\overrightarrow{\ll}$
$=$ $\left(  \overrightarrow{\ll}\right)  ^{\triangleright}.$ So, as
$\overrightarrow{\ll}$ is an $\mathbf{H}$-order, $\left(  \overrightarrow{\ll
}\right)  ^{\triangleright}$ is an $\mathbf{R}$-order by above. Thus
$\overrightarrow{\ll}$ is an $\mathbf{R}$-order. The proof of 1) is complete.

$2)$ As $\ll^{\triangleleft}$ is a dual $\mathbf{R}$-order, it is a dual
$\mathbf{T}$-order. By Theorem \ref{eur} and (\ref{2.11}), for each $\left[
a,b\right]  \subseteq Q$ there is a unique dual $\ll^{\triangleleft}$-radical
$\mathfrak{p}\in\left[  a,b\right]  $ such that $a$\ $\overrightarrow{\ll
^{\triangleleft}}$ $\mathfrak{p}$ $\ll^{\triangleleft}$ $b.$ As
$\overrightarrow{\ll^{\triangleleft}}$ $=$ $\overrightarrow{\ll}$ by 1), $a$
$\overrightarrow{\ll}$ $\mathfrak{p}$ $\ll^{\triangleleft}b.$

Let $a\leq x$ $\ll^{\triangleleft}$ $b.$ As $\ll^{\triangleleft}$ is a dual
$\mathbf{R}$-order, $a\leq x\wedge\mathfrak{p}$ $\ll^{\triangleleft}$
$b\wedge\mathfrak{p}=\mathfrak{p}$ by Lemma \ref{le1}(ii). As $a$
$\overrightarrow{\ll^{\triangleleft}}$ $\mathfrak{p}$ by above, we have from
(\ref{A1}) that $x\wedge\mathfrak{p}=\mathfrak{p}.$ Hence $\mathfrak{p}\leq
x.$

The proof of part (ii) is similar.\bigskip
\end{proof}

The following corollary strengthens Theorem \ref{eur1}

\begin{corollary}
\label{pr3}Let $\ll$ be a dual $\mathbf{H}$-relation. The following conditions
are equivalent.\smallskip

\emph{(i) \ \ }$\ll$ is a dual $\mathbf{R}$-order. \ \ \ \ \ \ \ \ \emph{(ii)
\ }$\ll$ $=$ $\ll^{\triangleleft}.$\ \smallskip

\emph{(iii)} $\ll$ is an order and $\wedge_{n\in\mathbb{N}}\left(
x_{n}\right)  \ll x_{1}$ for each descending $\ll$-series $\left(
x_{n}\right)  _{n\in\mathbb{N}}$\emph{.\smallskip}

\emph{(iv) }For each $[a,b]\subseteq Q,$ there is $c\in\lbrack a,b]$ such that
$a$ $\overrightarrow{\ll}$ $c\leq b.\smallskip$

\emph{(v) \ }Each $[a,b]\subseteq Q$ has a unique dual $\ll$-radical$.$
\end{corollary}

\begin{proof}
(ii) $\Rightarrow$ (i) As $\ll$ is a dual $\mathbf{H}$-relation,
$\ll^{\triangleleft}$ is a dual $\mathbf{R}$-order by Theorem \ref{inf}. So
$\ll$ is a dual $\mathbf{R}$-order.

(i) $\Rightarrow$ (iii). As $x_{n+1}\ll x_{n}$ for all $n,$ and $\ll$ is
transitive, $x_{n}\ll x_{1}.$ As $\ll$ is down-expanded by Definition
\ref{D2.2}, $\wedge\left(  x_{n}\right)  _{n\in\mathbb{N}}\ll x_{1}.$

(iii) $\Rightarrow$ (ii). If $a\ll^{\triangleleft}b,$ there is a descending
$\ll$-series $\left(  x_{\alpha}\right)  _{1\leq\alpha\leq\gamma}$ with
$a=x_{\gamma}$ and $b=x_{1}$. If $x_{\alpha}\ll b$ for some $\alpha,$ then
$x_{\alpha+1}\ll x_{\alpha}\ll b$ implies $x_{\alpha+1}\ll b$ by transitivity
of $\ll$. Hence all $x_{\alpha+n}\ll b.$ If $\beta$ is a limit ordinal and
$\alpha$ is the previous limit ordinal then, $x_{\beta}=\wedge_{n\in
\mathbb{N}}(x_{\alpha+n})\ll x_{\alpha}\ll b$. So $x_{\beta}\ll b$. Thus, by
transfinite induction, $a\ll b$. So $\ll$ $=$ $\ll^{\triangleleft}.$

(i) $\Rightarrow$ (iv). As dual $\mathbf{R}$-orders are dual $\mathbf{T}%
$-orders, it follows from Theorem \ref{eur} (ii) that each $[a,b]\subseteq Q$
has a unique dual $\ll$-radical $\mathfrak{p}$: $a$ $\overrightarrow{\ll
}\,\mathfrak{p}$ $\ll$ $b$ (see (\ref{2.11})). If $\mathfrak{p}=a$ then $a\ll
b.$ If $\mathfrak{p}\neq a,$ set $c=\mathfrak{p}.$ As $\mathfrak{p}$ $\ll$ $b$
implies $\mathfrak{p}\leq b,$ (iv) holds.

(iv) $\Rightarrow$ (ii). Let $a$ $\ll^{\triangleleft}$ $b$ for $a<b.$ If $a$
$\not \ll $ $b$ then, by (iv), $a$ $\overrightarrow{\ll}$ $c\leq b$ for some
$c\neq a$. On the other hand, by Theorem \ref{inf}, $\ll^{\triangleleft}=$
$\ll^{\text{lo}}.$ Thus, by Definition \ref{D2}, $\left[  a,b\right]  $ is a
lower $\ll$-set, so that there is $d\in\lbrack a,b]$ such that $c\neq d\ll c$
which contradicts $a$ $\overrightarrow{\ll}$ $c.$ Thus $a\ll b$. So
$\ll^{\triangleleft}$ $\subseteq$ $\ll.$ Hence $\left(  \ref{2,8}\right)  $
implies $\ll$ $=$ $\ll^{\triangleleft}$.

(i) $\Rightarrow$ (v). As dual $\mathbf{R}$-orders are dual $\mathbf{T}%
$-orders, it follows from Theorem \ref{eur1} that each $[a,b]\subseteq Q$ has
a unique dual $\ll$-radical.

(v) $\Rightarrow$ (iv). If $[a,b]\subseteq Q$ has a dual $\ll$-radical
$\mathfrak{p}$ then $a$ $\overrightarrow{\ll}\,\mathfrak{p}$ $\ll$ $b$ by
(\ref{2.11}). So (iv) holds.$\bigskip$
\end{proof}

It should be noted that Amitsur \cite{Am} got some results equivalent to the
implication (ii) $\Leftrightarrow$ (iv).\smallskip

By duality, the following results are equivalent for an $\mathbf{H}$-relation
$\ll$:\medskip

(i) $\ll$ \textit{is an }$\mathbf{R}$\textit{-order}; (ii) $\ll$ $=$
$\ll^{\triangleright};$ (iii) Each $[a,b]\subseteq Q$ has a unique $\ll
$-radical$.$\medskip

If $\ll$ is an $\mathbf{H}$-relation in $Q,$ it follows from Theorem \ref{inf}
that $\ll^{\triangleright}$ is an $\mathbf{R}$-order and $\overleftarrow{\ll}$
is a dual $\mathbf{R}$-order in $Q.$ Hence each $[a,b]\subseteq Q$ has a
unique $\ll^{\triangleright}$\textit{-}radical\textit{ }and a unique dual
$\overleftarrow{\ll}$-radical$.$

Similarly, if $\ll$ is a dual $\mathbf{H}$-relation in $Q,$ then
$\ll^{\triangleleft}$ is a dual $\mathbf{R}$-order and $\overrightarrow{\ll}$
is an $\mathbf{R}$-order in $Q.$ So each $[a,b]\subseteq Q$ has a unique dual
$\ll^{\triangleleft}$\textit{-}radical\textit{ }and a unique
$\overrightarrow{\ll}$-radical$.$

\begin{proposition}
\label{T2.1}\emph{(i) }Let $\ll$ be an $\mathbf{H}$-relation$.$ Then\emph{,}
for each $[a,b]\subseteq Q,\smallskip$

$\qquad1)$ the $\ll^{\triangleright}$-radical\textit{ }and the dual
$\overleftarrow{\ll}$-radical in $[a,b]$ coincide$;\smallskip$

$\qquad2)$ if there exists a $\ll$-radical in $[a,b],$ it coincides with the
$\ll^{\triangleright}$-radical.\smallskip

\emph{(ii) }Let $\ll$ be a dual $\mathbf{H}$-relation$.$ Then\emph{,} for each
$[a,b]\subseteq Q,\smallskip$

\qquad$1)$ the dual $\ll^{\triangleleft}$-radical and the $\overrightarrow{\ll
}$-radical in $[a,b]$ coincide$;\smallskip$

\qquad$2)$ if there exists a dual $\ll$-radical in $[a,b],$ it coincides with
the dual $\ll^{\triangleleft}$-radical.
\end{proposition}

\begin{proof}
(i) 1) By Theorem \ref{inf}(ii), the $\ll^{\triangleright}$\textit{-}radical
$\mathfrak{r}$ satisfies $a$ $\ll^{\triangleright}$ $\mathfrak{r}$
$\overleftarrow{\ll}$ $b$. By (\ref{2.11}), the dual\textit{ }%
$\overleftarrow{\ll}$\textit{-}radical $\mathfrak{p}$ in\textit{ }$[a,b]$
satisfy $a$ $\overrightarrow{\overleftarrow{\ll}}$ $\mathfrak{p}$
$\overleftarrow{\ll}$ $b.$ By Proposition \ref{p11},
$\overrightarrow{\overleftarrow{\ll}}$ $=$ $\ll^{\text{up}}.$ By Theorem
\ref{inf}, $\ll^{\text{up}}$ $=$ $\ll^{\triangleright}$, as $\ll$ is an
$\mathbf{H}$-relation, So $\overrightarrow{\overleftarrow{\ll}}$ $=$
$\ll^{\triangleright}.$ Hence $a$ $\overrightarrow{\overleftarrow{\ll}}$
$\mathfrak{p}$ $\overleftarrow{\ll}$ $b$ turns into $a$ $\ll^{\triangleright}$
$\mathfrak{p}$ $\overleftarrow{\ll}$ $b.$ Comparing it to $a$ $\ll
^{\triangleright}$ $\mathfrak{r}$ $\overleftarrow{\ll}$ $b$ and taking into
account the uniqueness of the radicals, we get $\mathfrak{r}=\mathfrak{p}.$

2) If $\mathfrak{r}_{_{1}}$ is the $\ll$-radical in $[a,b]$ then $a$ $\ll$
$\mathfrak{r}_{_{1}}$ $\overleftarrow{\ll}$ $b.$ As $\ll$ $\mathbf{\subseteq}$
$\ll^{\triangleleft}$ by (\ref{2,8}), $a$ $\ll^{\triangleleft}$ $\mathfrak{r}%
_{_{1}}$ $\overleftarrow{\ll}$ $b.$ Since (see (i)) $a$ $\ll^{\triangleright}$
$\mathfrak{r}$ $\overleftarrow{\ll}$ $b,$ it follows from the uniqueness of
the radical that $\mathfrak{r}=\mathfrak{r}_{_{1}}.\bigskip$
\end{proof}

For a subset $F\subseteq$ Ref($Q),$ define the relations $\ll_{_{\wedge F}%
}:=\cap_{_{\ll\,\in F}}\ll$ and $\ll_{_{\vee F}}:=\cup_{_{\ll\,\in F}}\ll$ by%
\begin{equation}
a\ll_{_{\wedge F}}b\text{ if }a\ll b\text{ for all }\ll\text{ in }F;\text{
}a\ll_{_{\vee F}}b\text{ if }a\ll b\text{ for some }\ll\text{ in }F. \label{i}%
\end{equation}
They belong to Ref($Q).$ In particular, for $\prec$ and $\ll,$ $a$ $(\prec
\cap\ll)$ $b$ if and only if $a\prec b$ and $a\ll b$.

If $\ll$ is an $\mathbf{H}$-relation and $\prec$ is not an $\mathbf{H}%
$-relation, then $\prec$ $\cap$ $\ll$ is not necessarily an $\mathbf{H}$-relation.

\begin{lemma}
\label{L5.1}Let $\ll$ be an $\mathbf{H}$-relation\emph{, }let $\prec,$
$\sqsubset$ be relations in \emph{Ref(}$Q)$ and $\ll$ be stronger than
$\sqsubset$ \emph{(}see \emph{(\ref{0})). }If $\prec$ $\cap$ $\sqsubset$ is an
$\mathbf{H}$-relation then $\prec$ $\cap$ $\ll$ is also an $\mathbf{H}$-relation.

The same is true for dual $\mathbf{H}$-relations.
\end{lemma}

\begin{proof}
Let $a$ $(\prec\cap\ll)$ $b$ for $a,b\in Q$. Then $a\prec b$ and $a\ll b.$ As
$\ll$ is stronger than $\sqsubset,$ we have $a\sqsubset b.$ By (\ref{i}), $a$
$(\prec\cap\sqsubset)$ $b.$ Since $\prec\cap\sqsubset$ is an $\mathbf{H}%
$-relation, $a\vee c$ ($\prec\cap\sqsubset)$ $b\vee c$ for all $c\in Q,$ by
Definition \ref{D2.2}. Hence $a\vee c$ $\prec$ $b\vee c$ by (\ref{i}). As
$\ll$ is an $\mathbf{H}$-relation, we also have $a\vee c$ $\ll$ $b\vee c.$ So
$a\vee c$ ($\prec\cap\ll)$ $b\vee c$ by (\ref{i}). Thus $\prec\cap\ll$ is a
$\mathbf{H}$-relation (see Definition \ref{D2.2}).\bigskip
\end{proof}

For convenience, we summarize below some results about $\mathbf{H}$-relations
which will be used later.

\begin{theorem}
\label{T3.5}\emph{(i) }Let $\ll$ be an $\mathbf{H}$-relation and $a\in Q.$
Then $\ll^{\triangleright}$ is an $\mathbf{R}$-order and\smallskip

$1)$ $\mathfrak{r}(a)=\vee\lbrack\ll^{\triangleright},a]$ is the
$\ll^{\triangleright}$-radical in $[a,\mathbf{1}],$ i.e.$,$ $a\ll
^{\triangleright}\mathfrak{r}(a)$ $\overleftarrow{\ll}$ $\mathbf{1}%
;\smallskip$

$2)$ $\mathfrak{r}(b)=\mathfrak{r}(a)$ for all $b\in\lbrack a,\mathfrak{r}%
(a)]$ -- there is an ascending $\ll$-series from $b$ to $\mathfrak{r}%
(a);\smallskip$

$3)$ each $b\in\lbrack a,\mathbf{1}]\diagdown\lbrack\mathfrak{r}%
(a),\mathbf{1}]$ has a $\ll$-successor $s_{b}$\emph{: }$b\ll s_{b}\neq b;$
$\mathfrak{r}(a)$ has no $\ll$-successor.\smallskip

\emph{(ii) }Let $\ll$ be a dual $\mathbf{H}$-relation and $b\in Q.$ Then
$\ll^{\triangleleft}$ is a dual $\mathbf{R}$-order and\smallskip

$1)$\emph{ }$\mathfrak{p}(b)=\wedge\lbrack\ll^{\triangleleft},b]$ is the dual
$\ll^{\triangleleft}$-radical in $[\mathbf{0},b]\subseteq Q,$ i.e.$,$
$\mathbf{0}$ $\overrightarrow{\ll}$ $\mathfrak{p}(b)$ $\ll^{\triangleleft
}b;\smallskip$

$2)$\emph{ }$\mathfrak{p}(a)=\mathfrak{p}(b)$ for all $a\in\lbrack
\mathfrak{p}(b),b]$ -- there is a descending $\ll$-series from $a$ to
$\mathfrak{p}(b);\smallskip$

$3)$ each $a\in\lbrack\mathbf{0},b]\diagdown\lbrack\mathbf{0},\mathfrak{p}%
(b)]$ has a $\ll$-predecessor $p_{a}\ll a\neq p_{a};$ $\mathfrak{p}(b)$ has no
$\ll$-predecessor.
\end{theorem}

Let $Q$ be a complete lattice. For $a\in Q,$ set
\begin{equation}
\sigma(a)=\wedge\lbrack\ll,a]\text{ and }s(a)=\vee\lbrack a,\ll]. \label{1.5}%
\end{equation}
Then if $\ll$ is an $\mathbf{H}$-relation, we have
\begin{equation}
\lbrack a,\ll]\overset{(\ref{2,8})}{\subseteq}[a,\ll^{\triangleright}],\text{
so that }s(a)\leq\vee\lbrack a,\ll^{\triangleright}]\overset{(\ref{2.14}%
)}{=}\mathfrak{r}_{_{\ll^{\triangleright}}}(a). \label{2'}%
\end{equation}
If $\ll$ is a dual $\mathbf{H}$-relation, we have%
\begin{equation}
\lbrack\ll,a]\overset{(\ref{2,8})}{\subseteq}[\ll^{\triangleleft},a],\text{ so
that }\sigma(a)\geq\wedge\lbrack\ll^{\triangleleft},a]\overset{(\ref{2.13}%
)}{=}\mathfrak{p}_{_{\ll^{\triangleleft}}}(a). \label{2}%
\end{equation}

A bijection $\theta$: $Q\rightarrow Q^{\prime}$ is an isomorphism if $x\leq
y\Leftrightarrow\theta\left(  x\right)  \leq\theta\left(  y\right)  $ for all
$x,y\in$ $Q.$ Then%
\begin{equation}
\theta\left(  \vee G\right)  =\vee\theta\left(  G\right)  \text{ and }%
\theta\left(  \wedge G\right)  =\wedge\theta\left(  G\right)  \text{ for each
}G\subseteq Q. \label{5.1}%
\end{equation}
Indeed, $\theta\left(  x\right)  \leq\theta\left(  \vee G\right)  ,$ as
$x\leq\vee G$ for $x\in G.$ Thus $\vee\theta\left(  G\right)  \leq
\theta\left(  \vee G\right)  $. As $\theta^{-1}$ is also an isomorphism, $\vee
G=\vee\theta^{-1}\left(  \theta\left(  G\right)  \right)  \leq\theta
^{-1}\left(  \vee\theta\left(  G\right)  \right)  .$ So $\theta\left(  \vee
G\right)  \leq\vee\theta\left(  \vee G\right)  $. Hence $\theta\left(  \vee
G\right)  =\vee\theta\left(  G\right)  $. Similarly, $\theta\left(  \wedge
G\right)  =\wedge\theta\left(  G\right)  $.

Let $\theta$ be an automorphism of $Q$. It preserves a relation $\ll$ from
Ref($Q)$ if%
\begin{equation}
\theta(x)\ll\theta(y)\Leftrightarrow x\ll y\text{\emph{ }for all }x,y\in Q.
\label{5.0}%
\end{equation}
For $x<y$ in $Q,$ $\theta$ is an isomorphism from $[x,y]$ onto $[\theta
(x),\theta(y)].$ Clearly, $[x,y]$ is a lower (upper) $\ll$-set if and only if
$[\theta(x),\theta(y)]$ is a lower (upper) $\ll$-set$.$ So $x\ll^{\text{lo}}y$
if and only if $\theta(x)\ll^{\text{lo}}\theta(y),$ and $x\ll^{\text{up}}y$ if
and only if $\theta(x)\ll^{\text{up}}\theta(y).$ Thus $\theta$ preserves the
relations $\ll^{\text{lo}}$ and $\ll^{\text{up}}$.

\begin{proposition}
\label{T3.11}Let an automorphism $\theta$ of $Q$ preserve a relation $\ll$
from \emph{Ref(}$Q).$\smallskip

\emph{(i) \ }If $\ll$ is an $\mathbf{H}$-relation\emph{,} $\theta$ preserves
$\ll^{\triangleright};$ if $\ll$ is a dual $\mathbf{H}$-relation\emph{,}
$\theta$ preserves $\ll^{\triangleleft}.\smallskip$

\emph{(ii) }Let $\theta(a)=a$ for some $a\in Q.$ Then\emph{ }$\theta
(\sigma(a))=\sigma(a)$ and $\theta(s(a))=s(a).$ Moreover\emph{,\smallskip}

$\qquad1)$ if $\ll$ is a $\mathbf{T}$-order then $\theta(\mathfrak{r}%
(a))=\mathfrak{r}(a);$ if it is a dual $\mathbf{T}$-order then $\theta
(\mathfrak{p}(a))=\mathfrak{p}(a);\smallskip$

$\qquad2)$\emph{ }if $\ll$ is an $\mathbf{H}$-relation\emph{,} $\theta
(\mathfrak{r}_{_{\ll^{\triangleright}}}(a))=\mathfrak{r}_{_{\ll
^{\triangleright}}}(a);$ if it is a dual $\mathbf{H}$-relation\emph{,}
$\theta(\mathfrak{p}_{_{\ll^{\triangleleft}}}(a))=\mathfrak{p}_{_{\ll
^{\triangleleft}}}(a).$
\end{proposition}

\begin{proof}
(i) follows, as $\theta$ preserves $\ll^{\text{lo}}$ and $\ll^{\text{up}}$,
and as $\ll^{\text{lo}}=$ $\ll^{\triangleleft},$ $\ll^{\text{up}}=$
$\ll^{\triangleright}$ by Theorem \ref{inf}.

(ii) As $\theta$ and $\theta^{-1}$ are automorphisms$,$ $\theta([\ll
,a])=[\ll,a]$ and $\theta([a,\ll])=[a,\ll].$ By (\ref{5.1}),%
\begin{equation}
\theta(\sigma(a))=\theta(\wedge\lbrack\ll,a])=\wedge\lbrack\ll,a]=\sigma
(a)\text{ and }\theta(s(a))=\theta(\vee\lbrack a,\ll])=\vee\lbrack
a,\ll]=s(a). \label{5.8}%
\end{equation}

By (\ref{2.14}), $s(a)=\mathfrak{r}(a)$ if $\ll$ is a $\mathbf{T}$-order. By
(\ref{2.13}), $\sigma(a)=\mathfrak{p}(a)$ if $\ll$ is a dual $\mathbf{T}%
$-order$.$ So 1) follows from (\ref{5.8}). Part 2) follows from (i) and 1).
\end{proof}

\section{Gap $<_{\mathfrak{g}}$ and continuity $<_{\mathfrak{c}}$ relations in
complete lattices.}

\begin{definition}
\label{D7.1}\emph{Let }$a\leq b$ \emph{in }$Q.$ \emph{We write }%
$a<_{\mathfrak{g}}b$ \emph{if either} $a=b,$ \emph{or} $[a,b]$ \emph{is a
gap}.\smallskip

\emph{We write }$a<_{\mathfrak{c}}b$ \emph{if either} $a=b,$ \emph{or there is
a complete \textbf{continuous} chain }$C$ \emph{from} $a$ \emph{to} $b:$
\emph{For all }$x<y$\emph{ in }$C,$\emph{ there is }$z\in C$\emph{ such that
}$x<z<y$\emph{, i.e., }$C$ \emph{has no gaps.}
\end{definition}

The relations $<_{\mathfrak{g}}$ and $<_{\mathfrak{c}}$ belong to Ref($Q)$ and
$<_{\mathfrak{c}}$ is an order.

\begin{proposition}
\label{P6.1}A complete lattice $Q$ has no gaps if and only if $<_{\mathfrak{c}%
}$ $=$ $\leq$\emph{.}
\end{proposition}

\begin{proof}
Clearly, if $<_{\mathfrak{c}}$ $=$ $\leq$ then $Q$ has no gaps. Conversely,
let $a<b.$ By Lemma \ref{L2.1p}, there is a maximal complete chain $C$ from
$a$ to $b$. If $C$ is not continuous, there is a gap $\left[  x,y\right]
_{_{C}}$ in $C$. As $Q$ has no gaps, there is $z\in Q,$ $z\notin C,$ with
$x<z<y.$ Hence the chain $C\cup\{z\}$ is larger than $C,$ a contradiction.
Thus $C$ is continuous. So $a<_{\mathfrak{c}}b.$ Hence $\leq$ $\subseteq$
$<_{\mathfrak{c}}.$ As $<_{\mathfrak{c}}$ $\subseteq$ $\leq,$ we have
$<_{\mathfrak{c}}$ $=$ $\leq.$
\end{proof}

\begin{example}
\label{E6.3}In general\emph{, }the order $<_{\mathfrak{c}}$ is neither
contiguous\emph{, }nor expanded.\smallskip

\emph{(i) Let }$Q=[0,1]\cup\{a\},$\emph{ where }$[0,1]\subset\mathbb{R}%
$\emph{,} $\mathbf{0}=0,$ $\mathbf{1}=1,$ $a\wedge t=\mathbf{0}$ \emph{and}
$a\vee t=\mathbf{1}$ \emph{for }$t\in(0,1).$\emph{ Then }$Q$\emph{ is a
complete lattice and }$\mathbf{0}<_{\mathfrak{c}}\mathbf{1}.$ \emph{As
}$a\not <  _{\mathfrak{c}}\mathbf{1,}$ \emph{we have} $[\mathbf{0,1}%
]\nsubseteq\lbrack<_{\mathfrak{c}},\mathbf{1}].$ \emph{So }$<_{\mathfrak{c}}%
$\emph{ is} not contiguous\emph{.\smallskip}

\emph{(ii) Let }$Q_{n},$ $2\leq n,$ \emph{be the interval in }$\mathbb{R}^{2}%
$\emph{ from }$\left(  \frac{1}{n},\frac{1}{n}\right)  $\emph{ to }$(1,0).$
\emph{Let }$L=\cup_{n=2}^{\infty}Q_{n}$ \emph{and} $(x,y)\leq(u,v)$ \emph{in
}$L$\emph{ if they lie in the same }$Q_{n}$\emph{ and} $x\leq u.$ \emph{Then}
$\mathbf{1}=(1,0).$ \emph{Set }$Q=(0,0)\cup L$\emph{ and }$\mathbf{0}%
=(0,0)<(x,y)$\emph{ for all }$(x,y)\in L.$\emph{ Then }$Q$\emph{ is a complete
lattice, }$[<_{\mathfrak{c}},\mathbf{1}]=L,$ $\mathbf{0}=\wedge\lbrack
<_{\mathfrak{c}},\mathbf{1}]$ \emph{and} $\mathbf{0}\not <  _{\mathfrak{c}%
}\mathbf{1.}$ \emph{So }$<_{\mathfrak{c}}$ is not expanded.
\ \ \ $\blacksquare$
\end{example}

We need now the following lemma.

\begin{lemma}
\label{L6.1}Let $T$ be a complete chain from $a$ to $b.$ Let each $t\in
T\diagdown\{b\}$ have an immediate successor $t_{s}$ in $T,$ i.e.\emph{,}
$[t,t_{s}]_{_{T}}=\{t,t_{s}\}$ is a gap\emph{, }and let $C_{t}$ be some
complete chain from $t$ to $t_{s}.$ Set $C_{b}=\{b\}.$ Then the chain
$C=\cup_{t\in T}C_{t}$ is complete$.$
\end{lemma}

\begin{proof}
Let $G\subseteq C.$ For $t\in T,$ set $G_{t}=G\cap C_{t}$ and $Y=\{t\in T$:
$G_{t}\neq\varnothing\}.$

Let $y=\vee Y$ and $c_{t}=\vee G_{t}$ for $t\in Y.$ As $T$ and all $C_{t}$ are
complete, $y\in T$ and $c_{t}\in C_{t}.$ As $G=\cup_{t\in Y}G_{t},$ we have
from (\ref{2.5}) that $\vee G=\vee\{c_{t}$: $t\in Y\}.$ If $y\in Y$ then $\vee
G=c_{y}\in C_{y}\subseteq C.$ If $y\notin Y$ then $t<t_{s}\leq y$ for all
$t\in Y,$ as $[t,t_{s}]_{_{T}}$ is a gap. So $y=\vee\{t_{s}$: $t\in Y\}.$ As
$t\leq c_{t}\leq t_{s},$ we have $y=\vee Y\leq\vee G=\vee\{c_{t}$: $t\in
Y\}\leq\vee\{t_{s}$: $t\in Y\}=y.$ Hence $\vee G=y\in T\subseteq C.$ Thus $C$
is $\vee$-complete.

Let $z=\wedge Y$ and $g_{t}=\wedge G_{t}$ for $t\in Y.$ As $T$ and all $C_{t}$
are complete, $z\in T$ and $g_{t}\in C_{t}.$ As $G=\cup_{t\in Y}G_{t},$ we
have from (\ref{2.5}) that $\wedge G=\wedge\{g_{t}$: $t\in Y\}.$ If $z\notin
Y$ then $z<t$ for all $t\in Y$, and $z=\wedge Y.$ As $[z,z_{s}]_{_{T}}$ is a
gap, we get a contradiction. Hence $z\in Y.$ Then $G_{z}\neq\varnothing$ and
$\wedge G=\wedge G_{z}=g_{z}\in C_{z}\subseteq C.$ Thus $C$ is $\wedge
$-complete. So $C$ is complete.\bigskip
\end{proof}

Lemma \ref{L6.1} also holds if each $t\in T\diagdown\{a\}$ has an immediate
predecessor $t_{p}$ in $T$: $[t_{p},t]_{_{T}}$ is a gap$.$

\begin{proposition}
\label{pr7}\emph{(i) }$<_{\mathfrak{c}}$ $=$ $<_{\mathfrak{c}}^{\triangleleft
}$ $=$ $<_{\mathfrak{c}}^{\triangleright}.$

\emph{(ii) }Each complete continuous chain from $a$ to $b$ in $Q$ is maximal
in $\left[  a,b\right]  .$
\end{proposition}

\begin{proof}
(i) Let $a<_{\mathfrak{c}}^{\triangleright}b.$ Then there is a complete upper
$<_{\mathfrak{c}}$-gap chain $T$ from $a$ to $b$, i.e., each $t\in
T\diagdown\{b\}$ has an immediate $<_{\mathfrak{c}}$-successor $t_{s}$ in $T$,
$t<_{\mathfrak{c}}t_{s}.$ Then there is a complete continuous chain $C^{t}$
from $t$ to $t_{s}.$ By Lemma \ref{L6.1}, $C=\{b\}\cup(\cup_{_{t\in T}}C^{t})$
is a complete chain from $a$ to $b.$

If $C$ is not continuous, there are $x<y$ in $C$ such that $[x,y]_{_{C}}$ is a
gap. Let $t\in T$ be such that $x\in C^{t}\diagdown\{t_{s}\}.$ If $y\in C^{t}$
then $[x,y]_{_{C}}=[x,y]_{_{C^{t}}}$ is not a gap, as $C^{t}$ is continuous.
If $y\notin C^{t}$ then $y\neq t_{s}$ and $y\in C^{u},$ $t<u\in T.$ Hence
$x<t_{s}<y.$ So $[x,y]_{_{C}}$ is not a gap. Thus $C$ is continuous and
$a<_{\mathfrak{c}}b$. So $<_{\mathfrak{c}}^{\triangleright}\subseteq$
$<_{\mathfrak{c}}.$ As also $<_{\mathfrak{c}}\subseteq$ $<_{\mathfrak{c}%
}^{\triangleright}$ (see (\ref{2,8})), $<_{\mathfrak{c}}=$ $<_{\mathfrak{c}%
}^{\triangleright}.$ Similarly, $<_{\mathfrak{c}}^{\triangleleft}=$
$<_{\mathfrak{c}}.$

(ii) Let $C$ be a complete continuous chain from $a$ to $b.$ If $C$ is not
maximal, there is $u\in Q,$ $u\notin C,$ such that either $z<u$ or $u<z$ for
each $z\in C.$ Let $c=\vee\left\{  z\in C\text{: }z<u\right\}  $ and
$d=\wedge\left\{  z\in C\text{: }u<z\right\}  $. As $C$ is complete, $c,d\in
C$ and $c<u<d$. As $C$ is continuous, there is $x\in C$ such that $c<x<d.$
Hence either $x<u,$ or $u<x$, a contradiction. Thus $C$ is maximal.\bigskip
\end{proof}

For $a\in Q$, we defined in (\ref{5.3}) the maps $\curlyvee_{a}\left(
x\right)  =a\vee x$ and $\curlywedge_{a}\left(  x\right)  =a\wedge x$ for
$x\in Q.$ Then
\begin{equation}
\curlyvee_{a}\left(  x\right)  \leq\curlyvee_{a}\left(  y\right)  \text{ and
}\curlywedge_{a}\left(  x\right)  \leq\curlywedge_{a}\left(  y\right)  \text{
if }x\leq y. \label{6.7}%
\end{equation}
For each $G\subseteq Q$, its images $\curlyvee_{a}\left(  G\right)  $ and
$\curlywedge_{a}\left(  G\right)  $ satisfy (\ref{e.1}).

A complete lattice $\left(  Q,\leq\right)  $ is \textit{modular} if
\begin{equation}
\curlyvee_{a}\curlywedge_{b}=\curlywedge_{b}\curlyvee_{a}\text{ for all }a\leq
b\text{ in }Q.\label{6.8}%
\end{equation}
This is equivalent to the condition that $Q$ has no sublattice $P$
order-isomorphic to a pentagon:
\begin{equation}
P=\{a,b,c,d,e\},\text{ }a<b<c<d,\text{ }a<e<d,\text{ }b\vee e=d\text{ and
}c\wedge e=a.\label{6,0}%
\end{equation}
The lattice $Q$ in Example \ref{E6.3}(i) is not modular, as it contains
pentagons $\left\{  \mathbf{0},a,t,t^{\prime},\mathbf{1}\right\}  $ for
$\mathbf{0}<t<t^{\prime}<\mathbf{1}$.

\begin{proposition}
\label{P6.3}Let $C$ be a complete chain from $a$ to $b$ in $Q$ and $z\in
Q$.\smallskip

\emph{(i) }$\curlywedge_{z}\left(  C\right)  $ is a $\wedge$-complete chain
and $\curlyvee_{z}\left(  C\right)  $ is a $\vee$-complete chain$.\smallskip$

\emph{(ii) }Let $Q$ be a modular complete lattice.\smallskip

$\qquad1)$ If $\left[  x,y\right]  $ is a gap then $a)$ either $\curlywedge
_{z}\left(  x\right)  =\curlywedge_{z}\left(  y\right)  ,$ or $[\curlywedge
_{z}\left(  x\right)  ,\curlywedge_{z}\left(  y\right)  ]$ is a gap$;$

$\qquad\ \ \ \ b)$ either $\curlyvee_{z}\left(  x\right)  =\curlyvee
_{z}\left(  y\right)  ,$ or $[\curlyvee_{z}\left(  x\right)  ,\curlyvee
_{z}\left(  y\right)  ]$ is a gap$.$\smallskip

$\qquad2)$ If $C$ is continuous then the chains $\curlywedge_{z}\left(
C\right)  $ and $\curlyvee_{z}(C)$ are continuous$.$
\end{proposition}

\begin{proof}
(i) By (\ref{6.7}), $\curlywedge_{z}\left(  C\right)  $ is a chain from
$z\wedge a$ to $z\wedge b.$ Let $\Gamma=\curlywedge_{z}(G)$ for some
$G\subseteq C.$ As $\wedge G\in C,$ we have $\wedge\curlywedge_{z}\left(
G\right)  =\curlywedge_{z}(\wedge G)\in\curlywedge_{z}\left(  C\right)  $ by
(\ref{e.1}). Thus $\curlywedge_{z}\left(  C\right)  $ is $\wedge$-complete.
Similarly, $\curlyvee_{z}\left(  C\right)  $ is $\vee$-complete.

(ii) 1) a) Let $\left[  x,y\right]  $ be a gap and $\curlywedge_{z}\left(
x\right)  \neq\curlywedge_{z}\left(  y\right)  .$ Assume that $\curlywedge
_{z}\left(  x\right)  <u<\curlywedge_{z}\left(  y\right)  $ for some $u\in Q.$
Then $u<z$ and $u<y$. If $u\leq x,$ then $u\leq\curlywedge_{z}\left(
x\right)  $, a contradiction. Thus $x<u\vee x\leq y$. As $[x,y]$ is a gap,
$y=u\vee x$. As $u<z,$ it follows from (\ref{6.8}) that%
\[
\curlywedge_{z}\left(  y\right)  =\curlywedge_{z}\curlyvee_{u}\left(
x\right)  =\curlyvee_{u}\curlywedge_{z}\left(  x\right)  =u\vee\curlywedge
_{z}\left(  x\right)  =u<\curlywedge_{z}\left(  y\right)  ,
\]
a contradiction. Thus $\left[  \curlywedge_{z}\left(  x\right)  ,\curlywedge
_{z}\left(  y\right)  \right]  $ is a gap. Similarly, we can prove $b).$

2) Let $C$ be continuous$.$ If $\curlywedge_{z}\left(  C\right)  $ is not
continuous, $[\curlywedge_{z}(x),\curlywedge_{z}(y)]_{_{\curlywedge_{z}\left(
C\right)  }}$ is a gap for some $x<y$ in $C.$ Then, for all $\xi\in\lbrack
x,y]_{_{C}},$ either $\curlywedge_{z}\left(  \xi\right)  =\curlywedge
_{z}\left(  x\right)  ,$ or $\curlywedge_{z}\left(  \xi\right)  =\curlywedge
_{z}\left(  y\right)  $. Set $K=\{\xi\in\lbrack x,y]_{_{C}}$: $\curlywedge
_{z}\left(  \xi\right)  =\curlywedge_{z}\left(  x\right)  \},$ $u=\vee K,$
$L=\{\xi\in\lbrack x,y]_{_{C}}$: $\curlywedge_{z}\left(  \xi\right)
=\curlywedge_{z}\left(  y\right)  \}$ and $v=\wedge L.$ As $C$ is complete,
$u,v\in\lbrack x,y]_{_{C}}.$ As $C$ is continuous, $u=v.$

Let $u\in K,$ so that $\curlywedge_{z}\left(  u\right)  =\curlywedge
_{z}\left(  x\right)  $ and $u<\xi$ for all $\xi\in L.$ As $\curlywedge
_{z}(\xi)=\curlywedge_{\xi}(z),$%
\[
\curlyvee_{u}\curlywedge_{z}(y)=\curlyvee_{u}\curlywedge_{z}(\xi
)=\curlyvee_{u}\curlywedge_{\xi}(z)\overset{(\ref{6.8})}{=}\curlywedge_{\xi
}\curlyvee_{u}(z)=\xi\wedge\curlyvee_{u}(z)\leq\xi.
\]
Thus $\curlyvee_{u}\curlywedge_{z}(y)$ is a lower bound of $L.$ Hence
$u\vee\curlywedge_{z}(y)=\curlyvee_{u}\curlywedge_{z}(y)\leq u$. So
$\curlywedge_{z}(y)\leq u.$ As $\curlywedge_{z}(y)\leq z,$ we have
$\curlywedge_{z}(y)\leq z\wedge u=\curlywedge_{z}(u)=\curlywedge_{z}\left(
x\right)  ,$ a contradiction. Similarly, we get a contradiction, if we assume
that $u\in L.$ Thus $\curlywedge_{z}\left(  C\right)  $ is continuous.
Similarly, $\curlyvee_{z}\left(  C\right)  $ is continuous.
\end{proof}

\begin{corollary}
\label{C6.1}If $Q$ is modular then $<_{\mathfrak{g}}$\emph{ }is an
$\mathbf{HH}$-relation$.$
\end{corollary}

\begin{proof}
Let $Q$ by be modular. If $x$ $<_{\mathfrak{g}}$ $y,$ $x\neq y,$ then $[x,y]$
is a gap. By Proposition \ref{P6.3}(ii), for each $z\in Q,$ either $x\wedge
z=y\wedge z,$ or $[x\wedge z,y\wedge z]$ is a gap. So $x\wedge z$
$<_{\mathfrak{g}}$ $y\wedge z.$

Similarly, $x\vee z$ $<_{\mathfrak{g}}$ $y\vee z.$ By Definition \ref{D2.2},
$<_{\mathfrak{g}}$ is an $\mathbf{HH}$-relation$.\bigskip$
\end{proof}

For each $G\subseteq Q$ and any $z\in Q,$%
\begin{equation}
\vee\{z\wedge x:x\in G\}\leq z\wedge(\vee G)\text{ and }z\vee(\wedge
G)\leq\wedge\{z\vee x:x\in G\}. \label{6.0}%
\end{equation}
Consider the Join and Meet Infinite Distributive Identities (JID) and (MID)%
\begin{equation}
z\vee(\wedge G)=\wedge\{z\vee x\text{: }x\in G\}\text{ and }z\wedge(\vee
G)=\vee\{z\wedge x\text{: }x\in G\}\text{ for all }z\in Q\text{ and
}G\subseteq Q\text{.} \label{6.9}%
\end{equation}
For example, the lattice of \textit{all subsets} of a set $X$ with $\leq$ $=$
$\subseteq,$ $\wedge=\cap$ and $\vee=\cup$ has properties (\ref{6.9}), while
the lattice of \textit{all closed subsets} of a topological space $X$ does not
always have them.

If $Q$ is a chain then conditions (\ref{6.9}) hold. Note that a lattice $Q$
satisfies (\ref{6.9}) if and only if $Q$ has a complete embedding into a
complete boolean lattice (see \cite[Theorem 166]{G}).

\begin{definition}
\label{D6.1}\emph{A lattice }$Q$\emph{ has properties (JIDC) and (MIDC) if,
respectively, the first and second part of (\ref{6.9}) holds for each chain
}$G$ \emph{in} $Q.$
\end{definition}

The properties (JIDC) and (MIDC) are weaker than (JID) and (MID). In
particular, they do not imply that $Q$ is distributive.

\begin{lemma}
\label{L6.2}Let $Q$ have properties \emph{(JIDC)} and \emph{(MIDC)}$.$ If
$\ll$ is an $\mathbf{HH}$-relation then\smallskip

\emph{(i) \ }$\ll^{\triangleleft}$ is an $\mathbf{H}$-order and a dual
$\mathbf{R}$-order$;$ \ \emph{(ii) }$\ll^{\triangleright}$ is an $\mathbf{R}%
$-order and a dual $\mathbf{H}$-order\emph{.}
\end{lemma}

\begin{proof}
(i) By Proposition \ref{inf}, $\ll^{\triangleleft}$ is a dual $\mathbf{R}%
$-order. Let us prove that it is an $\mathbf{H}$-order.

Let $z\in Q.$ If $x\ll^{\triangleleft}y$ then there is a complete lower $\ll
$-gap chain $T$ from $y$ to $x$: each $t\in T\diagdown\{x\}$ has the immediate
$\ll$-predecessor $t_{p}$ in $T$: $[t_{p},t]_{T}$ is a gap. Fix $t\in
T\diagdown\{x\}$ and set $S_{t}=\{s\in T$: $z\vee s=z\vee t\}.$ As $T$ is
complete, $u^{t}:=\wedge S_{t}\in T$. By (\ref{6.9}),%
\begin{equation}
z\vee u^{t}=z\vee\left(  \wedge S_{t}\right)  =\wedge\{z\vee s\text{: }s\in
S_{t}\}=z\vee t,\text{ so that }u^{t}\in S_{t}. \label{6.10}%
\end{equation}
Let $u^{t}\neq x.$ Then $z\vee x\neq z\vee u^{t}=z\vee t,$ otherwise $u^{t}=x$
by (\ref{6.10}). The element $u^{t}$ has the immediate $\ll$-predecessor
$u_{p}^{t}$ in $T,$ i.e., $u_{p}^{t}\ll u^{t}.$ So $z\vee u_{p}^{t}\ll z\vee
u^{t}=z\vee t,$ as $\ll$ is an $\mathbf{H}$-relation. As $u_{p}^{t}\notin
S_{t},$ we have $z\vee u_{p}^{t}\neq z\vee t.$ So $z\vee u_{p}^{t}$ is the
immediate $\ll$-predecessor of $z\vee t=z\vee u^{t}$ in $\curlyvee_{z}\left(
T\right)  .$ Thus $\curlyvee_{z}\left(  T\right)  $ is a lower $\ll$-gap chain
from $z\vee y$ to $z\vee x.$

By Proposition \ref{P6.3}, $\curlyvee_{z}(T)$ is a $\vee$-complete$.$ Hence,
by Lemma \ref{L2.1p}(iii), it is complete. Thus $z\vee x\ll^{\triangleleft
}z\vee y,$ so that $\ll^{\triangleleft}$ is an $\mathbf{H}$-order. The proof
of part (ii) is similar.
\end{proof}

\begin{theorem}
\label{T6.2}If a modular lattice $Q$ has properties \emph{(JIDC)} and
\emph{(MIDC)} then $<_{\mathfrak{c}}$ is an $\mathbf{RR}$-order and $Q$ is a
union of disjoint sets each of which has no gaps.
\end{theorem}

\begin{proof}
Let $x<_{\mathfrak{c}}y$ and $C$ be a complete continuous chain from $x$ to
$y.$ Let $z\in Q.$ By Proposition \ref{P6.3}, $\curlywedge_{z}\left(
C\right)  $ is a $\wedge$-complete, continuous chain from $z\wedge x$ to
$z\wedge y$. Let $\Gamma=\curlywedge_{z}(G)\subseteq\curlywedge_{z}(C)$ for
some $G\subseteq C.$ Then $\vee G\in C,$ as $C$ is complete and, by
(\ref{6.9})$,$
\[
\vee\Gamma=\vee\curlywedge_{z}(G)=\vee\{z\wedge g\text{: }g\in G\}=z\wedge
(\vee G)\in\curlywedge_{z}\left(  C\right)  \text{.}%
\]
So $\curlywedge_{z}\left(  C\right)  $ is $\vee$-complete. Thus $\curlywedge
_{z}\left(  C\right)  $ is complete, so that $z\wedge x<_{\mathfrak{c}}z\wedge
y.$ Hence $\leq_{\mathfrak{c}}$ is a dual $\mathbf{H}$-order. Similarly,
$<_{\mathfrak{c}}$ is an $\mathbf{H}$-order. By Proposition \ref{pr7},
$<_{\mathfrak{c}}$ $=$ $<_{\mathfrak{c}}^{\triangleleft}$ $=$ $<_{\mathfrak{c}%
}^{\triangleright}.$ Thus, by Lemma \ref{L6.2}, $\leq_{\mathfrak{c}}$ is an
$\mathbf{RR}$-order.

By Theorem \ref{T2.12}, $Q=\cup_{\lambda\in Q_{\mathfrak{r}}}[\mathfrak{p}%
(\lambda),\lambda],$ all $[\mathfrak{p}(\lambda),\lambda],$ $\lambda\in
Q_{\mathfrak{r}},$ are mutually disjoint and $<_{\mathfrak{c}}$ $=$ $\leq$ in
each $[\mathfrak{p}(\lambda),\lambda].$ By Proposition \ref{P6.1},
$[\mathfrak{p}(\lambda),\lambda]$ has no gaps.\bigskip
\end{proof}

As complete chains are modular lattices and have properties (JID) and (MID),
we have

\begin{corollary}
\label{C6.2}Let $Q$ be a complete chain. Then $<_{\mathfrak{g}}$ is an
$\mathbf{HH}$-relation\emph{, }$<_{\mathfrak{g}}^{\triangleleft}$ is an
$\mathbf{H}$-order and a dual $\mathbf{R}$-order$,$ $<_{\mathfrak{g}%
}^{\triangleright}$ is an $\mathbf{R}$-order and a dual $\mathbf{H}%
$-order\emph{. }The relation $<_{\mathfrak{c}}$ is an $\mathbf{RR}$-order.
\end{corollary}

\section{Radical maps, enveloping  and inscribing sets}

In this section we study special subsets in complete lattices -- enveloping
and inscribing sets. In particular, we consider the lattices Ref($Q)$ of
relations and Map($Q)$ of all maps on a complete lattice $Q.$ We describe some
of their enveloping and inscribing sets and construct the corresponding
radical maps.

\subsection{Enveloping and inscribing sets in lattices, radical maps and
topology}

Recall that a subset $L$ is a \textit{sublattice} of ($Q,\leq)$ if $x\vee y,$
$x\wedge y\in L$ for all $x,y\in L.$ It is $\wedge$-\textit{complete}, if
$\wedge G\in L;$ and $\vee$-\textit{complete}, if $\vee G\in L$ for all
$G\subseteq L.$ It is \textit{complete }if it is $\wedge$- and $\vee$-complete.

On the other hand, $L$ can be a lattice in its own right with respect to the
restriction of $\leq$ to $L,$ not being a sublattice of $Q.$ For example, the
smallest element in $L$ majorizing $x,y\in L$ may differ from $x\vee y$.
Denote it $x\vee^{L}y$. Similarly, we use the notation $x\wedge^{L}y.$ A
lattice $L$ is $\vee^{L}$\textit{-complete} if each $G\subseteq L$ has the
smallest element $\vee^{L}G$ in $L$ majorizing all $x\in G.$ In general, $\vee
G\leq\vee^{L}G.$ Similarly, if $L$ is $\wedge^{L}$\textit{-complete} then
$\wedge^{L}G\leq\wedge G.$ If $L$ is $\vee^{L}$- and $\wedge^{L}$-complete, it
is a \textit{complete lattice.}

If $L$ is a \textit{sublattice} of $Q$ then $x\vee y=x\vee^{L}y$ and $x\wedge
y=x\wedge^{L}y$ for $x,y\in L.$ It is $\wedge$-\textit{complete }($\vee
$-\textit{complete}), if and only if $\wedge^{L}G=\wedge G$ ($\vee^{L}G=\vee
G)$ for all $G\subseteq L.$

For $a\in Q$ and $G\subseteq Q,$ write $G\leq a$ ($a\leq G),$ if $a$ majorizes
(minorizes) all $x\in G.$ Set
\begin{equation}
L^{^{G}}=\{a\in L\text{: }G\leq a\}\text{ and }L_{_{G}}=\{a\in L\text{: }a\leq
G\}. \label{3.12}%
\end{equation}

\begin{lemma}
\label{L3.5}\emph{(i) }Let $L$ be $\wedge^{^{L}}$-complete$.$ If $L$ has the
largest element\emph{,} it is also $\vee^{^{L}}$-complete\emph{, }i.e.\emph{,
}$L$ is a complete lattice\emph{;} and $\vee^{^{L}}G=\wedge^{^{L}}L^{^{G}}$ is
the smallest in $L^{^{G}}$ for all $G\subseteq L$.$\smallskip$

\emph{(ii) \ }Let $L$ be $\wedge$-complete and $\mathbf{1}\in L.$ For
$G\subseteq Q,$ $\wedge^{^{L}}L^{^{G}}=\wedge L^{^{G}}$ is the smallest
element in $L^{^{G}}.\smallskip$

\emph{(iii) }Let $L\subseteq Q$ be $\vee^{^{L}}$-complete. If $L$ has the
smallest element\emph{,} then $L$ is also $\wedge^{^{L}}$-complete\emph{,
}i.e.\emph{, }$L$ is a complete lattice\emph{,} and $\wedge^{^{L}}G=\vee
^{^{L}}L_{_{G}}$ is the largest in $L_{_{G}}$ for all $G\subseteq
L$.$\smallskip$

\emph{(iv) \ }Let $L$ be $\vee$-complete and $\mathbf{0}\in L.$ For
$G\subseteq Q,$ $\vee^{L}L_{_{G}}=\vee L_{_{G}}$ is the largest element in
$L_{_{G}}.$
\end{lemma}

\begin{proof}
Part (i) is proved in Lemma 34 \cite{G}. The proof of (iii) is similar.

(ii) The set $L^{G}$ contains $\mathbf{1}\mathfrak{.}$ As $L$ is $\wedge
$-complete, it follows that $h:=\wedge^{^{L}}L^{^{G}}=\wedge L^{^{G}}\in L$ is
the largest element in $Q$ minorizing $L^{^{G}}.$ As $g\leq x$ for all $g\in
G$ and $x\in L^{^{G}},\ $we have $g\leq h.$ Thus $h\in L^{^{G}}$ and is the
smallest element in $L^{^{G}}.$ The proof of part (iv) is similar.
\end{proof}

\begin{definition}
\label{D3.1}\emph{(cf.} Definition \emph{27 \cite{G}) }A set $L$ in $Q$ is
\textbf{enveloping}\textit{ }\emph{(Cld }in \emph{\cite{G}) }if\emph{,} for
each $x\in Q,$ the set $\{l\in L$\emph{:} $x\leq l\}$ has the smallest element
$\overline{x},$ so that $x\leq\overline{x}.$

It is \textbf{inscribing}\textit{ }if\emph{,} for each $x\in Q,$ the set
$\{l\in L$: $l\leq x\}$ has the largest element $\underline{x}.$
\end{definition}

The following theorem refines Corollary 29 \cite{G} and characterizes
enveloping and inscribing sets in terms of lattice operations.

\begin{theorem}
\label{C3.1}\emph{(i) }For $L\subseteq Q,$ the following conditions are
equivalent.\smallskip

$\qquad1)$ $\ L$ is enveloping$;$ \ $\ $

$\qquad2)$ $L$ is $\wedge$-complete and $\mathbf{1}\in L;$

$\qquad3)$ $\ L$ is a complete lattice$,$ $\wedge$-complete and $\mathbf{1}\in
L;$

$\qquad4)$ \ For each $G\subseteq Q,$ the set $L^{G}$ \emph{(}see
\emph{(\ref{3.12}))} has the smallest element $m_{_{G}}$ and
\begin{equation}
m_{_{G}}=\wedge^{L}L^{G}=\wedge L^{G}=\vee^{L}\{\overline{x}\text{\emph{: }%
}x\in G\};\text{ \ \ if }G\subseteq L\text{ then }m_{_{G}}=\vee^{L}G.
\label{3}%
\end{equation}

\emph{(ii) }For $N\subseteq Q,$ the following conditions are
equivalent.\smallskip

$\qquad1)$ $\ N$ is inscribing$;$

$\qquad2)$ $N$ is $\vee$-complete and $\mathbf{0}\in N;$

$\qquad3)$ $\ N$ is a complete lattice$,$ $\vee$-complete and $\mathbf{0}\in
N;$

$\qquad4)$ \ For each $G\subseteq Q,$ the set $N_{G}$ \emph{(}see
\emph{(\ref{3.12}))} has the largest element $n_{_{G}}$ and%
\[
n_{_{G}}=\vee^{N}N_{_{G}}=\vee N_{_{G}}=\wedge^{N}\{\underline{x}\text{\emph{:
}}x\in G\};\text{ \ if }G\subseteq N\text{ then }n_{_{G}}=\wedge^{N}G.
\]

\end{theorem}

\begin{proof}
(i) 1) $\Rightarrow$ 2) As $L$ is enveloping, $\overline{\mathbf{1}%
}=\mathbf{1}\in L.$ For $G\subseteq L$, let $t=\wedge G.$ As $t\leq g$ for all
$g\in G,$ and $\overline{t}$ is the smallest element in $L$ majorizing $t,$
$t\leq\overline{t}\leq g.$ So $t\leq\overline{t}\leq\wedge G=t.$ Thus
$t=\overline{t}\in L$ and $L$ is $\wedge$-complete.

2) $\Rightarrow$ 3) If $L$ is $\wedge$-complete and $\mathbf{1}\in L,$ it is a
complete lattice by Lemma \ref{L3.5}(i).

3) $\Rightarrow$ 4) Let $G\subseteq Q.$ By Lemma \ref{L3.5}(ii), $\wedge
^{L}L^{G}$ exists, $\wedge^{L}L^{G}=\wedge L^{G}$ and it is the smallest
element in $L^{G}.$ So $m_{_{G}}=\wedge^{L}L^{G}=\wedge L^{G}.$ If $G\subseteq
L$ then $m_{_{G}}=\vee^{L}G$ by Lemma \ref{L3.5}(i).

For each $x\in G,$ the set $\{l\in L$\emph{:} $x\leq l\}$ majorizing $x$ has
the smallest element $\overline{x}$. As $L$ is a complete lattice, it has the
smallest element $s=\vee^{L}\{\overline{x}$: $x\in G\}$ majorizing all
$\overline{x}.$ Thus $x\leq\overline{x}\leq s$ for all $x\in G.$ So $s\in
L^{G}.$

If $y\in L$ majorizes all $x\in G,$ then $\overline{x}\leq y$ for all $x\in G$
(by definition of $\overline{x}),$ i.e., $y$ majorizes all $\overline{x}.$
Therefore $s\leq y.$ Hence\thinspace$s$ is the smallest element in $L$
majorizing all $x\in G$, i.e., $s$ is the smallest element in $L^{G}.$ Thus
$s=m_{_{G}}=\wedge^{L}L^{G}=\wedge L^{G}=\vee^{L}\{\overline{x}$\emph{: }$x\in
G\}.$

4) $\Rightarrow$ 1) For each $x\in Q,$ take $G=\{x\}.$

The proof of (ii) is similar.\bigskip
\end{proof}

Let $\{L_{\lambda}\}_{\lambda\in\Lambda}$ be subsets of $Q.$ Consider the set
$L_{\Lambda}=\mathcal{\{}x=\{x_{\lambda}\}_{\lambda\in\Lambda}$: $x_{\lambda
}\in L_{\lambda}\}$ and let%
\begin{equation}
\widehat{L_{\Lambda}}\text{ }=\left\{  \wedge x:=\wedge_{_{\lambda\in\Lambda}%
}x_{\lambda}\text{ for\ }x\in L_{\Lambda}\right\}  \text{ and }\overset{\vee
}{L_{\Lambda}}\text{ }=\left\{  \vee x:=\vee_{_{\lambda\in\Lambda}}x_{\lambda
}\text{ for\ }x\in L_{\Lambda}\right\}  . \label{3.31}%
\end{equation}

\begin{corollary}
\label{C3.2}Let all $\{L_{\lambda}\}_{\lambda\in\Lambda}$ be enveloping and
$\{N_{\lambda}\}_{\lambda\in\Lambda}$ inscribing sets in $Q.$ Then\smallskip

\emph{(i) \ }The sets $\cap_{\lambda\in\Lambda}L_{\lambda},$
$\widehat{L_{\Lambda}}$ and $\cup_{\lambda\in\Lambda}L_{\lambda}$ are
enveloping in $Q.\smallskip$

\emph{(ii) }The sets $\cup_{\lambda\in\Lambda}N_{\lambda},$ $\overset{\vee
}{N_{\Lambda}}$ and $\cap_{\lambda\in\Lambda}N_{\lambda}$ are inscribing in
$Q$.
\end{corollary}

\begin{proof}
(i) All $L_{\lambda}$ are $\wedge$-complete by Theorem \ref{C3.1}. So
$\cap_{\lambda\in\Lambda}L_{\lambda}$ is $\wedge$-complete. As $\mathbf{1}\in
L_{\lambda}$ for all $\lambda\in\Lambda$, $\mathbf{1}\in\cap_{\lambda
\in\Lambda}L_{\lambda}.$ By Theorem \ref{C3.1}, $\cap_{\lambda\in\Lambda
}L_{\lambda}$ is enveloping.

Let $g\in G\subseteq$ $\overset{\wedge}{L_{\Lambda}}.$ By $(\ref{3.31}),$
$g=\wedge_{\lambda\in\Lambda}g_{\lambda}$ for some $\{g_{\lambda}%
\}_{\lambda\in\Lambda}\in L_{\Lambda},$ and%
\[
\wedge G=\wedge_{_{g\in G}}g=\wedge_{_{g\in G}}(\wedge_{_{\lambda\in\Lambda}%
}g_{\lambda}).
\]
For $\lambda\in\Lambda,$ set $y_{\lambda}=\wedge_{_{g\in G}}g_{\lambda}.$ As
$g\leq g_{\lambda},$ we have $\wedge_{_{g\in G}}g\leq\wedge_{_{g\in G}%
}g_{\lambda}=y_{\lambda}.$ So $\wedge G\leq\wedge_{_{\lambda\in\Lambda}%
}y_{\lambda}.$ On the other hand, as $y_{\lambda}\leq g_{\lambda}$ for each
$g\in G,$ we have $\wedge_{_{\lambda\in\Lambda}}y_{\lambda}\leq\wedge
_{_{\lambda\in\Lambda}}g_{\lambda}=g.$ Thus
\[
\wedge G\leq\wedge_{_{\lambda\in\Lambda}}y_{\lambda}\leq\wedge_{_{g\in G}%
}g=\wedge G,\text{ so that }\wedge G=\wedge_{_{\lambda\in\Lambda}}y_{\lambda
}.
\]
As all $L_{\lambda}$ are $\wedge$-complete by Theorem \ref{C3.1}, all
$y_{\lambda}\in L_{\lambda}.$ So $\{y_{\lambda}\}_{\lambda\in\Lambda}\in
L_{\Lambda}$ and $\wedge_{_{\lambda\in\Lambda}}y_{\lambda}\in$
$\widehat{L_{\Lambda}}$. Hence $\wedge G\in\widehat{L_{\Lambda}}.$ Thus
$\widehat{L_{\Lambda}}$ is $\wedge$-complete and $\mathbf{1}\in
\widehat{L_{\Lambda}}$ as $\mathbf{1}\in L_{\lambda}$ for all $\lambda.$ By
Theorem \ref{C3.1}, $\widehat{L_{\Lambda}}$ is enveloping.

Now let $G\subseteq\cup_{\lambda\in\Lambda}L_{\lambda}$. Set $G_{\lambda
}=G\cap L_{\lambda}$ for $\lambda\in\Lambda$. Then $\wedge G_{\lambda}\in
L_{\lambda}$ for all $\lambda\in\Lambda$. Hence
\[
\{\wedge G_{\lambda}\}_{\lambda\in\Lambda}\in L_{\Lambda}\text{ and }\wedge
G=\wedge(\cup_{_{\lambda\in\Lambda}}G_{\lambda})\overset{(\ref{2.5})}{=}%
\wedge_{_{\lambda\in\Lambda}}\left(  \wedge G_{\lambda}\right)
\overset{(\ref{3.31})}{\in}\overset{\wedge}{L_{\Lambda}}.
\]
So $\cup_{_{\lambda\in\Lambda}}L_{\lambda}$ is $\wedge$-complete. As it
contains $\mathbf{1}$, it is enveloping by Theorem \ref{C3.1}.

Part (ii) can be proved similarly.\bigskip
\end{proof}

A map $g$: $Q\rightarrow Q$ is called%
\begin{align}
\text{a }T\text{\textit{-radical map} if }x  &  \leq g\left(  x\right)
=g\left(  g\left(  x\right)  \right)  \text{ and }g\left(  x\vee y\right)
=g\left(  x\right)  \vee g\left(  y\right)  ,\label{t3.1}\\
\text{a dual }T\text{\textit{-radical map }if }g(g(x))  &  =g\left(  x\right)
\leq x\text{ and }g\left(  x\wedge y\right)  =g\left(  x\right)  \wedge
g\left(  y\right)  \text{ for }x,y\in Q.\nonumber
\end{align}
$T$-radical maps are radical (see (\ref{f3.1})). Indeed, if $x\leq y$ then
$g\left(  x\right)  \leq g\left(  x\right)  \vee g\left(  y\right)  =g\left(
x\vee y\right)  =g\left(  y\right)  .$ Similarly, dual $T$-radical maps are
dual radical maps.

For a map $g$: $Q\rightarrow Q,$ we denote by Fix($g)$ the set of all
$g$\textit{-}fixed points:%
\[
\text{Fix}\left(  g\right)  =\left\{  x\in Q\text{: }g\left(  x\right)
=x\right\}  .
\]
For an enveloping set $L$ and an inscribing set $N,$ define the maps $f_{_{L}%
}$ and $g_{_{N}}$ on $Q$ by%
\begin{equation}
f_{_{L}}:x\in Q\mapsto\overline{x}\in L,\text{ and }g_{_{N}}:x\in
Q\mapsto\underline{x}\in N. \label{en1}%
\end{equation}

\begin{proposition}
\label{P3.1}\emph{(i) }$\theta\emph{:}$ $L\mapsto f_{L}$ is a bijection from
the class of all enveloping sets in $Q$\emph{ }onto\emph{ }the class of all
radical maps on $Q.$ Its inverse\emph{ }$\theta^{-1}$\emph{: }$g\mapsto$
$\emph{Fix}\left(  g\right)  .$ Moreover$,$ $L=$ \emph{Fix}$(f_{L})$ and
$g=f_{\text{\emph{Fix}}(g)}.\smallskip$

\emph{(ii) }$\theta$ is a bijection from the class of all enveloping
sublattices of $Q$\emph{ }onto\emph{ }the class of all $T$-radical maps on
$Q.\smallskip$

\emph{(iii) }$\phi\emph{:}$ $N\mapsto g_{_{N}}$ is a bijection from the class
of all inscribing sets in $Q$\emph{ }onto\emph{ }the class of all dual radical
maps on $Q.$ Its inverse $\phi^{-1}$\emph{: }$f\mapsto$ $\emph{Fix}\left(
f\right)  .$ Moreover$,$ $N=$ \emph{Fix}$(g_{_{N}})$ and $f=g_{_{\emph{Fix}%
(f)}}.\smallskip$

\emph{(iv) }$\phi$ is a bijection from the class of all inscribing sublattices
of $Q$\emph{ }onto\emph{ }the class of all dual $T$-radical maps on $Q.$
\end{proposition}

\begin{proof}
Part (i) was proved in \cite[Lemma 28]{G}. The proof of (iii) is similar.

(ii) Let $L$ be an enveloping sublattice of $Q.$ Set $f=f_{L}.$ By (i), $f$ is
a radical map. Hence, for $x,y\in Q,$ $x\leq f(x)\leq f(x\vee y)$ and $y\leq
f(y)\leq f(x\vee y).$ So $x\vee y\leq f(x)\vee f(y)\leq f(x\vee y).$ Thus
$f(x\vee y)\leq f(f(x)\vee f(y))\leq f(f(x\vee y))=f(x\vee y),$ so that
$f(x\vee y)=f(f(x)\vee f(y)).$ As $L$ is a sublattice, $f(x)\vee f(y)\in L=$
Fix($f).$ Hence $f(x\vee y)=f(x)\vee f(y).$ Thus $f$ is a $T$-radical map.

Conversely, let $g$ be a $T$-radical map. By (i), Fix$(g)$ is an enveloping
set. By Theorem \ref{C3.1}, Fix$(g)$ is $\wedge$-complete. If $x,y\in$
Fix$(g)$ then $g(x\vee y)=g(x)\vee g(y)=x\vee y.$ So Fix$(g)$ is an enveloping
sublattice of $Q$ which completes the proof of (ii). Part (iv) can be proved
similarly.\bigskip
\end{proof}

Consider the following link between $T$-radical maps and topology. For a set
$X,$ the set P$\left(  X\right)  $ of all subsets of $X$ is a complete lattice
with order $\leq$ $=$ $\subseteq$, $\mathbf{0}=\varnothing$, $\mathbf{1}=X,$
$\wedge$ $=$ $\cap$ and $\vee$ $=$ $\cup$.

For a \textit{topological space }$\left(  X,\tau\right)  ,$ denote by
$\tau^{\text{op}}$ the $\vee$-complete sublattice\textbf{ }of all open subsets
in P($X)$ and by $\tau^{\text{cl}}$ the $\wedge$-complete sublattice\textbf{
}of all closed subsets in P($X).$

\begin{proposition}
\label{P2.18}\emph{(i) }Let $\left(  X,\tau\right)  $ be a topological space
and $Q=\mathrm{P}\left(  X\right)  .\smallskip$

$1)$ If $f:$ $Q\rightarrow Q$ maps each $G\in Q$ into its $\tau$-closure then
$f$ is a $T$-radical map$.\smallskip$

$2)$ If $g:$ $Q\rightarrow Q$ maps each $G\in Q$ into its $\tau$-interior then
$g$ is a dual $T$-radical map$.\smallskip$

\emph{(ii) }Let $X$ be a set and $Q=\mathrm{P}\left(  X\right)  $. Let
$f$\emph{: }$Q\rightarrow Q$ be a map.\smallskip

$1)$ If $f$ is a $T$-radical map and $f\left(  \varnothing\right)
=\varnothing,$ then $(X,\tau)$ with $\tau^{\mathrm{cl}}=\left\{  f\left(
G\right)  \text{\emph{: }}G\in Q\right\}  $ is a topological space.\smallskip

$2)$ If $f$ is a dual $T$-radical map and $f\left(  X\right)  =X,$ then
$(X,\tau)$ with $\tau^{\mathrm{op}}=\left\{  f\left(  G\right)  \text{\emph{:
}}G\in Q\right\}  $ is a topological space.
\end{proposition}

\begin{proof}
(i) 1) For $G\in Q,$ its $\tau$-closure $f\left(  G\right)  =\cap\left\{
F\in\tau^{\text{cl}}\text{: }G\subseteq F\right\}  \in\tau^{\text{cl}}$. Hence
$G\subseteq f\left(  G\right)  =f\left(  f\left(  G\right)  \right)  $. Let
$K\in Q$. As $\tau^{\mathrm{cl}}$ is a sublattice of $Q,$ $G\cup K\subseteq
f\left(  G\right)  \cup f\left(  K\right)  \in\tau^{\mathrm{cl}}$. Hence
$f(G\cup K)\subseteq f\left(  G\right)  \cup f\left(  K\right)  \subseteq
f(G\cup K).$ Thus $f\left(  G\cup K\right)  =f\left(  G\right)  \cup f\left(
K\right)  $, so that $f$ is a $T$-radical map. Part 2) follows from duality.

$\left(  \mathrm{ii}\right)  $ 1) Let $L=\mathrm{Fix}\left(  f\right)  =\{G\in
Q$: $f(G)=G\}.$ Then $L=\left\{  f\left(  G\right)  \text{: }G\in Q\right\}
$. By Proposition \ref{P3.1}, $L$ is an enveloping sublattice. By Theorem
\ref{C3.1}, it is $\cap$-complete. By (\ref{t3.1}), $f\left(  G\right)  \cup
f\left(  K\right)  =f\left(  G\cup K\right)  \in L$ for $G,K\in Q.$ Hence $L$
is a $\cap$-complete sublattice of $Q$. As $X\subseteq f(X)\subseteq X$ by
(\ref{t3.1}), $X=f\left(  X\right)  \in L$. As $f\left(  \varnothing\right)
=\varnothing$, $\varnothing\in L$. Then $L=\tau^{\mathrm{cl}}$ for some
topology $\tau$ in $X$. Part 2) follows from duality.
\end{proof}

\subsection{Enveloping and inscribing sets in Ref($Q)$}

For a complete lattice $(Q,\leq),$ Ref$\left(  Q\right)  $ is also a complete
lattice with the order $\subseteq$ $($see $(\ref{0}))$. For a subset
$F\subseteq$ Ref($Q),$ the relations $\wedge F=$ $\ll_{_{\wedge F}}%
:=\cap_{_{\ll\,\in F}}\ll$ and $\vee F=$ $\ll_{_{\vee F}}:=\cup_{_{\ll\,\in
F}}\ll$ are defined in (\ref{i}). Consider the following subsets of
Ref$\left(  Q\right)  $:\smallskip

\begin{center}
$L_{\mathbf{H}}=\left\{  \text{all }\mathbf{H}\text{-relations in Ref}\left(
Q\right)  \right\}  ,$ $\ L_{d\mathbf{H}}=\left\{  \text{all dual }%
\mathbf{H}\text{-relations in Ref}\left(  Q\right)  \right\}  ,\smallskip$

$L_{\text{uc}}=\left\{  \text{all up-contiguous }\ll\text{ in Ref}\left(
Q\right)  \right\}  ,$ $\ L_{\text{dc}}=\left\{  \text{all down-contiguous
}\ll\text{ in Ref}\left(  Q\right)  \right\}  ,\smallskip$

$L_{\text{ue}}=\left\{  \text{all up-expanded }\ll\text{ in Ref}\left(
Q\right)  \right\}  ,$ $\ L_{\text{de}}=\left\{  \text{all down-expanded }%
\ll\text{ in Ref}\left(  Q\right)  \right\}  ,\smallskip$

$L_{\triangleright}=\left\{  \ll\text{ in Ref}\left(  Q\right)  \text{: }%
\ll\text{ }=\text{ }\ll^{\triangleright}\right\}  ,$ $\ L_{\triangleleft
}=\left\{  \ll\text{ in Ref}\left(  Q\right)  \text{: }\ll\text{ }=\text{ }%
\ll^{\triangleleft}\right\}  .$
\end{center}

\begin{theorem}
\label{P3.12}\emph{(i) }The sets $L_{_{\mathbf{H}}},$ $L_{_{d\mathbf{H}}},$
$L_{\text{\emph{uc}}},$ $L_{\text{\emph{dc}}}$ are inscribing and enveloping
sublattices of $\emph{Ref}\left(  Q\right)  $.$\smallskip$

\emph{(ii)\ }The sets $L_{\triangleright},$ $L_{\triangleleft},$
$L_{\text{\emph{ue}}},$ $L_{\text{\emph{de}}}$ and the set $L_{\text{\emph{o}%
}}$ of all orders are enveloping in $\emph{Ref}\left(  Q\right)  $.
\end{theorem}

\begin{proof}
(i) Let $F\subseteq L_{_{\mathbf{H}}}.$ If $a\ll_{_{\wedge F}}b$ (see
(\ref{i})) then $a\ll b$ for all $\ll$ in $F$. As all $\ll$ are $\mathbf{H}%
$-relations, $a\vee x\ll b\vee x$ for all $x\in Q$. Hence $a\vee
x\ll_{_{\wedge F}}b\vee x.$ So $\wedge F$ is an $\mathbf{H}$-relation. Let now
$a\ll_{_{\vee F}}b$ (see (\ref{i})). Then $a\ll b$ for some $\ll$ in $F$, and
$a\vee x\ll b\vee x$ for all $x\in Q$. Hence $a\vee x\ll_{_{\vee F}}b\vee x.$
So $\vee F$ is an $\mathbf{H}$-relation. Thus $L_{_{\mathbf{H}}}$ is complete.
As $\mathbf{0,1}\in L_{_{\mathbf{H}}}$, $L_{_{\mathbf{H}}}$ is an inscribing
and enveloping lattice by Theorem \ref{C3.1}. The proof for $L_{_{d\mathbf{H}%
}}$ is similar.

Let $F\subseteq L_{\text{uc}}$. If $a\ll_{_{\wedge F}}b$ then $a\ll b$ for all
$\ll$ in $F$. So%
\begin{equation}
\lbrack a,\ll_{_{\wedge F}}]=\cap_{_{\ll\,\in F}}[a,\ll]\text{ and }\left[
\ll_{_{\wedge F}},b\right]  =\cap_{_{\ll\,\in F}}\left[  \ll,b\right]  .
\label{3,1}%
\end{equation}
As all $\ll$ are up-contiguous, $\left[  a,b\right]  \subseteq\left[
\ll,b\right]  $ for all $\ll$ in $F.$ By (\ref{3,1}), $\left[  a,b\right]
\subseteq\left[  \ll_{_{\wedge F}},b\right]  $. So $\wedge F$ is
up-contiguous. Let now $a\ll_{_{\vee F}}b.$ Then $a\ll b$ for some $\ll$ in
$F$. Hence $\left[  a,b\right]  \subseteq\left[  \ll,b\right]  .$ Thus
$\left[  a,b\right]  \subseteq\left[  \ll_{_{\vee F}},b\right]  $. So $\vee F$
is up-contiguous. Thus $L_{\text{uc}}$ is complete. As $\mathbf{1}$ and
$\mathbf{0}$ are up-contiguous, $L_{\text{uc}}$ is an inscribing and
enveloping lattice by Theorem \ref{C3.1}. The proof for $L_{\text{dc}}$ is similar.

(ii) For $F\subseteq L_{\triangleright},$ let $a\ll_{_{\wedge F}%
}^{\triangleright}b.$ By (\ref{2.18}), there is a complete, upper
$\ll_{_{\wedge F}}$-gap chain $C$ from $a$ to $b$, i.e., each $x\in
C\diagdown\{b\}$ has an immediate $\ll_{_{\wedge F}}$-successor $s_{x}$:
$x\ll_{_{\wedge F}}s_{x}$ and $[x,s_{x}]\cap C=\{x,s_{x}\}.$ By (\ref{i}),
$x\ll s_{x}$ for all $\ll$ in $F.$ So $s_{x}$ is an immediate $\ll$-successor
of $x$ in $C.$ Thus $C$ is a complete, upper $\ll$-gap chain for all $\ll$ in
$F$, i.e., $a\,\ll^{\triangleright}$ $b.$ As $\ll\,=\,\ll^{\triangleright},$
$a\ll b$ for all $\ll$ in $F$. Hence $a\ll_{_{\wedge F}}b$. So $\ll_{_{\wedge
F}}^{\triangleright}$ is stronger than $\ll_{_{\wedge F}}$. By (\ref{2,8}),
$\ll_{_{\wedge F}}$ is stronger than $\ll_{_{\wedge F}}^{\triangleright}.$
Hence $\ll_{_{\wedge F}}^{\triangleright}=$ $\ll_{_{\wedge F}}.$ Thus
$L_{\triangleright}$ is $\wedge$-complete. As $\mathbf{1}\in L_{\triangleright
}$, $L_{\triangleright}$ is enveloping by Theorem \ref{C3.1}. The proof for
$L_{\triangleleft}$ is similar.

For $F\subseteq L_{\text{ue}}$ and $a\in Q,$ let $G\subseteq\lbrack
a,\ll_{_{\wedge F}}].$ By (\ref{3,1}), $G\subseteq\lbrack a,\ll]$ for all
$\ll$ in $F.$ As all $\ll$ are up-expanded, $\vee G\in\lbrack a,\ll]$ for all
$\ll$ in $F.$ Hence $\vee G\in\lbrack a,\ll_{_{\wedge F}}].$ So $\ll_{_{\wedge
F}}$ is up-expanded. Thus $L_{\text{ue}}$ is $\wedge$-complete. As
$\mathbf{1}\in L_{\text{ue}}$, $L_{\text{ue}}$ is enveloping by Theorem
\ref{C3.1}. The proof for $L_{\text{de}}$ is similar.

For $F\subseteq L_{\text{o}}$, let $a\ll_{_{\wedge F}}b\ll_{_{\wedge F}}c.$
Then $a\ll b\ll c$ for all $\ll$ in $F$, whence $a\ll c$. So $a\ll_{_{\wedge
F}}c.$ Thus $\ll_{_{\wedge F}}$ is an order. Thus $L_{\text{o}}$ is $\wedge
$-complete. As $\mathbf{1}\in L_{\text{o}}$, $L_{\text{o}}$ is enveloping by
Theorem \ref{C3.1}.
\end{proof}

\begin{corollary}
\label{C3.14}The sets of all $\mathbf{T}$-orders\emph{,} of all dual
$\mathbf{T}$-orders$,$ of all $\mathbf{R}$-orders\emph{,} of all dual
$\mathbf{R}$-orders are enveloping in \emph{Ref}$\left(  Q\right)  .$
\end{corollary}

\begin{proof}
By Definition \ref{D3.2}, the set of all $\mathbf{T}$-orders is the
intersection of $L_{\text{o}}$,$L_{\text{uc}},L_{\text{ue}}$ which are
enveloping. So it is enveloping by Corollary \ref{C3.2}. The rest of the proof
is similar.\bigskip
\end{proof}

By Proposition \ref{P3.1}, there is a bijection between enveloping sets $L$
and radical maps $f_{L}$. It is, however, often difficult to determine the
action of $f_{L}$. Below we do it for some sets in Ref($Q).$

\begin{lemma}
\label{L3.1}The radical maps $f_{L_{\triangleright}},$ $f_{L_{\triangleleft}}$
on \emph{Ref}$(Q)$ act by $f_{L_{\triangleright}}$\emph{: }$\ll$ $\mapsto$
$\ll^{\triangleright}$ and $f_{L_{\triangleleft}}$\emph{: }$\ll$ $\mapsto$
$\ll^{\triangleleft}.$
\end{lemma}

\begin{proof}
Set $g$:\emph{ }$\ll$ $\mapsto$ $\ll^{\triangleright}.$ Then $g(g(\ll
))=(\ll^{\triangleright})^{\triangleright}\overset{(\ref{2.19})}{=}$
$\ll^{\triangleright}=g(\ll)\supseteq$ $\ll.$ It is easy to check that $\ll$
$\subseteq$ $\prec$ implies $g(\ll)=$ $\ll^{\triangleright}\subseteq$
$\prec^{\triangleright}=g(\prec),$ as each complete upper $\ll$-gap chain is
also a complete upper $\prec$-gap chain. Thus (\ref{f3.1}) holds. So $g$ is a
radical map from Ref($Q)$ onto $L_{\triangleright}.$ As
Fix$(g)=L_{\triangleright},$ we have from Proposition \ref{P3.1} that
$g=f_{\text{Fix}(g)}=f_{L_{\triangleright}}.$ The proof for
$f_{L_{\triangleleft}}$ is similar.\bigskip
\end{proof}

As $L_{\mathbf{H}}$ and $L_{d\mathbf{H}}$ are inscribing and enveloping
sublattices of Ref$\left(  Q\right)  $, the maps $f_{_{L_{\mathbf{H}}}}$ and
$f_{_{L_{d\mathbf{H}}}}$ are $T$-radical$,$ and $g_{_{L_{\mathbf{H}}}}$ and
$g_{_{L_{d\mathbf{H}}}}$ (see $(\ref{en1}))$ are dual $T$-radical. We describe
their action below.

\begin{proposition}
\emph{(i) }Set $\prec_{_{1}}=g_{_{L_{\mathbf{H}}}}(\ll).$ Then $a\prec_{_{1}%
}b$ if $a\vee x\ll b\vee x$ for all $x\in Q.$ \smallskip

\emph{(ii) }Set $\prec_{_{2}}=g_{_{L_{d\mathbf{H}}}}(\ll).$ Then $a\prec
_{_{2}}b$ if $a\wedge x\ll b\wedge x$ for all $x\in Q.\smallskip$

\emph{(iii) }Set $\prec_{_{3}}=f_{_{L_{\mathbf{H}}}}(\ll).$ Then $a\prec
_{_{3}}b$ if there are $x\ll y$ such that $x\leq a$ and $b=a\vee y.\smallskip$

\emph{(iv) }Set\emph{ }$\prec_{_{4}}=f_{_{L_{d\mathbf{H}}}}(\ll).$ Then
$a\prec_{_{4}}b$ if there are $x\ll y$ such that $a=b\wedge x\;$and $b\leq y.$
\end{proposition}

\begin{proof}
(i) By Lemma \ref{le1}, $\prec_{_{1}}$ is an $\mathbf{H}$-relation. Set
$x=\mathbf{0.}$ We get that $\prec_{_{1}}$ stronger than $\ll$: $\prec_{_{1}}$
$\subseteq$ $\ll$. If an $\mathbf{H}$-relation $\prec$ is stronger than $\ll,$
$a\prec b$ implies $a\vee x\prec b\vee x$ for all $x\in Q,$ so that $a\vee
x\ll b\vee x.$ Thus $a\prec_{_{1}}b.$ So $\prec$ is stronger than $\prec
_{_{1}}$: $\prec$ $\subseteq$ $\prec_{_{1}}.$ Hence $\prec_{_{1}}$ is the
largest $\mathbf{H}$-relation minorizing $\ll.$ By Definition \ref{D3.1} and
(\ref{en1}), $g_{_{L_{\mathbf{H}}}}(\ll)=$ $\prec_{_{1}}.$ The proof of (ii)
is similar.

(iii) Let $a\prec_{_{3}}b$. Then there are
\begin{equation}
x,y\in Q\text{ such that }x\ll y,\text{ }x\leq a\text{ and }b=a\vee y.
\label{6.2}%
\end{equation}
For $z\in Q,$ $x\leq a\vee z$ and $b\vee z=(a\vee y)\vee z=(a\vee z)\vee y.$
Hence $(a\vee z)\prec_{_{3}}(b\vee z).$ Thus $\prec_{_{3}}$ is an $\mathbf{H}%
$-relation by Lemma \ref{le1}. It majorizes $\ll$ ($\ll$ $\subseteq$
$\prec_{_{3}}$), as $a\ll b$ implies $a\prec_{_{3}}b$ (set $x=a,$ $y=b).$

Let an $\mathbf{H}$-relation $\prec$ majorize $\ll$ ($\ll$ $\subseteq$
$\prec).$ If $a\prec_{_{3}}b$ then (\ref{6.2}) holds. As $x\ll y$ and $\ll$
$\subseteq$ $\prec,$ we have $x\prec y.$ As $\prec$ is an $\mathbf{H}%
$-relation, it follows from (\ref{6.2}) that $a=(a\vee x)\prec(a\vee y)=b.$
Thus $\prec_{_{3}}$ is stronger than $\prec$: $\prec_{_{3}}$ $\subseteq$
$\prec.$ Hence $\prec_{_{3}}$ is the smallest $\mathbf{H}$-relation majorizing
$\ll,$ i.e., $\prec_{_{3}}$ $=f_{L_{\mathbf{H}}}\left(  \ll\right)  $. The
proof of (iv) is similar.\bigskip
\end{proof}

For each $\ll$ in Ref$(Q),$ define the relations $\ll_{\text{uc}}$ and
$\ll_{\text{dc}}$ as follows\emph{:}%
\begin{align}
a  &  \ll_{\text{uc}}b\text{ if and only if there is }c\in Q\text{ satisfying
}c\ll b\text{ and }a\in\lbrack c,b];\label{6.3}\\
a  &  \ll_{\text{dc}}b\text{ if and only if there is }c\in Q\text{ such that
}a\ll c\text{ and }b\in\lbrack a,c].\nonumber
\end{align}

\begin{proposition}
\label{P3.2}For each $\ll$ in \emph{Ref}$(Q),$ $\ll_{\text{\emph{uc}}}$
$=f_{_{L_{\text{\emph{uc}}}}}(\ll)$ and $\ll_{\text{\emph{dc}}}$
$=f_{_{L_{\text{\emph{dc}}}}}(\ll).$
\end{proposition}

\begin{proof}
If $a\ll_{\text{uc}}b$ then $x\in\lbrack c,b]$ for $x\in\lbrack a,b]$. Then,
by (\ref{6.3}), $x\ll_{\text{uc}}b.$ Hence $[a,b]\subseteq\lbrack
\ll_{\text{uc}},b].$ So (see Definition \ref{D1.1}) $\ll_{\text{uc}}$ is up-contiguous.

If $a\ll b$ then $a\ll_{\text{uc}}b$ (set $c=a$ in (\ref{6.3})). So
$\ll_{\text{uc}}$ majorizes $\ll$ ($\ll$ $\subseteq$ $\ll_{\text{uc}}$ see
(\ref{0})).

Let an up-contiguous relation $\prec$ majorize $\ll$ ($\ll$ $\subseteq$
$\prec).$ If $a\ll_{\text{uc}}b$ then $c\ll b$ and $a\in\lbrack c,b]$ for some
$c\in Q.$ Hence $c\prec b.$ As $\prec$ is up-contiguous, $[c,b]\subseteq
\lbrack\prec,b],$ so that $a\prec b.$ Thus $\ll_{\text{uc}}$ is stronger than
$\prec.$ So $\ll_{\text{uc}}$ is the smallest element of $L_{\text{uc}}$
majorizing $\ll,$ i.e., $\ll_{\text{uc}}=f_{_{L_{\text{uc}}}}(\ll).$\ The
equality $f_{_{L_{\text{dc}}}}(\ll)=$ $\ll_{\text{dc}}$ is proved similarly.
\end{proof}

\subsection{Enveloping and inscribing sets of maps on lattices}

For a complete lattice $\left(  Q,\leq\right)  $, denote by Map$\left(
Q\right)  $ the set of all maps $Q\rightarrow Q$. For $f,g\in$ Map$\left(
Q\right)  ,$%
\begin{equation}
g\lesssim f\text{ if }g\left(  x\right)  \leq f\left(  x\right)  \text{ for
all }x\in Q;\text{ \ and set }\mathbf{0}_{\text{Map}}\text{: }x\mapsto
\mathbf{0}\text{ and }\mathbf{1}_{\text{Map}}\text{: }x\mapsto\mathbf{1}.
\label{3.k}%
\end{equation}
Then $\left(  \text{Map}\left(  Q\right)  ,\lesssim\right)  $ is a complete
lattice and, for each $G\subseteq$ Map$\left(  Q\right)  $,%
\begin{equation}
\left(  \vee G\right)  \left(  x\right)  =\vee_{_{f\in G}}f\left(  x\right)
\text{ and }\left(  \wedge G\right)  \left(  x\right)  =\wedge_{_{f\in G}%
}f\left(  x\right)  . \label{3.22}%
\end{equation}

Let Rad$_{_{Q}}$ and dRad$_{_{Q}}$ be the sets of all radical and dual radical
maps, respectively, on $Q$ (we write Rad and dRad$).$ They are partially
ordered with respect to the order $\lesssim$,%
\begin{align*}
\mathbf{0}_{\text{Rad}}  &  \text{: }x\mapsto x\text{ and }\mathbf{1}%
_{\text{Rad}}=\mathbf{1}_{\text{Map}}\text{: }x\mapsto\mathbf{1}\text{ for all
}x\in Q;\\
\mathbf{0}_{\text{dRad}}  &  =\mathbf{0}_{\text{Map}}\text{: }x\mapsto
\mathbf{0}\text{ and }\mathbf{1}_{\text{dRad}}=\mathbf{0}_{\text{Rad}}\text{:
}x\mapsto x\text{ for all }x\in Q.
\end{align*}

\begin{lemma}
\label{L3.14}\emph{(i) }Let $f$ and $g$ be radical maps in $Q.$ The following
conditions are equivalent$:\smallskip$

$1)$ $g\lesssim f;$ \ \ \ $2)$\emph{ }$f\left(  g\left(  x\right)  \right)
=f\left(  x\right)  $ for all $x\in Q;$ \ \ \ $3)$\emph{ }$\mathrm{Fix}\left(
f\right)  \subseteq\mathrm{Fix}\left(  g\right)  $.\smallskip

\emph{(ii) }For $G\subseteq$ \emph{Rad}$,$ the set $\cap_{g\in G}%
\emph{Fix}\left(  g\right)  $ is enveloping.
\end{lemma}

\begin{proof}
(i) $1)$\textrm{ }$\Rightarrow$ $2)$ Let $x\in Q$. By (\ref{f3.1}), $x\leq
g\left(  x\right)  \leq f\left(  x\right)  $, so that $f\left(  x\right)  \leq
f\left(  g\left(  x\right)  \right)  \leq f\left(  f\left(  x\right)  \right)
=f\left(  x\right)  $. Thus $f\left(  g\left(  x\right)  \right)  =f\left(
x\right)  .$

$2)$ $\Rightarrow$ $3)$ If $x\in$ Fix($f)$ then, by (\ref{f3.1}), $x\leq
g\left(  x\right)  \leq f\left(  g\left(  x\right)  \right)  =f\left(
x\right)  =x.$ So $g\left(  x\right)  =x$ and Fix$\left(  f\right)  \subseteq$
Fix$\left(  g\right)  $.

$3)$ $\Rightarrow$ $1)$ For all $x\in Q$, $f\left(  x\right)  \in
\,$Fix$\left(  f\right)  \subseteq$ Fix$\left(  g\right)  .$ So, by
(\ref{f3.1}), $g\left(  x\right)  \leq g\left(  f\left(  x\right)  \right)
=f\left(  x\right)  $.

(ii) By Proposition \ref{P3.1}, Fix($g)$ are enveloping sets for $g\in G.$ So
$\cap_{g\in G}$Fix$\left(  g\right)  $ is enveloping by Corollary \ref{C3.2}.
\end{proof}

\begin{theorem}
\label{T3.9}\emph{(i) Rad} is an enveloping set in \emph{Map}$(Q)$.\smallskip

\emph{(ii) }For $G\subseteq$ \emph{Rad}$,$ set $\mathfrak{h}=\vee
^{\text{\emph{Rad}}}G$ and $\mathfrak{g}=\wedge G.$ Then $\mathfrak{h,g}\in$
\emph{Rad}$,$ $\emph{Fix}\left(  \mathfrak{h}\right)  =\cap_{g\in G}%
\emph{Fix}\left(  g\right)  ,$%
\begin{equation}
\emph{Fix}\left(  \mathfrak{g}\right)  =\left(  \cup_{g\in G}\emph{Fix}\left(
g\right)  \right)  \cup\widehat{K_{_{G}}},\text{ where }K_{_{G}}%
=\{\{x_{g}\}_{g\in G}\emph{:\ }x_{g}\in\emph{Fix(}g)\}\text{ \emph{(}see
}\emph{(\ref{3.31}))}. \label{6.6}%
\end{equation}
The set$\ \emph{Fix}\left(  \mathfrak{g}\right)  $ is the smallest enveloping
set in $Q$ containing $\cup_{g\in G}$\emph{Fix}$\left(  g\right)  $.
\end{theorem}

\begin{proof}
(i) For $G\subseteq$ Rad$,$ set $\gamma=\wedge G\in$ Map($Q).$ As $x\leq f(x)$
for all $x\in Q,$ $f\in G$ by (\ref{f3.1}), we have $x\leq\gamma(x)\leq f(x)$
by (\ref{3.22})$.$ From this and (\ref{f3.1}) it follows that for all $x\in
Q,$
\[
\gamma(x)\leq\gamma(\gamma(x))\leq f(\gamma\left(  x\right)  )\leq
f(f(x))=f(x)\text{ for all }f\in G.
\]
Hence $\gamma(x)\leq\gamma(\gamma(x))\leq\wedge_{_{f\in G}}f\left(  x\right)
=\gamma(x).$ Thus $x\leq\gamma(x)=\gamma(\gamma(x))$ for all $x\in Q.$

Let $x\leq y$. Then $f(x)\leq f(y)$ for all $f\in G.$ Hence $\gamma\left(
x\right)  =\wedge_{_{f\in G}}f\left(  x\right)  \leq\wedge_{_{f\in G}}f\left(
y\right)  =\gamma\left(  y\right)  $. Thus $\gamma\in$ Rad. So Rad is $\wedge
$-complete. As $\mathbf{1}_{\text{Map}}\in$ Rad$,$ Rad is an enveloping set by
Theorem \ref{C3.1}.

(ii) As Rad\emph{ }is enveloping, it is complete\emph{ }and $\wedge$-complete
by Theorem \ref{C3.1}. Thus $\mathfrak{h,g}\in$ Rad$.$

By Lemma \ref{L3.14}, $L:=\cap_{g\in G}$Fix$\left(  g\right)  $ is enveloping.
By Proposition \ref{P3.1}, $L=$ Fix$(f_{L})$ for some radical map $f_{L}.$ As
Fix$(f_{L})\subseteq$ Fix$(g)$ for all $g\in G,$ $g\lesssim f_{L}$ by Lemma
\ref{L3.14}. As $\mathfrak{h}$ is the smallest radical map majorizing all
$g\in G,$ $g\lesssim\mathfrak{h\lesssim}$ $f_{L}.$ So, by Lemma \ref{L3.14},
$L=$ Fix$(f_{L})\subseteq$ Fix$\left(  \mathfrak{h}\right)  \subseteq$
$\cap_{_{g\in G}}$Fix$\left(  g\right)  =L.$ Thus Fix($\mathfrak{h}%
)=L=\cap_{_{g\in G}}$Fix$\left(  g\right)  .$

As $\mathfrak{g\in}$ Rad$,$ it is the largest radical map minorizing all $g\in
G.$ Hence, by Proposition \ref{P3.1} and Lemma \ref{L3.14}, Fix$\left(
\mathfrak{g}\right)  $ is an enveloping set in $Q$ and $F:=\cup_{_{g\in G}}%
$Fix$\left(  g\right)  \subseteq$ Fix$\left(  \mathfrak{g}\right)  .$ Let $S$
be an enveloping set such that $F\subseteq S$. Then $S=$ Fix$(f_{S})$ for a
radical map $f_{S},$ and $f_{S}\lesssim g$ for $g\in G.$ Hence $f_{S}%
\lesssim\mathfrak{g},$ so that Fix$\left(  \mathfrak{g}\right)  \subseteq S.$
Thus Fix$\left(  \mathfrak{g}\right)  $ is the smallest enveloping set
containing $F$.

Let $R=F\cup\widehat{K_{_{G}}}$ (see (\ref{3.31}) and (\ref{6.6})). By
Corollary \ref{C3.2}, $R$ is enveloping. Hence, as above, $F\subseteq$
Fix$\left(  \mathfrak{g}\right)  \subseteq R$. For $x=\{x_{g}\}\in K_{_{G}},$
$\wedge x=\wedge_{_{g\in G}}x_{g}\in\widehat{K_{_{G}}}$ and $\wedge x\leq
x_{g}$ for all $g\in G$. Hence $\mathfrak{g}\left(  \wedge x\right)
\leq\mathfrak{g}\left(  x_{g}\right)  =x_{g}$, as Fix($g)\subseteq$
Fix$\left(  \mathfrak{g}\right)  .$ So $\mathfrak{g}\left(  \wedge x\right)
\leq\wedge_{_{g\in G}}x_{g}=\wedge x$. As $\mathfrak{g}$ is radical
$\mathfrak{g}\left(  \wedge x\right)  =\wedge x$. Thus $\wedge x\in$
Fix$\left(  \mathfrak{g}\right)  ,$ so that $\widehat{K_{_{G}}}\subseteq$
Fix$\left(  \mathfrak{g}\right)  .$ Hence $R=F\cup\widehat{K_{_{G}}}\subseteq$
Fix$\left(  \mathfrak{g}\right)  \subseteq R.$ So $R=$ Fix$\left(
\mathfrak{g}\right)  $.\bigskip
\end{proof}

By duality, similarly to Theorem \ref{T3.9}, we get

\begin{theorem}
\label{T3.9'}\emph{(i) dRad} is an inscribing set in \emph{Map(}%
$Q)$.\smallskip

\emph{(ii) }For each $G\subseteq$ \emph{dRad}$,$ we have $\emph{Fix}\left(
\wedge^{\text{\emph{Rad}}}G\right)  =\cup_{g\in G}\emph{Fix}\left(  g\right)
$ and%
\[
\emph{Fix}\left(  \vee G\right)  =\left(  \cap_{g\in G}\emph{Fix}\left(
g\right)  \right)  \cap\overset{\vee}{K_{_{G}}},\text{ where }K_{_{G}%
}=\{\{x_{g}\}_{g\in G}\emph{:\ }x_{g}\in\emph{Fix(}g)\}
\]
$($see $(\ref{3.31})).$ The set $\emph{Fix}\left(  \vee G\right)  $ is the
largest inscribing set in $Q$ contained in $\cap_{g\in G}$\emph{Fix}$\left(
g\right)  .$
\end{theorem}

The subset $\mathcal{T}$ of Rad of all $T$-radical maps on $Q$ is partially
ordered with respect to $\lesssim.$ It has the smallest element $\mathbf{0}%
_{\mathcal{T}}=\mathbf{0}_{\text{Rad}}$: $x\mapsto x,$ and the largest element
$\mathbf{1}_{\mathcal{T}}=\mathbf{1}_{\text{Map}}$: $x\mapsto\mathbf{1}$.

\begin{theorem}
\label{T3.10}\emph{(i) }$\mathcal{T}$ is a complete lattice$\emph{,}$
i.e.\emph{, }$\vee^{\mathcal{T}}G$ and $\wedge^{\mathcal{T}}G$ exist for each
$G\subseteq\mathcal{T}.$ Moreover$,\smallskip$

\qquad\emph{(a) }$\vee^{\mathcal{T}}G=\vee^{\emph{Rad}}G=\wedge^{\mathcal{T}%
}\mathcal{T}^{^{G}}$ \emph{(}see \emph{(\ref{3.12})) }is the smallest element
in $\mathcal{T}^{^{G}};\smallskip$

\qquad\emph{(b) }$\wedge^{\mathcal{T}}G=\vee^{\mathcal{T}}\mathcal{T}_{_{G}}$
\emph{(}see \emph{(\ref{3.12})) }is the largest element in $\mathcal{T}_{_{G}%
};\smallskip$

\qquad\emph{(c) }$\mathrm{Fix}\left(  \wedge^{\mathcal{T}}G\right)  $ is the
smallest enveloping sublattice of $Q$ containing $\cup_{g\in G}\mathrm{Fix}%
\left(  g\right)  $.\smallskip

\emph{(ii) $\mathcal{T}$} is an inscribing sublattice of \emph{Rad}$_{_{Q}}.$
\end{theorem}

\begin{proof}
(i) Let $G\subseteq\mathcal{T}$. By Theorem \ref{T3.9}, there is
$\mathfrak{h}=\vee^{\text{Rad}}G$ in Rad$_{_{Q}}$ and $L:=$ Fix($\mathfrak{h}%
)=\cap_{_{g\in G}}$Fix$\left(  g\right)  $ is enveloping in $Q$. As each $g\in
G$ is $T$-radical, Fix($g)$ is a sublattice of $Q$. Hence $L$ is an enveloping
sublattice of $Q.$ By Proposition \ref{P3.1}, $\mathfrak{h}$ a $T$-radical
map. As $\mathfrak{h}$ is the smallest radical map majorizing all $g\in G$, it
is also the smallest $T$-radical map majorizing all $g\in G$, i.e.,
$\mathfrak{h}=\vee^{^{\mathcal{T}}}G$. Thus $\mathcal{T}$ is $\vee
^{^{\mathcal{T}}}$-complete and $\vee^{^{\mathcal{T}}}G=\vee^{\text{Rad}}G.$

As $\mathbf{0}_{_{\mathcal{T}}}$ is the smallest element in $\mathcal{T},$ it
follows from Lemma \ref{L3.5}(iii) that $\mathcal{T}$ is a complete lattice
and $\mathfrak{n}$: $=\wedge^{\mathcal{T}}G=\vee^{^{\mathcal{T}}}%
\mathcal{T}_{_{G}}$ is the largest element in $\mathcal{T}_{_{G}}$ (see
(\ref{3.12})). As $\mathcal{T}$ is a complete lattice, Lemma \ref{L3.5}(i)
implies that $\vee^{\mathcal{T}}G=\wedge^{\mathcal{T}}\mathcal{T}^{^{G}}$ is
the smallest in $\mathcal{T}^{^{G}}.$ This completes the proof of (a) and (b).

(c) As $\mathfrak{n}$ is $T$-radical, Fix$\left(  \mathfrak{n}\right)  $ is an
enveloping sublattice of $Q$ by Proposition \ref{P3.1}. As $\mathfrak{n}\leq
g$ for all $g\in G$, Fix$\left(  \mathfrak{n}\right)  $ contains $\cup_{g\in
G}$Fix$\left(  g\right)  $ by Lemma \ref{L3.14}. If $\cup_{_{g\in G}}%
$Fix$\left(  g\right)  \subseteq L$ for some enveloping sublattice $L,$ then
$L=$ Fix$(f_{L})$ for a $T$-radical map $f_{L}.$ By Lemma \ref{L3.14},
$f_{L}\leq g$ for $g\in G.$ Hence $f_{L}\leq\mathfrak{n},$ as $\mathfrak{n}$
is the largest $T$-radical map minorizing all $g\in G.$ By Lemma \ref{L3.14},
Fix$\left(  \mathfrak{n}\right)  \subseteq L.$

(ii) By Theorem \ref{T3.9}, Rad$_{_{Q}}$ is a complete lattice$.$ As
$\vee^{\mathcal{T}}G=\vee^{\text{Rad}}G$ for all $G\subseteq\mathcal{T},$
$\mathcal{T}$ is a $\vee^{\text{Rad}}$-complete sublattice of Rad$.$ As
$\mathbf{0}_{_{\mathcal{T}}}\in\mathcal{T},$ $\mathcal{T}$ is inscribing in
Rad by Theorem \ref{C3.1}.
\end{proof}

\subsection{Transfinite extensions of relations}

For $f,g\in$ Map($Q),$ their superposition $g\circ f$ is defined by $(g\circ
f)(x)=g(f(x)).$ If $g$ is a radical map then $f\lesssim g\circ f$. If $f,g$
are radical then $g\circ f$ is radical. Consider the following process of
transfinite superposition that has wide applications in the theory of radicals
of rings and algebras.

Let $G\subseteq$ Rad$.$ For an interval $\left[  1,\gamma\right]  $ of
ordinals, the set $\left(  h_{\alpha}\right)  _{1\leq\alpha\leq\gamma}$ in Rad
is an \textit{ascending} \textit{superposition }$G$-\textit{series }if%
\begin{align}
h_{1}  &  \in G,\text{ }h_{\alpha+1}=g^{\alpha}\circ h_{\alpha}\text{ for each
}\alpha\in\lbrack1,\gamma)\text{ and some }g^{\alpha}\in G,\nonumber\\
h_{\beta}  &  =\vee_{_{\alpha<\beta}}h_{\alpha}\text{ for each limit ordinal
}\beta\text{ (cf. (\ref{2.12})).}\label{3.33}\\
\text{If }g\circ h_{\gamma}  &  =h_{\gamma}\text{ for all }g\in G,\text{ then
the series is \textit{maximal.}} \label{8.1}%
\end{align}
Let $\mathbf{0}_{\text{Rad}}\notin G.$ As all $g\in G$ are radical maps,
$h_{\alpha+1}$ is larger than $h_{a}$ for all $\alpha.$ So the cardinality of
$\left(  h_{\alpha}\right)  _{1\leq\alpha\leq\gamma}$ is not larger than the
cardinality of Rad$.$ Thus each $G$-series extends to a maximal one.

\begin{proposition}
\label{P3.5}Let $G\subseteq$ \emph{Rad} and $L=\cap_{_{g\in G}}$%
\emph{Fix}$\left(  g\right)  .$ Then $L$ is an enveloping set\emph{.
}Moreover\emph{,}\smallskip

\emph{(i) \ \ }for each maximal ascending superposition $G$-series $\left(
h_{\alpha}\right)  _{\alpha\leq\gamma},$%
\[
\emph{\ }f_{_{L}}=h_{\gamma}=\vee^{\emph{Rad}}G\in\emph{Rad},\text{ }%
h_{\gamma}=g\circ h_{\gamma}\text{ for all }g\in G\text{ and }\emph{Fix}%
(h_{\gamma})=L;
\]

\emph{(ii)} \ if $f$ is a radical map and $g\circ f=f$ for all $g\in G,$ then
\emph{ }$h_{\gamma}\lesssim f;\smallskip$

\emph{(iii) }if $x\leq y\in L$ then $h_{\gamma}(x)\leq y.$
\end{proposition}

\begin{proof}
(i) Set\textbf{ }$\mathfrak{h}=\vee^{\text{Rad}}G.$ By Theorem \ref{T3.9},
$\mathfrak{h}\in$ Rad and Fix$(\mathfrak{h})=L.$ As $\left(  h_{\alpha
}\right)  _{1\leq\alpha\leq\gamma}$ is maximal, $h_{\gamma}=g\circ h_{\gamma}$
for all $g\in G$. So $h_{\gamma}\left(  x\right)  =g(h_{\gamma}\left(
x\right)  )\in L$ for all $x\in Q$. Hence $h_{\gamma}(Q)\subseteq L$ and
Fix($h_{\gamma})\subseteq L.$

If $z\in L$ then $g\left(  z\right)  =z$ for all $g\in G$. As $h_{1}\in G,$
$h_{1}\left(  z\right)  =z$ and, by induction, $h_{2}\left(  z\right)
=z$,\ldots, $h_{\gamma}\left(  z\right)  =z$. Thus $L\subseteq$ Fix$(h_{\gamma
}).$ Hence $L=$ Fix$(h_{\gamma})$ and $h_{\gamma}(h_{\gamma}(x))=h_{\gamma
}(x)$ for all $x\in Q$.

As $x\leq g(x)$ for all $x\in Q$ and $g\in G$, we have $x\leq h_{1}\left(
x\right)  \leq h_{2}\left(  x\right)  ...$ By induction, $x\leq h_{\gamma
}(x)=h_{\gamma}(h_{\gamma}(x))$. If $x\leq y$ then, similarly by induction,
$h_{\gamma}\left(  x\right)  \leq h_{\gamma}\left(  y\right)  .$ So
$h_{\gamma}\in$ Rad$.$ As Fix($\mathfrak{h})=L=$ Fix($h_{\gamma}),$ we have
$\mathfrak{h}=h_{\gamma}=f_{_{L}}$ by Lemma \ref{L3.14} and Proposition
\ref{P3.1}$.$

(ii) Let $f$ be a radical map and $g\circ f=f$ for all $g\in G.$ As $x\leq
f(x)$ for all $x\in Q,$ $g(x)\leq g(f(x))=f(x)$ for all $g\in G.$ Thus
$h_{1}\lesssim f.$ Similarly, we prove by induction that $h_{\gamma}\lesssim
f.$

(iii)\emph{ }As $h_{\gamma}$ is a radical map, $h_{\gamma}(x)\leq h_{\gamma
}(y)$ by (\ref{f3.1}), and $h_{\gamma}(y)=y$ by (i).\bigskip
\end{proof}

By Proposition \ref{P3.5}, the radical map $f_{L}=\vee^{\text{Rad}}%
G=h_{\gamma}$ for $L=\cap_{g\in G}$Fix$\left(  g\right)  $, is obtained "step
by step" via maximal ascending superposition $G$-series. Note that, if for
some $\ll$ in Ref($Q),$%
\begin{equation}
g(\ll)=\text{ }\ll\text{ for all }g\in G,\text{ then }\ll\text{ belongs to
}L\text{ and }f_{L}(\ll)=\text{ }\ll\text{.} \label{3.1a}%
\end{equation}

By Theorem \ref{P3.12}, $L_{\text{uc}},$ $L_{\text{dc}},$ $L_{\text{ue}%
},L_{\text{de}},L_{\text{o}}$ are enveloping sets in Ref$\left(  Q\right)  $.
So, by Proposition \ref{P3.1}, $f_{_{L_{\text{uc}}}},$ $f_{_{L_{\text{dc}}}},$
$f_{_{L_{\text{ue}}}},$ $f_{_{L_{\text{de}}}},$ $f_{_{L_{\text{o}}}}$ are
radical maps. Set $G=\{f_{_{L_{\text{uc}}}},f_{_{L_{\text{dc}}}}%
,f_{_{L_{\text{ue}}}},f_{_{L_{\text{de}}}},f_{_{L_{\text{o}}}}\}.$ By
Corollary \ref{C3.2}, $L=L_{\text{uc}}\cap L_{\text{dc}}\cap L_{\text{ue}}\cap
L_{\text{de}}\cap L_{\text{o}}=\cap_{_{g\in G}}$Fix$\left(  g\right)  $ is
enveloping. As above, $f_{L}=\vee^{\text{Rad}}G=h_{\gamma}$ for any maximal
ascending superposition $G$-series $\left(  h_{\alpha}\right)  _{1\leq
\alpha\leq\gamma}$ (see (\ref{3.33})). Set%
\begin{equation}
\widetilde{\ll}=h_{\gamma}(\ll)=f_{L}(\ll)\text{ for }\ll\text{ from
Ref}\left(  Q\right)  . \label{6.4}%
\end{equation}

\begin{proposition}
\label{P3.7}For each $\ll$ from \emph{Ref}$(Q),$ $\widetilde{\ll}$ is the
smallest $\mathbf{TT}$-order majorizing $\ll.$
\end{proposition}

\begin{proof}
By Proposition \ref{P3.5}, $h_{\gamma}\in$ Rad$_{_{\text{Ref(}Q)}}$ and
$h_{\gamma}=g\circ h_{\gamma}$ for all $g\in G.$ Hence $g(\widetilde{\ll
})=g(h_{\gamma}(\ll))=h_{\gamma}(\ll)=$ $\widetilde{\ll}$ for all $g\in G.$
So, by (\ref{3.1a}), $\widetilde{\ll}$ belongs to $L,$ so that it is an up-
and down-contiguous, and up- and down-expanded order, i.e., it is an
$\mathbf{TT}$-order.

If $\prec$ is a $\mathbf{TT}$-order then it is contiguous and expanded (see
Definition \ref{D3.2}). So $\prec$ belongs to $L=\cap_{g\in G}$Fix$\left(
g\right)  .$ Thus $g(\prec)=$ $\prec$ for all $g\in G.$ If $\ll$ $\subseteq$
$\prec$ (see (\ref{0})) then$\ \widetilde{\ll}=h_{\gamma}(\ll)\subseteq$
$\prec$ by Proposition \ref{P3.5}(iii).\bigskip
\end{proof}

By Theorem \ref{P3.12}, $L_{\triangleright},L_{\triangleleft}$ are enveloping
sets in Ref$\left(  Q\right)  $. Hence $G_{\triangleright\triangleleft
}=\{f_{_{L_{\triangleright}}},f_{_{L_{\triangleleft}}}\}$ consists of radical
maps, so that $L_{\triangleright\triangleleft}=L_{\triangleright}\cap
L_{\triangleleft}$ is enveloping. The radical map $f_{L}=\vee^{\text{Rad}%
}G_{\triangleright\triangleleft}$ is obtained by constructing a maximal
ascending superposition $G_{\triangleright\triangleleft}$-series $\left(
h_{\alpha}\right)  _{1\leq\alpha\leq\gamma}$ (see (\ref{3.33})) and
$f_{_{L_{\triangleright\triangleleft}}}=h_{\gamma}.$ For $\ll$\emph{ }from
Ref$(Q),$ set\emph{ }%
\begin{equation}
\overline{\ll}=h_{\gamma}(\ll)=f_{_{L_{\triangleright\triangleleft}}}(\ll).
\label{6.5}%
\end{equation}

\begin{example}
\label{E8.1}The relation $\overline{\ll}$ in $Q$ can be neither
contiguous\emph{, }nor expanded. \emph{Indeed,}\smallskip

$1)$ \emph{Let }$Q=\{\mathbf{0},a,\mathbf{1\},}$ $\mathbf{0}<a<\mathbf{1}$
\emph{and} $\mathbf{0}\ll\mathbf{1.}$ \emph{Then} $\overline{\ll}$ $=$ $\ll$
is not contiguous.\smallskip

$2)$ \emph{Let }$Q$ \emph{and }$\ll$ \emph{be as in} \emph{Example \ref{E3.2}.
Then }$\overline{\ll}$ $=$ $\ll$ is not expanded.
\end{example}

\begin{proposition}
\label{L3.3}\emph{(i)\ }$\overline{\ll}=$ $\overline{\ll}^{\triangleright}=$
$\overline{\ll}^{\triangleleft}$ and $\overline{\ll}$\textbf{ }is\emph{ }the
smallest of all relations $\prec$ satisfying $\ll$ $\subseteq$ $\prec$ $=$
$\prec^{\triangleright}=$ $\prec^{\triangleleft}.\smallskip$

\emph{(ii) } $\overline{\ll}\subseteq$ $\widetilde{\ll};$ if $\overline{\ll}$
is a $\mathbf{TT}$-order then $\overline{\ll}=$ $\widetilde{\ll}.\smallskip$

\emph{(iii) }If $\ll$ is an expanded order then $\overline{\ll}$ $=$
$\ll.\smallskip$

\emph{(iv) }Let $Q$ be a chain. If $\ll$ is up- or down-contiguous\emph{,}
$\overline{\ll}$ is up- or down-contiguous\emph{, }respectively.
\end{proposition}

\begin{proof}
The proof of (i) is the same as in Proposition \ref{P3.7}.

(ii) As $\widetilde{\ll}$ is a $\mathbf{TT}$-order, $\ll$ $\subseteq$
$\widetilde{\ll}$ $=$ $\widetilde{\ll}^{\triangleright}=\widetilde{\ll
}^{\triangleleft}$ by Theorem \ref{P2.1}. Hence $\overline{\ll}$ $\subseteq$
$\widetilde{\ll}$ by (i).

By Proposition \ref{P3.7}, $\widetilde{\ll}$ is the smallest $\mathbf{TT}%
$-order majorizing $\ll.$ So, if $\overline{\ll}$ is a $\mathbf{TT}$-order
then $\widetilde{\ll}$ $\subseteq$ $\overline{\ll}.$ Thus $\overline{\ll}=$
$\widetilde{\ll}.$

(iii) If $\ll$ is an expanded order, it follows from Theorem \ref{P2.1} that
$\ll$ $=$ $\ll^{\triangleleft}=$ $\ll^{\triangleright}.$ Thus $g(\ll)=$ $\ll$
for all $g\in G_{\triangleright\triangleleft}.$ By (\ref{3.1a}),
$\overline{\ll}=$ $h_{\gamma}(\ll)=$ $\ll.$

(iv) If $Q$ is a chain then, by Lemma \ref{L3.4}, if $\ll$ is up-contiguous,
the relations $\ll^{\triangleleft}$ and $\ll^{\triangleright}$ are
up-contiguous. By induction, $\overline{\ll}$ is up-contiguous. The
down-contiguous case is similar.\bigskip
\end{proof}

By Proposition \ref{L3.3}, $\overline{\ll}$ $\subseteq$ $\widetilde{\ll}.$
Below we consider the case when they coincide.

\begin{proposition}
\label{P7.1}Let $Q$ have properties \emph{(JIDC)} and \emph{(MIDC)}
\emph{(}see Definition \emph{\ref{D6.1}). }If $\ll$ is an $\mathbf{HH}%
$-relation\emph{ }then $\overline{\ll}$ $=$ $\widetilde{\ll}$ is an\emph{
}$\mathbf{RR}$-order.
\end{proposition}

\begin{proof}
Let $Q$ have properties (JIDC) and (MIDC) and let $G_{\triangleright
\triangleleft}=\{f_{_{L_{\triangleright}}},f_{_{L_{\triangleleft}}}\}.$ By
Lemma \ref{L3.1}, the radical maps $f_{_{L_{\triangleright}}}%
,f_{_{L_{\triangleleft}}}$ on $Q$ act by $f_{_{L_{\triangleright}}}(\ll)=$
$\ll^{\triangleright}$ and $f_{_{L_{\triangleleft}}}(\ll)=$ $\ll
^{\triangleleft}.$ Let $\ll$ be an $\mathbf{HH}$-relation. By Lemma
\ref{L6.2}, $\ll^{\triangleleft}$ is an $\mathbf{H}$-order and a
dual\ $\mathbf{R}$-order, and $\ll^{\triangleright}$ is an $\mathbf{R}$-order
and a dual $\mathbf{H}$-order. So $f_{_{L_{\triangleleft}}}(\ll)=$
$\ll^{\triangleleft}$ and $f_{_{L_{\triangleright}}}(\ll)=$ $\ll
^{\triangleright}$ are $\mathbf{HH}$-orders.

Let $\left(  h_{\alpha}\right)  _{1\leq\alpha\leq\gamma}$ be a maximal
ascending superposition $G_{\triangleright\triangleleft}$-series Set
$\ll_{\alpha}=h_{\alpha}(\ll)$. If $\ll_{\alpha}$ is an $\mathbf{HH}$-order,
then $\ll_{\alpha+1}$ is either $f_{_{L_{\triangleright}}}(\ll_{\alpha})$, or
$f_{_{L_{\triangleleft}}}(\ll_{\alpha}),$ so that it is an $\mathbf{HH}$-order
by above.

Let $\beta$ be a limit ordinal. By (\ref{3.33}), $\ll_{\beta}=h_{\beta}%
(\ll)\overset{(\ref{3.33})}{=}\vee_{_{\alpha<\beta}}h_{\alpha}(\ll
)=\vee_{_{\alpha<\beta}}\ll_{\alpha}.$ If $a\ll_{\beta}b$ in $Q$, it follows
from (\ref{i}) that $a\ll_{\alpha}b$ for some $\alpha<\beta.$ As $\ll_{\alpha
}$ is an $\mathbf{HH}$-order, $a\wedge z\ll_{\alpha}b\wedge z$ and $a\vee
z\ll_{\alpha}b\vee z$ for all $z\in Q.$ Hence, by (\ref{i}), $a\wedge
z\ll_{\beta}b\wedge z$ and $a\vee z\ll_{\beta}b\vee z.$ Thus $\ll_{\beta}$ is
an $\mathbf{HH}$-order. So, by induction, $\overline{\ll}=h_{\gamma}(\ll)$ is
an $\mathbf{HH}$-order.

By (\ref{8.1}), $\overline{\ll}^{\triangleleft}=f_{_{L_{\triangleleft}}%
}(\overline{\ll})=f_{_{L_{\triangleleft}}}(h_{\gamma}(\ll))=h_{\gamma}%
(\ll)=\overline{\ll}$ and $\overline{\ll}^{\triangleright}%
=f_{_{L_{\triangleright}}}(\overline{\ll})=f_{_{L_{\triangleright}}}%
(h_{\gamma}(\ll))=h_{\gamma}(\ll)=\overline{\ll}.$ Then, by Corollary
\ref{pr3}, $\overline{\ll}$ is an $\mathbf{RR}$-order. So $\overline{\ll}$ is
a $\mathbf{TT}$-order. By Proposition \ref{L3.3}, $\overline{\ll}=$
$\widetilde{\ll}.$
\end{proof}

\begin{corollary}
\label{C8.1}Let a modular complete lattice $Q$ have properties \emph{(JIDC)}
and \emph{(MIDC) (}in particular\emph{, }$Q$ is a complete chain\emph{).} Let
$<_{\mathfrak{g}}$ and $<_{\mathfrak{c}}$ be the relations in $Q$ defined in
Definition \emph{\ref{D7.1}. }Then $\overline{<_{\mathfrak{g}}}%
=\widetilde{<_{\mathfrak{g}}}$ and $<_{\mathfrak{c}}=\overline{<_{\mathfrak{c}%
}}=\widetilde{<_{\mathfrak{c}}}$ are $\mathbf{RR}$-orders.
\end{corollary}

\begin{proof}
By Corollary \ref{C6.1}, $<_{\mathfrak{g}}$ is an $\mathbf{HH}$-relation, as
$Q$ is modular. Since $Q$ has properties (JIDC) and (MIDC), Proposition
\ref{P7.1} implies that $\overline{<_{\mathfrak{g}}}=$
$\widetilde{<_{\mathfrak{g}}}$ is an $\mathbf{RR}$-order.

By Theorem \ref{T6.2}, $\leq_{\mathfrak{c}}$ is an $\mathbf{RR}$-order, as $Q$
is modular and has properties (JIDC) and (MIDC)$.$ Hence $<_{\mathfrak{c}}$
$=$ $<_{\mathfrak{c}}^{\triangleleft}$ $=$ $<_{\mathfrak{c}}^{\triangleright}$
by Corollary \ref{pr3}. So $<_{\mathfrak{c}}=$ $\overline{<_{\mathfrak{c}}}$
by (\ref{3.1a}). It also follows from Proposition \ref{P7.1} that
$\overline{\ll_{\mathfrak{c}}}$ $=$ $\widetilde{\ll_{\mathfrak{c}}}.$
\end{proof}

\begin{problem}
\label{Pr8.1}\emph{Let a modular lattice do not have properties (JIDC) and
(MIDC)}$.$\emph{ Will }$\overline{<_{\mathfrak{g}}}=$\emph{ }%
$\widetilde{<_{\mathfrak{g}}}$\emph{ and }$<_{\mathfrak{c}}=$\emph{
}$\overline{<_{\mathfrak{c}}}=$\emph{ }$\widetilde{<_{\mathfrak{c}}}?$
\emph{Will they\ still be} $\mathbf{RR}$-\emph{orders?}
\end{problem}

Let $\ll$ belong to Ref$(Q).$ We say that a chain $C$ in $Q$ is
\textit{complete }$\ll$\textit{-gap dense} if,
\begin{equation}
\text{for all }z<w\text{ in }C\text{ there are }u\ll v\text{ in }[z,w]_{_{C}%
}\text{ such that }[u,v]_{_{C}}\text{\emph{ }is a gap in }C\emph{.}
\label{8.2}%
\end{equation}
Consider the following reflexive relation $\ll_{\text{g-d}}$ on $Q$: for
$a<b,$ we write%
\begin{equation}
a\ll_{\text{g-d}}b\text{ if there is a complete }\ll\text{-gap dense chain
from }a\text{ to }b. \label{3.x}%
\end{equation}

\begin{example}
\label{E6.1}The relation $\ll_{\text{\emph{g-d}}}$ in $Q$ can be neither
contiguous\emph{, }nor expanded. \emph{Indeed,}\smallskip

$1)$\emph{ Let} $Q=\{\mathbf{0},a,b,\mathbf{1\},}$ $\mathbf{0}<a<\mathbf{1,}$
$\mathbf{0}<b<\mathbf{1.}$ \emph{Let }$\ll$ \emph{be a reflexive order in }%
$Q$\emph{ such that }$\mathbf{0}\ll a\ll\mathbf{1.}$ \emph{Then }%
$\ll_{\emph{g-d}}=$ $\ll$ \emph{and} $\mathbf{0}\ll_{\emph{g-d}}\mathbf{1.}$
\emph{However} $[\mathbf{0,1}]\nsubseteq\lbrack\ll_{\emph{g-d}},\mathbf{1].}$
\emph{So }$\ll_{\emph{g-d}}$\emph{ }is not contiguous.\smallskip

$2)$\emph{ Let }$\ll$ \emph{be} \emph{the relation considered in Example
\ref{E3.2}. Then }$\ll_{\emph{g-d}}=$ $\ll$ is not expanded\emph{.
\ \ }$\blacksquare$\emph{\smallskip}
\end{example}

Note that $\ll_{\text{g-d}}$ is an order. Each complete upper (lower) $\ll
$-gap chain is $\ll$-gap dense.

The proof of the following lemma is evident.

\begin{lemma}
\label{L3.11}Let $T$ be a chain from $a$ to $b$ such that each $t\in
T\diagdown\{b\}$ has an immediate successor $t_{s}$ in $T,$ i.e.\emph{,}
$[t,t_{s}]_{_{T}}=\{t,t_{s}\}$ is a gap in $T.$ Let $\ll$ belong to
\emph{Ref(}$Q)$ and suppose that\emph{,} for each $t\in T\diagdown\{b\},$
there is a $\ll$-gap dense chain $C_{t}$ in $Q$ from $t$ to $t_{s}.$ Set
$C_{b}=\{b\}.$ Then\emph{\ }the chain $C=\cup_{t\in T}C_{t}$ is $\ll$-gap dense.
\end{lemma}

\begin{corollary}
\label{C3.3}\emph{(i) }For each $\ll$ from \emph{Ref(}$Q),$ $\ll$ $\subseteq$
$\ll_{\text{\emph{g-d}}}=(\ll_{\text{\emph{g-d}}})^{\triangleright}%
=(\ll_{\text{\emph{g-d}}})^{\triangleleft}.$ So $\overline{\ll
_{\text{\emph{g-d}}}}=$ $\ll_{\text{\emph{g-d}}}$.\smallskip

\emph{(ii) }For each $\ll$ from \emph{Ref(}$Q),$ $\overline{\ll}$ $\subseteq$
$\ll_{\text{\emph{g-d}}}.$
\end{corollary}

\begin{proof}
(i) Clearly, $\ll$ $\subseteq$ $\ll_{\text{g-d}}$ and $\ll_{\text{g-d}}$ is an
order.\emph{ }Set $\prec$ $=$ $\ll_{\text{g-d}}.$ If $a$ $\prec
^{\triangleright}$ $b$ then there is a complete upper $\prec$-gap chain $T$
from $a$ to $b,$ i.e., each $t\in T\diagdown\{b\}$ has the immediate $\prec
$-successor $t_{s}$ in $T.$ Then there is a complete $\ll$-gap dense chain
$C^{t}$ in $Q$ from $t$ to $t_{s}.$ Set $C^{b}=\{b\}.$ By Lemma \ref{L3.11},
$C=\cup_{_{t\in T}}C^{t}$ is a complete $\ll$-gap dense chain from $a$ to $b.$
Thus $a\prec b$, so that $\prec^{\triangleright}$ $\subseteq$ $\prec.$ As
$\prec$ $\subseteq$ $\prec^{\triangleright}$ by (\ref{2,8}), we have
$\prec^{\triangleright}$ $=$ $\prec.$ Similarly, $\prec^{\triangleleft}$ $=$
$\prec.$ By (\ref{3.1a}), $\overline{\prec}$ $=$ $\prec.$

(ii) By Proposition \ref{L3.3}(i), $\overline{\ll}$\textbf{ }is\emph{ }the
smallest of all relations $\prec$ satisfying the condition $\ll$ $\subseteq$
$\prec$ $=$ $\prec^{\triangleright}=$ $\prec^{\triangleleft}.$ As the relation
$\ll_{\text{g-d}}$ satisfies this condition by (i), $\overline{\ll}$
$\subseteq$ $\ll_{\text{g-d}}.\bigskip$
\end{proof}

In general, $\overline{\ll}\neq$ $\ll_{\text{g-d}}$ in $Q,$ even if $Q$ is a
chain. To show this, consider the gap relation $<_{\mathfrak{g}}$ (see
Definition \ref{D7.1}). First, it should be noted that, although the relations
$<_{\mathfrak{g}}$ and $<_{\mathfrak{c}}$ are not compatible (if
$a<_{\mathfrak{c}}b$ then $a\not <  _{\mathfrak{g}}b$ and vice versa), the
relations $\overline{<_{\mathfrak{g}}}$ and $<_{\mathfrak{c}}$ can be
compatible, as the lattice $Q$ in Example \ref{E6.3} shows: $\mathbf{0}%
<_{\mathfrak{c}}\mathbf{1}$ and $\mathbf{0}$ $\overline{<_{\mathfrak{g}}}$
$\mathbf{1},$ as $\mathbf{0}<_{\mathfrak{g}}a$ and $a<_{\mathfrak{g}%
}\mathbf{1}$\emph{.}

\begin{proposition}
\label{Pr3.1}If $Q$ has no continuous chains \emph{(}see Definition
\emph{\ref{D7.1}) }then $(<_{\mathfrak{g}})_{\text{\emph{g-d}}}$ $=$ $\leq.$
\end{proposition}

\begin{proof}
For $a<b$ in $Q,$ let $C$ be a maximal chain from $a$ to $b.$ By Lemma
\ref{L2.1p}, $C$ is complete. By (\ref{3.x}), $C$ is $<_{\mathfrak{g}}$-gap
dense if
\begin{equation}
\text{each interval }[x,y]_{_{C}}\subseteq C\text{ contains a gap }%
[u,v]_{_{Q}}. \label{3.b}%
\end{equation}
Hence, if $C$ is not $\ll_{\mathfrak{g}}$-gap dense, then there is an interval
$[x,y]_{_{C}}$ in $C$ without gaps$.$ So, by Definition \ref{D7.1},
$[x,y]_{_{C}}$ is a continuous chain$,$ a contradiction. Thus $C$ is
$<_{\mathfrak{g}}$-gap dense, so that $a$ $(<_{\mathfrak{g}})_{\text{g-d}}$
$b.$ Hence $\leq$ $\subseteq$ $(<_{\mathfrak{g}})_{\text{g-d}}.$ As we always
have $(<_{\mathfrak{g}})_{\text{g-d}}\subseteq$ $\leq,$ so $(<_{\mathfrak{g}%
})_{\text{g-d}}=$ $\leq.\bigskip$
\end{proof}

We shall now consider an example of a chain, for which $\overline
{<_{\mathfrak{g}}}\neq$ $(<_{\mathfrak{g}})_{\text{g-d}}.$

\begin{example}
\label{e7}\emph{Let }$Q=\{(t;n)$\emph{: }$t\in\left[  0,1\right]
\subset\mathbb{R},$ $n=0,1\}.$ \emph{Let }$\left(  t;n\right)  <\left(
s;m\right)  ,$ \emph{if either }$t<s,$\emph{ or }$t=s,$\emph{ }$n=0$
\emph{and} $m=1$\emph{. Then }$Q$\emph{ is }a complete chain\emph{ from
}$\mathbf{0}=(0;0)$\emph{ to }$\mathbf{1}=(1;1)$\emph{.}

\emph{As each interval }$[(t,0),(t,1)]$ \emph{in} $Q$ \emph{is a gap, all
intervals in }$Q$ \emph{have gaps. So, by (\ref{3.b}), }$Q$ \emph{is
}$<_{\mathfrak{g}}$\emph{-gap dense. By Proposition \ref{Pr3.1},
}$(<_{\mathfrak{g}})_{\text{\emph{g-d}}}=$ $\leq.$ \emph{Hence }$\mathbf{0}$
$(<_{\mathfrak{g}})_{\text{\emph{g-d}}}$ $\mathbf{1.}$

\emph{On the hand, it is easy to see that} $<_{\mathfrak{g}}\,=\,\overline
{<_{\mathfrak{g}}}.$ \emph{So }$\mathbf{0}$ $\not <  _{\mathfrak{g}}$
$\mathbf{1.}$ \emph{Thus} $\overline{<_{\mathfrak{g}}}\neq$ $(<_{\mathfrak{g}%
})_{\text{\emph{g-d}}}.$ \ \ \ $\blacksquare$
\end{example}

Although the relations $\overline{\ll}$ and $\ll_{\text{g-d}}$ do not
coincide, they are closely linked.

\begin{theorem}
\label{T3.8}If $a$ $\overline{\ll}$ $b$ then there is a complete $\ll$-gap
dense chain $C$ from $a$ to $b$ such that $x$ $\overline{\ll}$ $y$ for all
$x<y$ in $C.$
\end{theorem}

\begin{proof}
Let $G_{\triangleright\triangleleft}=\{f_{L_{\triangleright}}%
,f_{L_{\triangleleft}}\}$ and $\left(  h_{\alpha}\right)  _{1\leq\alpha
\leq\gamma}$ be a maximal ascending superposition $G_{\triangleright
\triangleleft}$-series, $h_{1}=f_{L_{\triangleleft}}$ and $h_{\gamma
}=f_{_{L_{\triangleright\triangleleft}}}.$ Set $\ll_{\alpha}=h_{\alpha}(\ll)$
and $\overline{\ll}=$ $\ll_{\gamma}.$ Consider the following induction.

Let $a\ll_{_{1}}b,$ i.e., $a\ll^{\triangleleft}b$ for some $a,b\in Q.$ Then
there is a complete lower $\ll$-gap chain $C^{1}$ from $a$ to $b$. As each
$x\in C^{1}\diagdown\{a\}$ has an immediate $\ll$-predecessor $x_{p}$:
$[x_{p},x]_{_{C}}$ is a gap in $C^{1}$ and $x_{p}\ll x$. So $C^{1}$ is $\ll
$-gap dense$.$ For all $x<y$ in $C^{1},$ $[x,y]_{_{C^{1}}}$ is a complete
lower $\ll$-gap chain from $x$ to $y$. Thus $x\ll^{\triangleleft}y.$

Assume that if $a\ll_{\alpha}b,$ for some $a,b\in Q$ and some $\alpha<\gamma,$
then%
\begin{align}
1)\text{ there is a complete }  &  \ll\text{-gap dense chain }C^{\alpha}\text{
from }a\text{ to }b\text{ and}\label{3.16}\\
2)\text{ }x  &  \ll_{\alpha}y\text{ for all }x<y\text{ in }C^{\alpha}.
\label{3.17}%
\end{align}

Let now $a\ll_{\alpha+1}b$, say, $a$ $(\ll_{\alpha})^{\triangleright}b$ for
some $a,b\in Q.$ Then there is a complete upper gap $\ll_{\alpha}$-chain $T$
from $a$ to $b,$ i.e., each $t\in T\diagdown\{b\}$ has an immediate
$\ll_{\alpha}$-successor $t_{s}\in T$ -- $[t,t_{s}]_{_{T}}=\{t,t_{s}\}$ is a
gap in $T$ and $t\ll_{\alpha}t_{s}$. By (\ref{3.16}), for each $t\in
T\diagdown\{b\},$ there is a complete $\ll$-gap dense chain $C_{t}$ from $t$
to $t_{s}$ satisfying (\ref{3.17})$.$ Set $C_{b}=\{b\}.$ By Lemma \ref{L3.11},
$C^{\alpha+1}=\cup_{_{t\in T}}C_{t}$ is a complete $\ll$-gap dense chain from
$a$ to $b$.

Let $x<y$ in $C^{\alpha+1}.$ Then $x\in C_{t},$ $y\in C_{t^{\prime}}$ for some
$t\leq t^{\prime}$ in $T.$ If $t=t^{\prime}$ then $x,y\in C_{t}$ and
$x\ll_{\alpha}y$ , by (\ref{3.17}). Thus $x\ll_{\alpha+1}y$ by (\ref{2,8}). If
$t\neq t^{\prime}$ then either $x<t_{s}<y,$ or $x=t_{s}<y,$ or $x<t_{s}=y.$
Let $x<t_{s}<y.$ As $C_{t},C_{t^{\prime}}$ satisfy (\ref{3.17}), $x\ll
_{\alpha}t_{s}$ and $t^{\prime}\ll_{\alpha}y.$ Hence $x\ll_{\alpha+1}t_{s}$
and $t^{\prime}\ll_{\alpha+1}y$ by (\ref{2,8}). As $T$ is a complete upper gap
$\ll_{\alpha}$-chain, $t_{s}\ll_{\alpha+1}t^{\prime}.$ As $\ll_{\alpha+1}$ is
an order, $x\ll_{\alpha+1}y.$ Similarly, $x\ll_{\alpha+1}y$ when $x=t_{s}<y,$
or $x<t_{s}=y.$ Thus (\ref{3.17}) holds for all $x<y$ in $C^{\alpha+1}$ and
$\ll_{\alpha+1}.$

Let $\beta$ be a limit ordinal and let, for each $\alpha<\beta,$ $a\ll
_{\alpha}b$ imply (\ref{3.16}) and (\ref{3.17}). Let now $a\ll_{\beta}b.$ As
$\ll_{\beta}=h_{\beta}(\ll)\overset{(\ref{3.33})}{=}\vee_{_{\alpha<\beta}%
}h_{\alpha}(\ll)=\vee_{_{\alpha<\beta}}\ll_{\alpha},$ it follows from
(\ref{i}) that $a\ll_{\alpha}b$ for some $\alpha<\beta.$ Hence there is a
complete $\ll$-gap dense chain $C^{\alpha}$ from $a$ to $b$ satisfying
(\ref{3.17}). Set $C^{\beta}=C^{\alpha}.$ Then $C^{\beta}$ satisfies
(\ref{3.16}). Let $x<y$ in $C^{\beta}.$ By (\ref{3.17}), $x\ll_{\alpha}y.$
Hence, by (\ref{i}), $x\ll_{\beta}y.$ Thus (\ref{3.16}) and (\ref{3.17}) hold
for $\beta.$ By induction, if $a\ll_{\gamma}b$ then there is a complete $\ll
$-gap dense chain $C^{\gamma}$ from $a$ to $b$ satisfying (\ref{3.17}) for
$\alpha=\gamma.$ As $\ll_{\gamma}=\overline{\ll},$ the proof is complete.
\end{proof}

\section{Relations in lattices of subspaces of Banach spaces}

Let $X$ be a Banach space. The set Ln$(X)$ of all linear subspaces of $X$ and
the set Cl($X)$ of all closed subspaces of $X$ are complete lattices with
order $\subseteq$,%
\begin{align}
\mathbf{0}  &  =\{0\},\text{ }\mathbf{1}=X,\text{ }\wedge G=\bigcap
\limits_{Y\in G}Y\text{ for a subset }G\neq\varnothing\text{ in Ln(}X)\text{
and Cl(}X),\nonumber\\
\vee G  &  =\sum_{Y\in G}Y\text{ for }G\subseteq\text{ Ln(}X)\text{ and }\vee
G=\overline{\sum_{Y\in G}Y}\text{ for }G\subseteq\text{ Cl(}X). \label{9.0}%
\end{align}
As before, for $L\subseteq M$ in $Q,$ we write%
\begin{equation}
L<_{\mathfrak{g}}M\text{ if either }L=M,\text{ or }[L,M]_{_{Q}}\text{ is a gap
in }Q. \label{9.5}%
\end{equation}
In this section we concentrate on the relation $<_{\mathfrak{g}}$ in the
sublattices of Ln($X)$ and Cl($X).$ We show that all sublattices of Ln($X)$
are modular, so that $<_{\mathfrak{g}}$ is an $\mathbf{HH}$-relation in all of
them$.$ The lattice Cl($X)$ is not modular. Although $<_{\mathfrak{g}}$ is an
$\mathbf{HH}$-relation in many of its sublattices (Corollary \ref{C9.3}),
there are sublattices of Cl($X),$ where it is neither an $\mathbf{H}$-, nor a
dual $\mathbf{H}$-relation (Corollary \ref{C9.1}).

Let $Q$ be a complete sublattice of Ln($X),$ or Cl($X).$ As in (\ref{1}) and
(\ref{1.5}), for $\ll$ from Ref($Q),$
\begin{align}
\lbrack &  \ll,L]_{_{Q}}=\{K\in Q\text{: }K\ll L\mathbf{\},}\text{ }%
[L,\ll]_{_{Q}}=\{K\in Q\text{: }L\ll K\mathbf{\}}\text{ for }L\in
Q;\nonumber\\
\sigma_{_{\ll}}(L)  &  =\wedge\lbrack\ll,L]_{_{Q}}=\cap\{K\text{: }K\in
\lbrack\ll,L]_{_{Q}}\}\text{ and }s^{\ll}(L)=\vee\lbrack L,\ll]_{_{Q}},\text{
so that}\label{3.3}\\
s^{\ll}(L)  &  =\sum{}_{K\in\lbrack L,\ll]_{_{Q}}}K\text{\ in Ln(}X),\text{
and }s^{\ll}(L)=\overline{\sum{}_{K\in\lbrack L,\ll]_{_{Q}}}K}\text{ in
Cl(}X). \label{3.2}%
\end{align}
It follows from Theorem \ref{T3.5} that if $\ll$ is an $\mathbf{H}$-relation
then $\ll^{\triangleright}$ is an $\mathbf{R}$-order in $Q$ and $\mathfrak{r}%
_{_{L}}=s^{\ll^{\triangleright}}(L)$ is the unique $\ll^{\triangleright}%
$-radical in $[L,X]_{_{Q}}$, i.e., $L\ll^{\triangleright}\mathfrak{r}_{_{L}}$
$\overleftarrow{\ll}$ $X.$

If $\ll$ is a dual $\mathbf{H}$-relation then $\ll^{\triangleleft}$ is a dual
$\mathbf{R}$-order in $Q$ and $\mathfrak{p}_{_{L}}=\sigma_{_{\ll
^{\triangleleft}}}(L)$ is the unique dual $\ll^{\triangleleft}$-radical in
$[\{0\},L]_{_{Q}}$, i.e., $\{0\}$ $\overrightarrow{\ll}$ $\mathfrak{p}_{_{L}%
}\ll^{\triangleleft}L.$ By (\ref{2'}) and (\ref{2}),%
\begin{equation}
s^{\ll}(L)\subseteq\mathfrak{r}_{_{L}}\text{ for an }\mathbf{H}\text{-relation
}\ll,\text{ and }\mathfrak{p}_{_{L}}\subseteq\sigma_{_{\ll}}(L)\text{ for a
dual }\mathbf{H}\text{-relation }\ll\text{.} \label{3.4}%
\end{equation}
Thus, for a complete sublattice $Q$ of Ln($X),$ or of Cl($X),$ Theorem
\ref{T3.5} yields

\begin{theorem}
\label{T4.6}\emph{(i) }Let $\ll$ be an $\mathbf{H}$-relation and $L\subseteq
K$ in $Q.\smallskip$

$1)$ If $K\subseteq\mathfrak{r}_{_{L}}$ then there is an ascending $\ll
$-series of spaces in $Q$ from $K$ to $\mathfrak{r}_{_{L}}.$\smallskip

$2)$ If $\mathfrak{r}_{_{L}}\nsubseteq K\neq X$ then $K$ has a $\ll$-successor
$S\in Q$\emph{: }$K\ll S;$ $\mathfrak{r}_{_{L}}$ has no $\ll$%
-successor$.\smallskip$

$3)$ If there is an ascending $\ll$-series of spaces in $Q$ from $L$ to $K$
then $K\subseteq\mathfrak{r}_{_{L}}.\smallskip$

\emph{(ii) }Let $\ll$ be a dual $\mathbf{H}$-relation and $K\subseteq L$ in
$Q.$

$1)$ If $\mathfrak{p}_{_{L}}\subseteq K$ then there is a descending $\ll
$-series of spaces in $Q$ from $K$ to $\mathfrak{p}_{_{L}}.\smallskip$

$2)$ If $\{0\}\neq K\nsubseteq\mathfrak{p}_{_{L}}$ then $K$ has a $\ll
$-predecessor $P\in Q$\emph{: }$P\ll K;$ $\mathfrak{p}_{_{L}}$ has no $\ll
$-predecessor$.\smallskip$

$3)$ If there is a descending $\ll$-series of spaces in $Q$ from $L$ to $K$
then $\mathfrak{p}_{_{L}}\subseteq K.$
\end{theorem}

If $X$ is separable then the $\ll$-series in Theorem \ref{T4.6} are isomorphic
to$\ \mathbb{N}.$

\begin{proposition}
\label{P9.5}Let $X$ be a separable Banach space and $Q$ be a complete
sublattice of \emph{Cl(}$X)$.$\smallskip$

\emph{(i) }If\emph{ }$\ll$ is a dual $\mathbf{H}$-relation in $Q$ then there
are spaces%
\begin{equation}
...\ll Y_{n}\ll...\ll Y_{1}\ll X\text{ in }Q\text{ such that }\cap
_{n=1}^{\infty}Y_{n}=\sigma_{_{\ll}}(X). \label{3.0}%
\end{equation}

There are also spaces $...\ll^{\triangleleft}Z_{n}\ll^{\triangleleft}%
...\ll^{\triangleleft}Z_{1}\ll^{\triangleleft}X$ \ in $Q$ such that%
\begin{equation}
\cap_{n=1}^{\infty}Z_{n}=\mathfrak{p}_{_{X}}\text{ is the dual }%
\ll^{\triangleleft}\text{-radical in }X. \label{9.8}%
\end{equation}

\emph{(ii) }If\emph{ }$\ll$ is an $\mathbf{H}$-relation then there are spaces
$\{0\}\ll Y_{1}\ll...\ll Y_{n}\ll...$ in $Q$ and

$\qquad$spaces $\{0\}\ll^{\triangleright}Z_{1}\ll^{\triangleright}%
...\ll^{\triangleright}Z_{n}\ll^{\triangleright}...$ in $Q$ such that
\[
\overline{\sum{}_{n=1}^{\infty}Y_{n}}=s^{\ll}(\{0\})\text{ and }\overline
{\sum{}_{n=1}^{\infty}Z_{n}}=\mathfrak{r}_{_{\{0\}}}\text{ is the }%
\ll^{\triangleright}\text{-radical in }X.
\]

\end{proposition}

\begin{proof}
(i) Set $\omega=\sigma_{_{\ll}}(X).$ Let $\mathbf{B}$ be the unit ball of
$X^{\ast}.$ For each $L\in\lbrack\ll,X]_{_{Q}}$, set $W_{L}=\{f\in
\mathbf{B}:L\subseteq\ker(f)\}.$ Set $W=\cup_{L\ll X}W_{L}.$ Then
$W\subseteq\mathbf{B}.$ As $X$ is separable, $\mathbf{B}$\textbf{ }is a
separable metric space in the $\sigma(X^{\ast},X)$-topology (see \cite[Section
4.1.7]{Sch}). Hence $W$ has a $\sigma(X^{\ast},X)$ dense sequence $\{f_{k}$:
$k\in\mathbb{N}\}$. As each $L=\cap_{f\in W_{L}}\ker f,$ we have $\cap_{k}%
\ker(f_{k})=\cap_{f\in W}\ker\left(  f\right)  =\cap_{L\ll X}L=\omega.$

For each $k\in\mathbb{N},$ choose $L_{k}\in\lbrack\ll,X]_{_{Q}}$ such that
$f_{k}\in W_{L_{k}}$. Then, by (\ref{3.2}), $\omega\subseteq\cap_{k=1}%
^{\infty}L_{k}\subseteq\cap_{k}\ker(f_{k})=\omega$. Hence $\cap_{k=1}^{\infty
}L_{k}=\omega$.

Set $Y_{n}=L_{1}\cap...\cap L_{n}\in Q.$ Then $Y_{k+1}\subseteq Y_{k}$ for all
$k,$ and $\cap_{n=1}^{\infty}Y_{n}=\cap_{k=1}^{\infty}L_{k}=\omega.$ Suppose,
by induction, that $Y_{k}\ll Y_{k-1}\ll...\ll Y_{1}\ll X$ for some $k.$ As all
$L_{n}\in\lbrack\ll,X]_{_{Q}},$ we have $L_{k+1}\ll X.$ Since $\ll$ is a dual
$\mathbf{H}$-order, it follows from Lemma$\ $\ref{le1}(ii) that%
\[
Y_{k+1}=Y_{k}\cap L_{k+1}=Y_{k}\wedge L_{k+1}\ll Y_{k}\wedge X=Y_{k}\cap
X=Y_{k}.
\]
Hence (\ref{3.0}) holds. As $\mathfrak{p}_{_{X}}=\sigma_{_{\ll^{\triangleleft
}}}(X)$ and $\ll^{\triangleleft}$ is a dual $\mathbf{R}$-order in $Q,$
(\ref{9.8}) follows from (\ref{3.0}).

Part (ii) can be proved similarly.
\end{proof}

\subsection{Relations $<_{\mathfrak{g}}$ and $\ll_{_{n}}$in sublattices of
Ln($X).$}

\begin{proposition}
\label{P9.1}\emph{Ln(}$X)$ is modular and $<_{\mathfrak{g}}$ is an
$\mathbf{HH}$-relation in \emph{Ln(}$X)$.
\end{proposition}

\begin{proof}
Let $I,J,K\in$ Ln($X)$ and $I\subseteq J.$ Then $\curlyvee_{_{I}}%
\curlywedge_{_{J}}(K)=I+(J\cap K)$ and $\curlywedge_{_{J}}\curlyvee_{_{I}%
}(K)=J\cap(I+K)$. Clearly, $I+(J\cap K)\subseteq J\cap(I+K).$ Conversely, if
$x\in J\cap(I+K)$ then $x=i+k\in J,$ where $i\in I,$ $k\in K.$ Then $i\in J,$
as $I\subseteq J.$ Hence $k\in J\cap K.$ So $x\in I+(J\cap K).$ Thus
$J\cap(I+K)\subseteq I+(J\cap K).$ Hence $I+(J\cap K)=J\cap(I+K).$ So
$\curlyvee_{_{I}}\curlywedge_{_{J}}(K)=\curlywedge_{_{J}}\curlyvee_{_{I}}(K).$
Thus Ln($X)$ is modular by (\ref{6.8}).

It follows from Corollary \ref{C6.1} that $<_{\mathfrak{g}}$ is an
$\mathbf{HH}$-relation in Ln($X)$.\bigskip
\end{proof}

To introduce more $\mathbf{HH}$-relations$,$ we will use the following result
proved in Theorem 2.2 \cite{Di}.

\begin{lemma}
\label{Clos}Let $L\subset M$ in \emph{Ln(}$X)$ and $n:=\dim(M/L)<\infty.$ Let
$K\in$ \emph{Ln(}$X).\smallskip$

\emph{(i) \ \ }If $K\subseteq M$ then $\dim\left(  K/(L\cap K)\right)
=\dim\left(  \left(  L+K\right)  /L\right)  \leq n.\smallskip$

\emph{(ii) \ }If $L\subseteq K$ then $\dim\left(  M/(M\cap K)\right)
=\dim\left(  \left(  M+K\right)  /K\right)  \leq n.\smallskip$

\emph{(iii) }If $K,L,M$ are closed subspaces\emph{,} then $L+K$ in \emph{(i)
}and $M+K$ in \emph{(ii) }are closed.
\end{lemma}

Consider the following relation in Ln($X).$ For $n\in\mathbb{N\cup\infty}$ and
$L\subseteq M$ in Ln$(X),$ we write%
\begin{equation}
L\ll_{n}M\text{ if }\dim(M/L)<n\text{.} \label{9.2}%
\end{equation}

\begin{corollary}
\label{C9.6}All $\ll_{n}$ are $\mathbf{HH}$-relations in each sublattice of
\emph{Ln(}$X)$ and $\ll_{\infty}$ is an $\mathbf{HH}$-order.
\end{corollary}

\begin{proof}
Let $L\ll_{n}M,$ i.e., dim($M/L)<n.$ Let $L\subseteq K.$ By (\ref{9.0}) and
Lemma \ref{Clos}(ii), $M\vee K=M+K$ and $\dim\left(  \left(  M+K\right)
/K\right)  \leq n.$ Hence $K\ll_{n}M\vee K.$ So, by Lemma \ref{le1}(i),
$\ll_{n}$ is an $\mathbf{H}$-relation.

Let $K\subseteq M.$ By (\ref{9.0}) and Lemma \ref{Clos}(i), $L\wedge K=L\cap
K$ and $\dim\left(  K/(L\cap K)\right)  \leq n.$ Hence $L\wedge K\ll_{n}K.$
So, by Lemma \ref{le1}(ii), $\ll_{n}$ is a dual $\mathbf{H}$-relation.
Clearly, $\ll_{\infty}$ is an $\mathbf{HH}$-order.
\end{proof}

\subsection{Relations $\ll_{_{n}},$ $<_{\mathfrak{g}},$ $\sqsubset
_{\mathfrak{g}}$ and $\prec_{\mathfrak{g}}$ in sublattices of Cl($X).$}

Note that the restrictions of $\mathbf{H}$- and dual $\mathbf{H}$-relations to
sublattices $Q$ of Cl($X)$ also have the same properties. However, an
$\mathbf{H}$-, or a dual $\mathbf{H}$-relation in $Q$ is not necessarily a
restriction of a relation in Cl($X)$ with the same property.

Making use of Lemma \ref{Clos}(iii), and repeating the proof of Corollary
\ref{C9.6}, we get

\begin{corollary}
\label{T9.2}$\{\ll_{n}\}_{n=1}^{\infty}$ are $\mathbf{HH}$-relations in all
sublattices of \emph{Cl(}$X)$ and $\ll_{\infty}$ is an $\mathbf{HH}$-order.
\end{corollary}

The relation $<_{\mathfrak{g}}$ is an $\mathbf{HH}$-relation in many
sublattices of Cl($X)$. Corollary \ref{C6.2} yields

\begin{lemma}
\label{L9.3}If $Q$ is a nest \emph{(}a complete linearly ordered set\emph{)}
in \emph{Cl(}$X)$ then $<_{\mathfrak{g}}$ is an $\mathbf{HH}$-relation.
\end{lemma}

However, there are sublattices of Cl($X),$ where $<_{\mathfrak{g}}$ is neither
an $\mathbf{H}$-, nor a dual $\mathbf{H}$-relation (see Corollary \ref{C9.1}).
The main obstacle is the fact that the sum of subspaces is not necessarily closed.

To avoid this, introduce relations $\sqsubset$ and $\prec$ in sublattices $Q$
of Cl($X).$ For $L\subseteq M$ in $Q,$ we write%
\begin{align}
L  &  \sqsubset M\text{ if }M+K\text{ is closed for each }K\in Q\text{ such
that }L\subseteq K,\nonumber\\
L  &  \prec M\text{ if }L+K\text{ is closed for each }K\in Q\text{ such that
}K\subseteq M. \label{9.4}%
\end{align}
We will show that the intersection of $<_{\mathfrak{g}}$ with $\sqsubset$ is
an $\mathbf{H}$-relation and with $\prec$ is a dual $\mathbf{H}$-relation.

\begin{proposition}
\label{P9.2}$\sqsubset$ is an $\mathbf{H}$-relation and $\prec$ is a dual
$\mathbf{H}$-relation in any sublattice $Q$ of \emph{Cl}$(X).$
\end{proposition}

\begin{proof}
Let $L\sqsubset M$ in $Q.$ For $L\subseteq K\in Q,$ let us show that
$K\sqsubset M\vee K$. Indeed, for each $R\in Q,$ $K\subseteq R,$ we have
$L\subseteq R,$ so that $(M\vee K)+R=M+K+R=M+R$ is closed. Thus $K\sqsubset
M\vee K.$ Hence $\sqsubset$ is an $\mathbf{H}$-relation by Lemma \ref{le1}(i).

Let $L\prec M$ in $Q.$ For $K\subseteq M$ in $Q,$ let us show that $L\cap
K\prec K$. Indeed, for each $R\in Q,$ $R\subseteq K,$ we have $R\subseteq M,$
so that $L\cap K+R=(L+R)\cap K$ is closed, as $L+R$ is closed. Thus $L\cap
K\prec K$. Hence $\prec$ is a dual $\mathbf{H}$-relation by Lemma
\ref{le1}(ii).\bigskip
\end{proof}

In many important sublattices $Q$ of Cl($X)$ the relations $\sqsubset,$
$\prec,$ $\subseteq$ coincide. For example, they coincide if $Q$ is a nest (a
linearly ordered set of subspaces)$,$ or a commutative subspace lattice, if
$X$ is a Hilbert space (Theorem \ref{T9.1}).

Consider another example. Let $B(X)$ the algebra of all bounded operators on
$X.$ A projection $p$ in $B(X)$ is an $L$-\textit{projection} if $\left\Vert
x\right\Vert =\left\Vert px\right\Vert +\left\Vert x-px\right\Vert $ for $x\in
X.$ The subspace $pX$ of $X$ is called an $L$-\textit{summand}. A subspace $J$
of $X$ is an $M$-\textit{ideal} if%
\begin{equation}
J^{\bot}=\{f\in X^{\ast}:f(x)=0\text{ for all }x\in J\}\text{ is an
}L\text{-summand in }X^{\ast}. \label{9,1}%
\end{equation}
Some spaces (for example, $L^{q}$-spaces, $1<q<\infty)$ have no non-trivial
$M$-ideals. In Banach algebras $M$-ideals are closed subalgebras but not
necessarily ideals \cite[Theorem V.2.3]{Har}. In C*-algebras the sets of
$M$-ideals and all closed two-sided ideals coincide.

If $I,$ $J$ are $M$-ideals of $X$ then $I+J$ is an $M$-ideal \cite[Proposition
1.11]{Har}. This yields

\begin{corollary}
\label{C3.5}If a sublattice $Q$ of \emph{Cl}$(X)$ consists of $M$-deals then
$\sqsubset$ $=$ $\prec$ $=$ $\subseteq$ in $Q.$
\end{corollary}

For a subset $S\subseteq$ Cl($X),$ Alg $S$ is the algebra of all operators in
$B(X)$ that leave all $L\in S$ invariant. For a subalgebra $\mathcal{A}$ of
$B(X),$ Lat $\mathcal{A}$\ is the lattice of all $\mathcal{A}$-invariant
subspaces of $X.$ A sublattice $Q$ of Cl($X)$ is \textit{reflexive }if $Q=$
Lat(Alg $Q)).$ Similarly, a subalgebra $\mathcal{A}$ of $B(X)$ is
\textit{reflexive} if $\mathcal{A}=$ Alg(Lat $\mathcal{A})).$ Alg $Q$ is a
reflexive subalgebra of $B(X)$ and Lat $\mathcal{A}$ is a strongly closed,
complete reflexive sublattice of Cl($X).$

For $L\subseteq M$ in Lat $\mathcal{A},$ let $p$: $M\rightarrow M/L$ be the
quotient map. Set $A_{p}p(x)=p(Ax)$ for $A\in\mathcal{A}$ and $x\in M.$ Then
$\mathcal{A}_{p}=\{A_{p}$: $A\in\mathcal{A}\}$ is a subalgebra of the algebra
$B(M/L)$ of all operators acting on $M/L.$ To proceed further we need the
following lemma.

\begin{lemma}
\label{Closed}Let $L\subset M$ in \emph{Lat} $\mathcal{A},$ let $p$\emph{:
}$M\rightarrow M/L$ and $K\in$ \emph{Lat} $\mathcal{A}.$\smallskip

\emph{(i) }Suppose that $K\subseteq M$ and $L+K$ is closed. Set $q$\emph{:
}$K\rightarrow K/(L\cap K).$ For $x\in K,$ let $Sq(x)=p(x).$ Then $p(K)$ is
closed in $p(M),$ $S$ is an invertible contraction from $q(K)$ onto $p(K)$ and%
\[
SA_{q}q(x)=A_{p}Sq(x)\text{ for all }A\in\mathcal{A}\text{ and }x\in K.
\]

\emph{(ii)} Suppose that $L\subseteq K$ and $N:=M+K$ be closed. Set $q$\emph{:
}$N\rightarrow N/K.$ For $x\in M,$ let $Tp(x)=q(x).$ Then $T$ is a contraction
from $p(M)$ onto $q(N),$ $\ker T=p(M\cap K)$ and%
\[
TA_{p}p(x)=A_{q}Tp(x)\text{ for all }A\in\mathcal{A}\text{ and }x\in M.
\]

\end{lemma}

\begin{proof}
(i) As $L+K$ is closed, $p(L+K)$ is closed. So $p(K)=p(L+K)$ is closed in
$p(M).$

If $q(x)=q(y)$ then $x-y\in L\cap K.$ So $p(x)=p(y).$ Thus $S$ is a well
defined linear operator from $q(K)$ onto $p(K)$. If $Sq(x)=0$ for $x\in K,$
then $p(x)=0.$ So $x\in L.$ Thus $x\in L\cap K$ and $q(x)=0.$ Hence $\ker
S=\{0\}.$ Moreover, $\left\Vert S\right\Vert \leq1,$ as%
\[
\left\Vert p(x)\right\Vert _{M/L}=\underset{t\in L}{\inf}\left\Vert
x+t\right\Vert \leq\underset{t\in L\cap K}{\inf}\left\Vert x+t\right\Vert
=\left\Vert q(x)\right\Vert _{K/(L\cap K)}\text{ for all }x\in K.
\]
Hence $S$ is an invertible contraction. For $A\in\mathcal{A}$ and $x\in K,$
$A_{q}q(x)=q(Ax),$ so that%
\[
SA_{q}q(x)=Sq(Ax)=p(Ax)=A_{p}p(x)=A_{p}Sq(x).
\]

(ii) As $N$ is closed, $q(M)=q(N)$ is a Banach space. If $p(x)=p(y)$ then
$x-y\in L.$ So $q(x)=q(y),$ as $L\subseteq K.$ Thus $T$ is a well defined
linear operator from $p(M)$ onto $q(N)$. If $Tp(x)=q(x)=0$ then $x\in M\cap
K.$ So $\ker T=p(M\cap K).$ Moreover, $\left\Vert T\right\Vert \leq1,$ since%
\[
\left\Vert q(x)\right\Vert _{N/K}=\underset{t\in K}{\inf}\left\Vert
x+t\right\Vert \leq\underset{t\in L}{\inf}\left\Vert x+t\right\Vert
=\left\Vert p(x)\right\Vert _{M/L}\text{ for all }x\in M.
\]
We also have $TA_{p}p(x)=Tp(Ax)=q(Ax)=A_{q}q(x)=A_{q}Tp(x)$ for $A\in
\mathcal{A}$ and $x\in M.\bigskip$
\end{proof}

Recall that a subalgebra $\mathcal{A}$ of $B(X)$ is \textit{irreducible} on
$L\in$ Lat $\mathcal{A},$ if $L$ has no non-trivial invariant subspaces. It is
\textit{algebraically irreducible}$,$ if $L$ has no non-trivial invariant
linear manifolds.

For a sublattice $Q$ of Cl($X),$ consider the following relations from
Ref($Q$) stronger than $<_{\mathfrak{g}}$:%
\begin{equation}
\sqsubset_{\mathfrak{g}}=\text{ }<_{\mathfrak{g}}\cap\text{ }\sqsubset\text{
and }\prec_{\mathfrak{g}}=\text{ }<_{\mathfrak{g}}\cap\text{ }\prec,
\label{9.6}%
\end{equation}
that is, $L\sqsubset_{\mathfrak{g}}M$ if $L<_{\mathfrak{g}}M$ and $L\sqsubset
M;$\ and $L\prec_{\mathfrak{g}}M$ if $L<_{\mathfrak{g}}M$ and $L\prec M$ for
$L,M\in Q.$

However, for some pairs $(L,M),$ $L<_{\mathfrak{g}}M$ implies $L\sqsubset
_{\mathfrak{g}}M,$ for some it implies $L\prec_{\mathfrak{g}}M.$ For example,
if $L<_{\mathfrak{g}}X$ then $L\sqsubset_{\mathfrak{g}}X;$ if
$\{0\}<_{\mathfrak{g}}M$ then $\{0\}\prec_{\mathfrak{g}}M.$

\begin{theorem}
\label{T4.5}$\sqsubset_{\mathfrak{g}}$ is an $\mathbf{H}$-relation and
$\prec_{\mathfrak{g}}$ is a dual $\mathbf{H}$-relation in each reflexive sublattice.
\end{theorem}

\begin{proof}
If a sublattice $Q$ of Cl($X)$ is reflexive then $Q=$ Lat $\mathcal{A},$ where
$\mathcal{A}=$ Alg $Q.$

Let $L\sqsubset_{\mathfrak{g}}M$ in $Q$ and $L\subseteq K\in Q.$ As
$L\sqsubset M,$ we have that $N:=K+M=K\vee M$ is invariant and closed, i.e.,
$N\in Q$. If we show that $K\sqsubset_{\mathfrak{g}}N$ then, by Lemma
\ref{le1}(i), $\sqsubset_{\mathfrak{g}}$ is an $\mathbf{H}$-relation.

By Proposition \ref{P9.2}, $\sqsubset$ is an $\mathbf{H}$-relation. Hence it
follows from Lemma \ref{le1}(i) that $K\sqsubset N.$ As $\sqsubset
_{\mathfrak{g}}=$ $<_{\mathfrak{g}}\cap$ $\sqsubset,$ we only need to prove
$K<_{\mathfrak{g}}N$. As $L\subseteq M\cap K\subseteq M$ and $L<_{\mathfrak{g}%
}M,$ either $L=M\cap K$ or $M\cap K=M.$ In the second case $M\subseteq K,$ so
that $K=N.$ Thus $K<_{\mathfrak{g}}N$ by (\ref{9.5}).

Let now $L=M\cap K$ and $p$: $M\rightarrow M/L,$ $q$: $N\rightarrow N/K$ be
the quotient maps. The operator $T$ in Lemma \ref{Closed}(ii) is a contraction
from $p(M)$ onto $q(N)$ and $\ker T=p(M\cap K)=p(L)=\{0\}$. Thus $T$ is
invertible. Moreover, $TA_{p}p(x)=A_{q}Tp(x)$ for all $A\in\mathcal{A}$ and
$x\in M.$

As $L<_{\mathfrak{g}}M,$ the algebra $\mathcal{A}_{p}$ on\thinspace$M/L$ is
irreducible. Hence, as $T$ is invertible, the algebra $\mathcal{A}_{q}$ on
$N/K$ is irreducible, i.e., $K<_{\mathfrak{g}}N.$ Thus $K\sqsubset
_{\mathfrak{g}}N,$ so that $\sqsubset_{\mathfrak{g}}$ is an $\mathbf{H}$-relation.

Similarly, one can prove that $\prec_{\mathfrak{g}}$ is a dual $\mathbf{H}%
$-relation in $Q.$
\end{proof}

\begin{corollary}
\label{C9.2}Let $Q=$ \emph{Lat }$\mathcal{A}$ be a reflexive sublattice of
\emph{Cl(}$X).\smallskip$

\emph{(i) \ }$<_{\mathfrak{g}}\cap$ $\ll_{n}$ is an $\mathbf{HH}$-relation in
$Q$ for each $1\leq n\leq\infty.\smallskip$

\emph{(ii) }Suppose that whenever $L<_{\mathfrak{g}}M$ in $Q$ then
$\mathcal{A}_{p}$ is algebraically irreducible in $M/L,$ where $p$\emph{:
}$M\rightarrow M/L.$ Then $<_{\mathfrak{g}}=$ $\prec_{\mathfrak{g}}$ is a dual
$\mathbf{H}$-relation.
\end{corollary}

\begin{proof}
(i) By Lemma \ref{Clos}, the $\mathbf{HH}$-relation $\ll_{n}$ is stronger than
$\sqsubset$ and than $\prec.$ By Theorem \ref{T4.5}, $<_{\mathfrak{g}}\cap$
$\sqsubset$ $=$ $\sqsubset_{\mathfrak{g}}$ is an $\mathbf{H}$-relation and
$<_{\mathfrak{g}}\cap$ $\prec$ $=$ $\prec_{\mathfrak{g}}$ is a dual
$\mathbf{H}$-relation. Hence, by Lemma \ref{L5.1}, $<_{\mathfrak{g}}\cap$
$\ll_{n}$ is an $\mathbf{HH}$-relation.

(ii) If $K\subset M$ in $Q$ then $L+K$ is an $\mathcal{A}$-invariant manifold
and $L\subseteq L+K\subseteq M.$ As $\mathcal{A}_{p}$ is algebraically
irreducible in $M/L,$ either $L+K=L,$ or $L+K=M.$ Hence $L+K$ is closed. Thus
$L\prec M$. So $L\prec_{\mathfrak{g}}M.$\bigskip
\end{proof}

Algebraically irreducible representations of Banach algebras were studied in
the papers of Poguntke \cite{P}, Jeu and Tomiyama \cite{JT}, Radjavi \cite{R},
Barnes \cite{Ba} and of many others.

\subsection{Superinvariant subspaces in Lat $\mathcal{A}.$}

A subspace $W\subseteq B(X)$ is a \textit{Lie subalgebra }of $B(X)$ if the
\textit{commutator }$AB-BA\in W$ for all $A,B\in W.$ Let $\mathcal{A}$ be a
closed subalgebra of $B(X)$ and $Q=$ Lat $\mathcal{A}.$ Then%
\begin{equation}
\text{Nor }\mathcal{A}=\{S\in B(X)\text{: }SA-AS\in\mathcal{A}\text{ for all
}A\in\mathcal{A\}} \label{3.1}%
\end{equation}
is a closed Lie subalgebra of $B(X)$ and $\mathcal{A}^{\prime}+\mathcal{A}%
\subseteq$ Nor $\mathcal{A},$ where $\mathcal{A}^{\prime}$ is the commutant of
$\mathcal{A}$:%
\[
\mathcal{A}^{\prime}=\{B\in B(X):AB-BA=0\text{ for all }A\in\mathcal{A}\}.
\]
A space in $Q$ is \textit{superinvariant }(\cite{K1}) if it is invariant for
all operators in Nor $\mathcal{A}.$

For $S\in B(X),$ we define the map $\theta_{_{S}}$ on Cl($X)$ by the formula
$\theta_{_{S}}(L)=e^{S}L.$ If $S\in$ Nor $\mathcal{A}$ then (see \cite[Lemmas
2.1 and 2.3]{K2}) $\theta_{_{S}}$ is an isomorphism of $Q$ and%
\begin{equation}
L\in Q\text{ is superinvariant if and only if }\theta_{_{S}}(L)=L\text{ for
all }S\in\text{ Nor }\mathcal{A}. \label{3.7}%
\end{equation}
Recall (see (\ref{5.0}) that $\theta_{_{S}}$ preserves a relation $\ll$ in $Q$
if $\theta_{_{S}}(L)\ll\theta_{_{S}}(K)\Leftrightarrow L\ll K$ for $L,K\in Q.$

\begin{proposition}
\label{P9.4}Let all isomorphisms $\theta_{_{S}},$ $S\in$ \emph{Nor
}$\mathcal{A},$ preserve a relation $\ll$ in $Q$. If $L\in Q$ is
superinvariant then\smallskip

\emph{(i) \ \ }The subspaces $s^{\ll}(L)$ and $\sigma_{_{\ll}}(L)$ in
\emph{(\ref{3.3}) }and \emph{(\ref{3.2}) }are superinvariant$.\smallskip$

\qquad In particular\emph{,} $s^{\ll}(\{0\})$ and $\sigma_{_{\ll}}(X)$ are
superinvariant.\smallskip

\emph{(ii)} \ If $\ll$ is an $\mathbf{H}$-relation then the $\ll
^{\triangleright}$-radical in $[L,X]_{_{Q}}$ is superinvariant.\smallskip

\qquad In particular\emph{, }the $\ll^{\triangleright}$-radical in $Q$ is
superinvariant.\smallskip

\emph{(iii) }If $\ll$ is a dual $\mathbf{H}$-relation then the dual
$\ll^{\triangleleft}$-radical in $[\{0\},L]_{_{Q}}$ is
superinvariant.\smallskip

\qquad In particular\emph{, }the dual $\ll^{\triangleleft}$-radical in $Q$ is superinvariant.
\end{proposition}

\begin{proof}
(i) As $\theta_{_{S}}$ preserve $\ll$ and as $\theta_{_{S}}(L)=L$ for all
$S\in$ Nor $\mathcal{A}$ by (\ref{3.7}), Proposition \ref{T3.11} gives
$\theta_{_{S}}(s^{\ll}(L))=s^{\ll}(L)$ and $\theta_{_{S}}(\sigma_{_{\ll}%
}(L))=\sigma_{_{\ll}}(L).$ Then, by (\ref{3.7}), $s^{\ll}(L)$ and
$\sigma_{_{\ll}}(L)$ are superinvariant. Parts (ii) and (iii) have similar
proofs and follow from Proposition \ref{T3.11}.\bigskip
\end{proof}

Combining Theorem \ref{T4.6}, Propositions \ref{P9.5} and \ref{P9.4} and
(\ref{3.4}), we obtain

\begin{corollary}
\label{C9.4}Let $\ll$ be a relation in $Q=$ \emph{Lat }$\mathcal{A}$ and let
all $\theta_{_{S}},$ $S\in$ \emph{Nor }$\mathcal{A},$ preserve $\ll
.\smallskip$

\emph{(i) \ }\ Suppose that $\mathcal{A}$ has no non-trivial superinvariant
subspaces.\smallskip

\qquad If $[\ll,X]_{_{Q}}\neq\{X\}$ then $\sigma_{_{\ll}}(X)=\{0\}.$ If
$[\{0\},\ll]_{_{Q}}\neq\{0\}$ then $s^{\ll}(\{0\})=X.\smallskip$

\emph{(ii) }Let $\ll$ be a dual $\mathbf{H}$-relation. If $\sigma_{_{\ll}%
}(X)=\{0\}$ then the dual $\ll^{\triangleleft}$-radical $\mathfrak{p}=\{0\}$
and\smallskip

$\qquad1)$ for each $L\neq\{0\}$ in $Q,$ there is $K\subset L$ in $Q$ such
that $K\ll L.$\smallskip

$\qquad2)$ If\emph{ }$X$ is separable\emph{,} there are $...\ll Y_{n}\ll...\ll
Y_{1}\ll X$ in $Q$ such that $\cap_{n=1}^{\infty}Y_{n}=\{0\}.\smallskip$

\emph{(iii) }Let $\ll$ be an $\mathbf{H}$-relation. If $s^{\ll}(\{0\})=X$ then
the $\ll^{\triangleright}$-radical $\mathfrak{r}=X$ and\smallskip

$\qquad1)$ for each $L\neq X$ in $Q,$ there is $M\supset L$ in $Q$ such that
$L\ll M.$

$\qquad2)$ If $X$ is separable\emph{,} there are $\{0\}\ll Y_{1}\ll...\ll
Y_{n}\ll...$ in $Q$ such that $\overline{\sum{}_{n=1}^{\infty}Y_{n}%
}=X.\smallskip$
\end{corollary}

\begin{corollary}
\label{C9.5}Isomorphisms $\theta_{_{S}},$ $S\in$ \emph{Nor }$\mathcal{A},$
preserve the relations $<_{\mathfrak{g}},$ $\sqsubset,$ $\prec,$ $\ll_{n}$ for
$n\in\mathbb{N}\cup\infty.$ Thus all results of Propositions \emph{\ref{P9.5}}
and \emph{\ref{P9.4}} and Corollary \emph{\ref{C9.4}} hold for these relations.
\end{corollary}

\begin{proof}
For $S\in$ Nor $\mathcal{A}$ and each $R\in Q,$ set $R_{_{S}}=e^{S}R.$ Let
$L<_{\mathfrak{g}}M$ in $Q.$ If $L_{_{S}}\subsetneqq K\subsetneqq M_{_{S}}$
for some $K\in Q,$ then $L\subsetneqq K_{_{-S}}\subsetneqq M$ and $K_{_{-S}%
}\in Q,$ a contradiction. Thus $L_{_{S}}<_{\mathfrak{g}}M_{_{S}}.$

Let $L\sqsubset M.$ Let $L_{_{S}}\subset K\in Q.$ Then $L\subset K_{_{-S}},$
so that $M+K_{_{-S}}$ is closed. Hence $(M+K_{_{-S}})_{_{S}}=M_{_{S}}+K$ is
closed. So $L_{_{S}}\sqsubset M_{_{S}}.$ Conversely, if $L_{_{S}}\sqsubset
M_{_{S}}$ then, by above, $L=(L_{_{S}})_{_{-S}}\sqsubset(M_{_{S}})_{_{-S}}=M.$
So $\theta_{_{S}}$ preserves the relation $\sqsubset.$ Similarly,
$\theta_{_{S}}$ preserves the relation $\prec.$

If $L\subset M$ then $\dim(M/L)=\dim(M_{_{S}}/L_{_{S}})$. So $\theta_{_{S}}$
preserve all relations $\ll_{n},$ $n\in\mathbb{N}\cup\infty.\bigskip$
\end{proof}

Let us consider some cases where Corollary \ref{C9.4} can be applied.

\begin{proposition}
\label{P9.6}Let an operator $T\in B(X)$ have eigenvectors $\{e_{\lambda
}\}_{\lambda\in\Lambda}$.\smallskip

\emph{(i)\ }If $X=\overline{\emph{span}(e_{\lambda})_{\lambda\in\Lambda}}$
then the $\ll_{_{1}}^{\triangleright}$-radical $\mathfrak{r}=s^{\ll_{_{1}}%
}(\{0\})=X$ in \emph{Lat }$T$ and all results of Corollary \emph{\ref{C9.4}%
(iii)} hold for the $\mathbf{HH}$-relation $\ll_{_{1}}$ \emph{(}see
\emph{(\ref{9.2})).}\medskip

\emph{(ii) }Set $X_{\mu}=$ $\overline{\emph{span}(e_{\lambda})_{\mu\neq
\lambda\in\Lambda}}$ for $\mu\in\Lambda.$ If $\cap_{\mu\in\Lambda}X_{\mu
}=\{0\}$ then the dual $\ll_{_{1}}^{\triangleleft}$-radical $\mathfrak{p}%
=\sigma_{\ll_{_{1}}}(X)=\{0\}$ in \emph{Lat }$T$ and all results of Corollary
\emph{\ref{C9.4}(ii)} hold for the $\mathbf{HH}$-relation $\ll_{_{1}}$\emph{.}
\end{proposition}

\begin{proof}
(i) As $\{0\}\ll_{_{1}}\mathbb{C}e_{\lambda},$ we have $\mathbb{C}e_{\lambda
}\in\lbrack\{0\},\ll_{_{1}}]$ (see (\ref{3.3})) for all $\lambda\in\Lambda.$
By (\ref{3.2}), $X=\vee\mathbb{C}e_{\lambda}=s^{\ll_{_{1}}}(\{0\}).$ By
Corollary \ref{C9.5}, $\mathfrak{r}=X$ and Corollary \ref{C9.4}(iii) holds for
$\ll_{_{1}}$.

Part (ii) has a similar proof.\bigskip
\end{proof}

To illustrate Proposition \ref{P9.6}, consider the following example.

\begin{example}
\emph{Let }$X$\emph{ be a Hilbert space with basis }$\{e_{n}\}_{n=1}^{\infty}%
$\emph{ and }$T$\emph{ be the adjoint to the unilateral shift: }%
$Te_{n}=e_{n-1}.$\emph{ Let }$\Lambda=\{\lambda\in\mathbb{C}$\emph{:
}$\left\vert \lambda\right\vert <1\}.$\emph{ For each }$\lambda\in\Lambda
,$\emph{ }$e_{\lambda}=(1,\lambda,\lambda^{2},...)\in X$\emph{ is an
eigenvector of }$T$\emph{: }$Te_{\lambda}=\lambda e_{\lambda},$ \emph{and
}$X=\overline{\text{\emph{span}}(e_{\lambda})_{\lambda\in\Lambda}}.$\emph{
Thus all results of Corollary \ref{C9.4}(iii) hold. }$\blacksquare$
\end{example}

Let the commutant $S^{\prime}$ of a set $S\subset B(X)$ contain projections
$\{P_{\lambda}\}_{\lambda\in\Lambda},$ $\dim P_{\lambda}<\infty.$ Set $Q=$ Lat
$S.$ All subspaces $X_{\lambda}=P_{\lambda}X$ belong to $\left[
\{0\},\ll_{_{\infty}}\right]  _{_{Q}}$ and all $(\mathbf{1}-P_{\lambda})X$ to
$\left[  \ll_{_{\infty}},X\right]  _{_{Q}}.$ As in Proposition \ref{P9.6}, we obtain

\begin{proposition}
\label{P4.6}\emph{(i) }If $\overline{\sum{}_{\lambda\in\Lambda}X_{\lambda}}=X$
then the $\ll_{_{\infty}}^{\triangleright}$-radical $\mathfrak{r}%
=s^{\ll_{_{\infty}}}(\{0\})=X$ in $Q$ and all results of Corollary
\emph{\ref{C9.4}(iii)} hold for $\ll_{_{\infty}}$.\smallskip

\emph{(ii) }If $\cap_{\lambda\in\Lambda}(\mathbf{1}-P_{\lambda})X=\{0\}$ then
the dual $\ll_{_{\infty}}^{\triangleleft}$-radical $\mathfrak{p}=\sigma
_{\ll_{_{\infty}}}(X)=\{0\}$ in $Q$ and all results of Corollary
\emph{\ref{C9.4}(ii)} hold for $\ll_{_{\infty}}$.
\end{proposition}

For example, suppose that the commutant $S^{\prime}$ contains a compact
operator $T.$ Then Sp$(T)=\{\lambda_{n}\}_{n=0}^{m}$ for $m\leq\infty,$ where
$\lambda_{0}=0.$ The algebra $S^{\prime}$ contains projections $P_{n},$ $1\leq
n\leq m,$ such that $\dim P_{n}<\infty.$ The subspaces $X_{n}=P_{n}X$ and
$V_{n}=(1-P_{n})X$ are $T$-invariant$,$
\begin{equation}
X=X_{n}\dotplus V_{n},\text{ \ Sp}(T|_{X_{n}})=\{\lambda_{n}\}\text{ and
Sp}(T|_{V_{n}})=\text{Sp}(T)\diagdown\{\lambda_{n}\}\text{ for }1\leq n\leq m.
\label{5.2}%
\end{equation}
Moreover, $X_{n}\subseteq V_{k}$ for $n\neq k.$ So, for each $1\leq k\leq m,$
$X=Y_{k}\dotplus\sum_{n=1}^{k}\dotplus X_{n},$ where $Y_{k}=\cap_{n=1}%
^{k}V_{k},$ and all subspaces in the decompositions are $S$-invariant.
Thus\medskip

$1)$ if $\overline{\sum{}_{n\geq1}X_{n}}=X$ then all results of Corollary
\ref{C9.4}(iii) hold for the relation $\ll_{_{\infty}}$;\smallskip

$2)$ if $\cap_{n\geq1}V_{n}=\{0\}$ then all results of Corollary
\ref{C9.4}(ii) hold for the relation $\ll_{_{\infty}}$.

\begin{proposition}
The conditions $\overline{\sum{}_{n}X_{n}}=X$ and $\cap_{n}V_{n}=\{0\}$ above
are not equivalent.
\end{proposition}

\begin{proof}
To show that $\overline{\sum{}_{n\geq1}X_{n}}=X$ $\not \Longrightarrow
\cap_{n\geq1}V_{n}=\{0\},$ consider the compact operator $T$ constructed in
\cite{Ham} (see also \cite[p. 262]{N}) on a Hilbert space $X.$ It has the
following properties:
\[
\text{Sp}(T)=\{\lambda_{n}\}_{n=0}^{\infty},\text{ }\overline{\sum{}_{n\geq
1}X_{n}}=X,\text{ where }X_{n}=\text{Ker}(T-\lambda_{n}\mathbf{1}),
\]
and there exists a $T$-invariant subspace $E\neq\{0\}$ satisfying $T|_{E}%
\neq0$ and Sp$(T|_{E})=0.$ Let $P_{n}$ be the projections on $X_{n}.$ Then
(\ref{5.2}) hold.

As the spectral radius $r(T|_{E})=0,$ we have, for each $e\in E,$%
\begin{equation}
\left\Vert T^{k}e\right\Vert ^{1/k}\leq\left\Vert (T|_{E})^{k}\right\Vert
^{1/k}\longrightarrow r(T|_{E})=0,\text{ as }k\longrightarrow\infty.
\label{5.5}%
\end{equation}
Let $e\in E$ and $n\geq1.$ By (\ref{5.2}), $e=x+y,$ where $x\in X_{n}$ and
$y\in V_{n}.$ As $T$ and $P_{n}$ commute,
\[
\left\Vert T^{k}x\right\Vert ^{1/k}=\left\Vert T^{k}P_{n}e\right\Vert
^{1/k}=\left\Vert P_{n}T^{k}e\right\Vert ^{1/k}\leq\left\Vert P_{n}\right\Vert
^{1/k}\left\Vert T^{k}e\right\Vert ^{1/k}\overset{(\ref{5.5})}{\longrightarrow
}0,
\]
as $k\rightarrow\infty.$ As $\dim X_{n}<\infty$ and Sp$(T|_{X_{n}}%
)=\lambda_{n}\neq0,$ it is easy to prove that $x=0.$ Thus $e\in V_{n}.$ Hence
$\{0\}\neq E\subseteq\cap_{n\geq1}V_{n}.$ Thus $\overline{\sum{}_{n}X_{n}}=X$
does not imply $\cap_{n}V_{n}=\{0\}.$

To prove conversely that $\cap_{n}V_{n}=\{0\}$ $\not \Longrightarrow $
$\overline{\sum{}_{n}X_{n}}=X,$ consider $A=T^{\ast}.$ Then Sp($A)=\{\lambda
_{n}^{\ast}\}_{n=0}^{\infty}.$ As in (\ref{5.2}), let $Y_{n}$ and $U_{n}$ be
$A$-invariant subspaces satisfying%
\[
X=Y_{n}\dotplus U_{n},\text{ Sp}(A|_{Y_{n}})=\{\lambda_{n}^{\ast}\},\text{
}\dim Y_{n}<\infty\text{ and Sp}(A|_{U_{n}})=\text{Sp}(A)\diagdown
\{\lambda_{n}^{\ast}\},
\]
for each $n.$ Set $U=\cap_{n\geq1}U_{n}.$ Since Sp($A|_{U})=\{0\},$ we have as
in (\ref{5.5}) that $\left\Vert A^{k}u\right\Vert ^{1/k}\longrightarrow0,$ as
$k\rightarrow\infty,$ for each $u\in U.$ Let $x\in X_{n}$ and $Tx=\lambda
_{n}x$ for $n\geq1.$ Then for all $u\in U,$%
\begin{equation}
\left\vert \lambda_{n}\right\vert \left\vert (x,u)\right\vert ^{1/k}%
=\left\vert (T^{k}x,u)\right\vert ^{1/k}=\left\vert (x,A^{k}u)\right\vert
^{1/k}\leq\left\Vert x\right\Vert ^{1/k}\left\Vert A^{k}u\right\Vert
^{1/k}\longrightarrow0, \label{e3.1}%
\end{equation}
as $k\rightarrow\infty.$ Hence $x\bot U.$ If $Tx=\lambda_{n}x+z$ and
$Tz=\lambda_{n}z,$ then $z\bot U$ and, by (\ref{e3.1}), $x\bot U.$ As $\dim
X_{n}<\infty,$ continuing this$,$ we get that all $X_{n}\bot U.$ As
$\overline{\sum{}_{n\geq1}X_{n}}=X,$ we have $U=\{0\}.$

On the other hand, let $y\in Y_{n}$ be such that $Ay=\lambda_{n}^{\ast}y.$
Then, for each $e\in E,$%
\begin{equation}
\left\vert \lambda_{n}^{\ast}\right\vert \left\vert (y,e)\right\vert
^{1/k}=\left\vert (A^{k}y,e)\right\vert ^{1/k}=\left\vert (y,T^{k}%
e)\right\vert ^{1/k}\leq\left\Vert y\right\Vert ^{1/k}\left\Vert
T^{k}e\right\Vert ^{1/k}\overset{(\ref{5.5})}{\longrightarrow}0, \label{3.2e}%
\end{equation}
as $k\rightarrow\infty.$ Hence $y\bot E.$ If $Ay=\lambda_{n}^{\ast}y+z$ and
$Az=\lambda_{n}^{\ast}z,$ then $z\bot U$ and, by (\ref{3.2e}), $y\bot U.$ As
$\dim Y_{n}<\infty,$ continuing this, we get that all $Y_{n}\bot E,$ so that
$\overline{\sum{}_{n\geq1}Y_{n}}\neq X.$ Thus $\cap_{n\geq1}U_{n}=\{0\}$ does
not imply $\overline{\sum{}_{n\geq1}Y_{n}}=X.$
\end{proof}

\subsection{Relation $<_{\mathfrak{g}}$ in sublattices of Cl($H),$ where $H$
is a Hilbert space}

In this section $X=H$ is a Hilbert space. First we consider reflexive
sublattices of Cl($H),$ $\dim$ $H=\infty,$ where $<_{\mathfrak{g}}$ is neither
$\mathbf{H}$-, nor a dual $\mathbf{H}$-relation.

For $x,y\in H,$ define the rank one operator $x\otimes y$ on $H$ by%
\begin{equation}
(x\otimes y)\xi=(\xi,x)y\text{ for }\xi\in H\mathfrak{.} \label{d}%
\end{equation}
Let $F$ be a closed symmetric operator on $H$ with domain $D(F)$ and $F^{\ast
}$ be its adjoint. Let $F^{\ast}\neq F.$ Then $D(F)\subsetneqq D(F^{\ast}).$
For $u,v\in H$, $y\in D(F)$ and $x\in D(F^{\ast}),$%
\begin{equation}
(x\otimes y)(u\otimes v)=(v,x)(u\otimes y)\text{ and }F(x\otimes y)=x\otimes
Fy,\text{ }(x\otimes y)F=F^{\ast}x\otimes y. \label{f}%
\end{equation}
Set $X=H\oplus H$ and let $\Omega=\{\{0\},H\oplus\{0\},\{0\}\oplus
H\mathfrak{,}X\}.$ Consider%
\[
\mathcal{A}=\text{span}\left\{  A(x,y)=%
\begin{pmatrix}
x\otimes Fy & 0\\
0 & Fx\otimes y
\end{pmatrix}
\text{ for }x,y\in D(F)\right\}  .
\]
By (\ref{f}), $\mathcal{A}$ is a subalgebra of $B(X)$. For each closed
operator $S$ on $H$ with domain $D(S),$%
\[
M_{_{S}}=\left\{  \widehat{\xi}=S\xi\oplus\xi:\xi\in D(S)\right\}  \text{ is a
closed subspace of }X.
\]

\begin{lemma}
\label{L9.1}Let $\overline{FD(F)}=H.$ Then \emph{Lat} $\mathcal{A}=\Omega
\cup(\cup_{t\in\mathbb{C}\backslash\{0\}}\mathcal{K}_{t})$\emph{ }and\emph{,}
for each $t,$%
\[
\mathcal{K}_{t}=\{M_{_{tT}}\text{\emph{:} }F\subseteq T\subseteq F^{\ast
},\text{ }T\text{ is closed}\}=\{L\text{\emph{:} }L\text{ is a space},\text{
}M_{_{tF}}\subseteq L\subseteq M_{_{tF^{\ast}}}\}.
\]

\end{lemma}

\begin{proof}
Let $L\in$ Lat $\mathcal{A}$ and $L\notin\Omega.$ As $FD(F)$ is dense in
$H\mathfrak{,}$ the subspaces $H\oplus\{0\}$ and $\{0\}\oplus H$ have no
non-trivial $\mathcal{A}$-invariant subspaces$.$ Hence $L=M_{_{S}}$ for some
closed operator $S\neq0$ on $H\mathfrak{.}$ As $A(x,y)\widehat{\xi}\subseteq
M_{S}$ for all $x,y\in D(F)$ and $\xi\in D(S)$,%
\[
\left(  Fx\otimes y\right)  \xi\overset{(\ref{d})}{=}\sum_{i=1}^{n}%
(\xi,Fx)y\in D(S)\text{ and }\left(  x\otimes Fy\right)  S=S\left(  Fx\otimes
y\right)  \overset{(\ref{f})}{=}Fx\otimes Sy.
\]
As $FD(F)$ is dense in $H\mathfrak{,}$ we have $D(F)\subseteq D(S).$ As
$\left(  x\otimes Fy\right)  S=Fx\otimes Sy,$ we have%
\begin{equation}
\left(  x\otimes Fy\right)  S\eta\overset{(\ref{d})}{=}(S\eta,x)Fy=(Fx\otimes
Sy)\eta\overset{(\ref{d})}{=}(\eta,Fx)Sy\text{ for }\eta\in D(S). \label{3.1b}%
\end{equation}

Fix $\eta$ and choose $x$ such that $(\eta,Fx)\neq0.$ Then $Sy=tFy$ for $y\in
D(F)$ with $t=(S\eta,x)/(\eta,Fx)\in\mathbb{C}.$ As $0\neq S$ is closed and
$D(F)\subseteq D(S),$ we have $t\neq0.$ Thus $tF\subseteq S.$

Choose now $y\in D(F)$ in (\ref{3.1b}) such that $Fy\neq0.$ Then
$(S\eta,x)=t(\eta,Fx)$ for all $x\in D(F)$ and $\eta\in D(S).$ Hence $\eta\in
D(F^{\ast})$ and $S\eta=tF^{\ast}\eta.$ Thus $S\subseteq tF^{\ast}$ and
$tF\subseteq S\subseteq tF^{\ast}.$ Set $T=S/t.$ Then $M_{_{S}}=M_{_{tT}}$ and
$F\subseteq T\subseteq F^{\ast}.$

Conversely, let $F\subseteq T\subseteq F^{\ast}.$ Then, for $A(x,y)\in
\mathcal{A},$ $\xi\in D(T)$ and $t\in\mathbb{C}\backslash\{0\},$ we have%
\[
A(x,y)(tT\xi\oplus\xi)\overset{(\ref{d})}{=}t(T\xi,x)Fy\oplus(\xi
,Fx)y=(T\xi,x)(tTy\oplus y)\in M_{_{tT}}.
\]
Thus $M_{_{tT}}\in$ Lat $\mathcal{A}$ for all $t\in\mathbb{C}\backslash\{0\}.$

If $L$ is a subspace of $X$ and $M_{_{tF}}\subseteq L\subseteq M_{_{tF^{\ast}%
}},$ then there is a linear operator $T$ such that $D(F)\subseteq
D(T)\subseteq D(F^{\ast})$ and $L=M_{_{tT}}.$ As $L$ is closed, $T$ is closed.
Thus $L\in\mathcal{K}_{t}.$
\end{proof}

\begin{corollary}
\label{C9.1}The sublattice \emph{Lat }$\mathcal{A}$ of \emph{Cl(}$H)$ in Lemma
\emph{\ref{L9.1} }is not modular and the relation $<_{\mathfrak{g}}$ in
\emph{Lat }$\mathcal{A}$ is neither an $\mathbf{H}$-$,$ nor a dual
$\mathbf{H}$-relation$.$
\end{corollary}

\begin{proof}
As $\{0\},H\oplus\{0\},X,M_{_{F^{\ast}}},M_{_{F}}$ form a pentagon in Lat
$\mathcal{A}$ (see (\ref{6,0})), Lat $\mathcal{A}$ is not modular.

By Lemma \ref{L9.1}, $\{0\}<_{\mathfrak{g}}H\oplus\{0\}<_{\mathfrak{g}}X.$ If
$<_{\mathfrak{g}}$ is an $\mathbf{H}$-relation then, by Lemma \ref{le1}(iii),
$M_{_{F}}=\{0\}\vee M_{_{F}}<_{\mathfrak{g}}(H\oplus\{0\})\vee M_{_{F}%
}=\overline{(H\oplus\{0\})+M_{_{F}}}=X.$ However, $M_{_{F}}\not <
_{\mathfrak{g}}X,$ since $M_{_{F}}\subsetneqq M_{_{F^{\ast}}}\subsetneqq X.$
Thus $<_{\mathfrak{g}}$ is not an $\mathbf{H}$-relation.

If $<_{\mathfrak{g}}$ is a dual $\mathbf{H}$-relation then $\{0\}=(H\oplus
\{0\})\cap M_{_{F^{\ast}}}<_{\mathfrak{g}}X\cap M_{_{F^{\ast}}}=M_{_{F^{\ast}%
}}$ by Lemma \ref{le1}(ii). However, $\{0\}\not <  _{\mathfrak{g}}%
M_{_{F^{\ast}}},$ since $\{0\}\subsetneqq M_{_{F}}\subsetneqq M_{_{F^{\ast}}%
}.$ Thus $<_{\mathfrak{g}}$ is not a dual $\mathbf{H}$-relation.\bigskip
\end{proof}

As sublattices of a modular lattice are modular, it follows from Corollary
\ref{C9.1} that Cl($H)$ is not modular if $\dim H=\infty$. On the other hand,
Cl($H)=$ Ln($H)$ if $\dim H<\infty$, so that Cl($H)$ is modular by Proposition
\ref{P9.1}. This yields (cf. Proposition 1.5.5 \cite{K}).

\begin{corollary}
\emph{Cl(}$H)$ is a modular lattice if and only if $\dim(H)<\infty$.
\end{corollary}

We will now consider some sublattices of Cl($H),$ where $<_{\mathfrak{g}}$ is
an $\mathbf{HH}$-relation. Denote by $P(H)$ the set of all orthogonal
projections in $B(H).$ For $L\in$ Cl($H),$ $p_{_{L}}$ denotes the projection
on\thinspace$L.$ So Cl($H)$ can be identified with $P(H)$ and each sublattice
of Cl($H)$ with a sublattice of $P(H).$

A \textit{complete sublattice} $Q$ of Cl($H)$ is a \textit{commutative
subspace lattice} (CSL) if all projections $p_{_{L}},$ $L\in Q,$ commute. For
$p\in P(H),$ let $p^{\bot}=\mathbf{1}-p.$ If $Q$ is CSL and $p,q\in Q$ then%
\begin{equation}
p\wedge q=pq,\text{ }p\vee q=p+q-pq=p+p^{\bot}q. \label{9.11}%
\end{equation}

\begin{theorem}
\label{T9.1}Let $Q$ be a \emph{CSL} in \emph{Cl(}$H).$ Then\smallskip

\emph{(i) \ \ }$Q$ has properties \emph{(JID)} and \emph{(MID)} \emph{(}see
\emph{(\ref{6.9}))}$,$ so that it is modular.\smallskip

\emph{(ii) \ }$L+K\in Q$ for all $L,K\in Q$ and $<_{\mathfrak{g}}$ is an
$\mathbf{HH}$-relation in $Q.\smallskip$

\emph{(iii)} $\sqsubset$ $=$ $\prec$ $=$ $\subseteq$ and $\sqsubset
_{\mathfrak{g}}$ $=$ $\prec_{\mathfrak{g}}$ $=$ $<_{\mathfrak{g}}$ in
$Q.\smallskip$

\emph{(iv)\ }$<_{\mathfrak{g}}^{\triangleleft}$ is an $\mathbf{H}$-order and a
dual $\mathbf{R}$-order$;$ $<_{\mathfrak{g}}^{\triangleright}$ is an
$\mathbf{R}$-order and a dual $\mathbf{H}$-order\emph{.\smallskip}

\emph{(v) \ }$\overline{<_{\mathfrak{g}}}=$ $\widetilde{<_{\mathfrak{g}}}$
\emph{(}see \emph{(\ref{6.4}) }and \emph{(\ref{6.5})) }is an $\mathbf{RR}$-order.

\emph{(vi) }$Q$ is a union of disjoint intervals that have no gaps.
\end{theorem}

\begin{proof}
(i) It was noted in \cite[p. 357]{Da} that $Q$ has properties (JID) and (MID).
For completeness, we prove this here directly. Let $G\subseteq Q$ and $z\in
Q.$ Set $q=\vee G$ and $p=\vee\{zx$: $x\in G\}.$

As $Q$ is a complete sublattice of Cl($H),$ $q\in Q.$ So%
\begin{equation}
p=\vee\{zx\text{: }x\in G\}\overset{(\ref{9.11})}{=}\vee\{z\wedge x\text{:
}x\in G\}\overset{(\ref{6.0})}{\leq}z\wedge(\vee G)=z\wedge
q\overset{(\ref{9.11})}{=}zq. \label{8.8}%
\end{equation}
We have $x=zx\oplus z^{\bot}x\leq p\oplus z^{\bot}q$ for all $x\in G.$ Hence
$q=\vee\{x$: $x\in G\}\leq p\oplus z^{\bot}q\overset{(\ref{8.8})}{\leq
}zq\oplus z^{\bot}q=q.$ So $zq=p,$ i.e., $z\wedge(\vee G)=\vee\{z\wedge x$:
$x\in G\}.$ Thus $Q$ has property (MID) (see (\ref{6.9})).

If $G=\{x,y\}$ then $z\wedge(x\vee y)=(z\wedge x)\vee(z\wedge y).$ Thus $Q$ is
distributive. So it is modular.

Set now $r=\wedge G$ and $s=\wedge\{z^{\bot}x$: $x\in G\}.$ As $Q$ is
complete$,$ $r\in Q.$ So%
\begin{align}
z\oplus z^{\bot}r\overset{(\ref{9.11})}{=}z\vee r  &  =z\vee(\wedge
G)\overset{(\ref{6.0})}{\leq}\wedge\{z\vee x\text{: }x\in
G\}\overset{(\ref{9.11})}{=}\wedge\{z\oplus z^{\bot}x\text{: }x\in
G\}\nonumber\\
&  =z\oplus\wedge\{z^{\bot}x\text{: }x\in G\}=z\oplus s. \label{9.12}%
\end{align}
So $z^{\bot}r\leq s.$ For $x\in G,$ we have $x=zx\oplus z^{\bot}x\geq zr\oplus
s.$ Therefore $r=\wedge\{x$: $x\in G\}\geq zr\oplus s.$ So $z^{\bot}r\geq s.$
Hence $z^{\bot}r=s,$ so that $z\oplus z^{\bot}r=z\oplus s.$ Thus
\[
z\vee(\wedge G)\overset{(\ref{9.11})}{=}z\oplus z^{\bot}r=z\oplus
s\overset{(\ref{9.12})}{=}\wedge\{z\vee x:x\in G\}.
\]
Therefore (see (\ref{6.9})) $Q$ has properties (JID).

(ii) Let $a=p_{_{L}},b=p_{_{K}}$ for $L,K\in Q.$ Then $L+K=L\oplus a^{\bot
}K=L\oplus a^{\bot}bH.$ As $Q$ is CSL, $a$ and $b$ commute. So $a^{\bot}b$ is
a projection$.$ Thus $a^{\bot}bH$ is a closed subspace of $K^{\bot}.$ So $L+K$
is closed.

As $Q$ is modular, it follows from Corollary \ref{C6.1} that $<_{\mathfrak{g}%
}$ is an $\mathbf{HH}$-relation in $Q.$

(iii) As $L+K\in Q$ for all $L,K\in Q,$ it follows from (\ref{9.4}) that
$\sqsubset$ $=$ $\prec$ $=$ $\subseteq.$ Hence, since $<_{\mathfrak{g}}=$
$<_{\mathfrak{g}}\cap$ $\subseteq,$ $\sqsubset_{\mathfrak{g}}=$
$<_{\mathfrak{g}}\cap$ $\sqsubset$ and $\prec_{\mathfrak{g}}=$
$<_{\mathfrak{g}}\cap$ $\prec,$ we have $<_{\mathfrak{g}}=$ $\sqsubset
_{\mathfrak{g}}$ $=$ $\prec_{\mathfrak{g}}.$

As $Q$ is modular, has properties (JID) and (MID) (see (\ref{6.9})) and
$<_{\mathfrak{g}}$ is an $\mathbf{HH}$-relation in $Q,$ (iv) and (v) follow
from Lemma \ref{L6.2} and Corollary \ref{C8.1}. Part (vi) follows from Theorem
\ref{T6.2}.\bigskip
\end{proof}

The lattice of all finite projections in a W*-algebra is modular (Theorem
V.1.37 \cite{T}). So Corollary \ref{C6.1} yields

\begin{corollary}
\label{C9.3}$<_{\mathfrak{g}}$ is an $\mathbf{HH}$-relation in the lattice of
all finite projections in any W$^{\ast}$-algebra.
\end{corollary}

The lattice Cl($H)\approx P(H)$ has a large variety of $\mathbf{H}$- and dual
$\mathbf{H}$-relations which is difficult to describe. However, it is possible
to describe all $\mathbf{HH}$-relations in $P(H)$. To do this, recall that
projections $p,q\in P(H)$ are equivalent ($p\sim q)$ if%
\begin{equation}
p=vv^{\ast}\text{ and }q=v^{\ast}v\text{ for some partial isometry }v.
\label{9.3}%
\end{equation}
To introduce a new class of $\mathbf{HH}$-relations in $P(H)$, we need the
following result.

\begin{proposition}
\label{P9.3}\emph{\cite{K3} }Let $p\leq q$ in $P(H)$. Then$,$ for each
projection $r\in P(H),$%
\begin{align}
(q-p)^{\bot}  &  \leq(q\vee r-p\vee r)^{\bot};\label{5,1}\\
(q-p)^{\bot}  &  \sim\mathfrak{z}\text{\emph{ }for some }\mathfrak{z}\in
P(H),\text{ and }\mathfrak{z}\leq(q\wedge r-p\wedge r)^{\bot}. \label{5,2}%
\end{align}

\end{proposition}

Consider now the reflexive relations $\ll_{_{n}}^{^{\bot}}$in $P(H),$ for
$0\leq n\leq\infty,$ defined by the condition%
\begin{equation}
p\ll_{_{n}}^{^{\bot}}q\text{ if }p\leq q\text{ in }P(H)\text{ and }%
\dim(q-p)^{\bot}\geq n,\text{ so that }\ll_{_{0}}^{^{\bot}}=\text{ }\leq.
\label{5,7}%
\end{equation}

\begin{theorem}
\label{T9.3}The relations $\ll_{_{n}}^{^{\bot}},$ $0\leq n\leq\infty,$ are
$\mathbf{HH}$-relations in $P(H).$
\end{theorem}

\begin{proof}
Let $p,q,r\in Q$ and $p\ll_{_{n}}^{^{\bot}}q$ for some $n.$ Then $p\leq q$
and\ $\dim(q-p)^{\bot}\geq n.$

We have that $p\vee r,q\vee r$ belong to $Q$ and $p\vee r\leq q\vee r.$ By
(\ref{5,1}), $(q-p)^{\bot}\leq(q\vee r-p\vee r)^{\bot}.$ So $\dim(q\vee
r-p\vee r)^{\bot}\geq n.$ Hence $p\vee r\ll_{_{n}}^{^{\bot}}q\vee r$ by
(\ref{5,7}).

We have $p\wedge r,q\wedge r\in Q$ and $p\wedge r\leq q\wedge r.$ By
(\ref{5,2}), $(q-p)^{\bot}\sim\mathfrak{z}$\emph{ }and $\mathfrak{z}%
\leq(q\wedge r-p\wedge r)^{\bot}$ for some projection $\mathfrak{z}$ in
$P(H).$ It follows from (\ref{9.3}) that $\dim(q-p)^{\bot}=\dim\mathfrak{z.}$
Hence $\dim(q\wedge r-p\wedge r)^{\bot}\geq n.$ So $p\wedge r\ll_{_{n}%
}^{^{\bot}}q\wedge r$ by (\ref{5,7}). Thus, by Definition \ref{D2.2}
$\ll_{_{n}}^{^{\bot}}$ is an $\mathbf{HH}$-relation.\bigskip
\end{proof}

By Corollary \ref{T9.2}, all $\ll_{n},$ $1\leq n\leq\infty,$ are $\mathbf{HH}%
$-relations in Cl($H)\approx P(H)$. The following result obtained
\cite{K3}\emph{ }in describes all $\mathbf{HH}$-relation in $P(H)$.

\begin{theorem}
\label{T6.3}For a separable $H,$ each $\mathbf{HH}$-relation in $P(H)$ is
either $\ll_{_{n}},$ or $\ll_{_{n}}^{^{\bot}}$ for $n\in\mathbb{N}\cup\infty.$
\end{theorem}

\section{$\mathbf{H}$-relations in the lattices of ideals of Banach algebras}

\subsection{$\mathbf{H}$-relations in the lattices of ideals of Banach
algebras}

Let $A$ be a Banach algebra. Denote by Id$_{A}$ the set of all closed
two-sided ideals in $A$ (we call them \textit{ideals}). Let $LR\mathcal{(}A)$
be the Banach algebra of operators on $A$ generated by all operators of left
and right multiplication by elements from $A.$ Then Id$_{A}=$ Lat
$LR\mathcal{(}A)$ is a sublattice of Cl($A).$

The following Proposition was obtained in \cite[Example 4.6]{Ru1} and
rediscovered in \cite[Lemma 4.9(iv)]{Ki} (see also \cite[Lemma 8.1]{W} and
\cite[Proposition 2.4]{Di}).

\begin{proposition}
\label{P3.3}Let $L$ and $R$ be closed left and right ideals of $A$. If $L$ has
a bounded right approximate identity $($a. i.$),$ or $R$ has a bounded left a.
i.$,$ then $L+R$ is closed.
\end{proposition}

\begin{corollary}
\label{C3.4}Let $a$ be a Banach algebra. If each $I\in$ \emph{Id}$_{A}$ has a
bounded left or right a. i.\emph{,} then $\prec$ $=$ $\sqsubset$ $=$
$\subseteq$ and $<_{\mathfrak{g}}$ $=$ $\prec_{\mathfrak{g}}$ $=$
$\sqsubset_{\mathfrak{g}}$ are $\mathbf{HH}$-relations in \emph{Id}$_{A}$.
\end{corollary}

\begin{proof}
It follows from (\ref{9.4}) and Proposition \ref{P3.3} that $\prec$ $=$
$\sqsubset$ $=$ $\subseteq$. From this and from (\ref{9.5}) and (\ref{9.6}) we
have $<_{\mathfrak{g}}=$ $\prec_{\mathfrak{g}}=$ $\sqsubset_{\mathfrak{g}}$ in
Id$_{A}.$ By Theorem \ref{T4.5}, it is an $\mathbf{HH}$-relation.\bigskip
\end{proof}

Let $\ll$ be a relation in Id$_{A}$. As in (\ref{3.3}), for $L\in$ Id$_{A},$
let
\begin{equation}
\sigma_{_{\ll}}(L)=\cap\{J\in\text{ Id}_{A}\text{: }J\ll L\}\text{ and
}s^{^{_{\ll}}}(L)=\overline{\sum\{J\in\text{ Id}_{A}\text{: }L\ll J\}}.
\label{9.7}%
\end{equation}
For an $\mathbf{H}$-relations $\ll,$ $\ll^{\triangleright}\ $is an
$\mathbf{R}$-order in Id$_{A}$ and the ideal
\[
\mathfrak{r}_{_{L}}\overset{(\ref{2.14})}{=}\overline{\text{span\{}J\in\text{
Id}_{A}\text{: }L\ll^{\triangleright}\text{ }J\}}\text{ is the unique }%
\ll^{\triangleright}\text{-radical in }[L,A]\subseteq\text{ Id}_{A}.
\]
If $\ll$ is a dual $\mathbf{H}$-relations then $\ll^{\triangleleft}\ $is a
dual $\mathbf{R}$-order in Id$_{A}$ and the ideal%
\[
\mathfrak{p}_{_{L}}\overset{(\ref{2.13})}{=}\cap\{J\in\text{ Id}_{A}\text{:
}J\ll^{\triangleleft}L\}\text{ is the unique dual }\ll^{\triangleleft
}\text{-radical in }[\{0\},L]\subseteq\text{Id}_{A}.
\]

\begin{corollary}
\label{C4.2n}The results of Theorem \emph{\ref{T4.6}} and Proposition
\emph{\ref{P9.5} }hold for $X=A,$ where $Q=$ \emph{Id}$_{A}$ is a sublattice
of \emph{Cl(}$A)$ and the word $"$spaces$"$ replaced by $"$ideals$"$ of $A.$
\end{corollary}

By Corollary \ref{T9.2}, all $\ll_{n}$ are $\mathbf{HH}$-relations in Id$_{A}$
and $\ll_{\infty}$ is an $\mathbf{HH}$-order. Proposition \ref{P9.5} yields

\begin{corollary}
\label{C.0}Let $A$ be a separable Banach algebra.\smallskip

\emph{(i) }Let $\sigma_{_{\ll_{_{\infty}}}}(A)=\{0\}.$ Then there is a chain
$...I_{n}\subset...\subset I_{0}=A$ of ideals of finite codimension in $A$
such that $\cap_{n}I_{n}=\{0\}.$ Each $\{0\}\neq J\in$ \emph{Id}$_{A}$
contains $I\in$ \emph{Id}$_{A}$ such that $0<\dim(J/I)<\infty.\smallskip$

\emph{(ii) }Let $s^{^{_{\ll_{_{\infty}}}}}(\{0\})=A.$ Then there is a chain
$\{0\}=I_{0}\subset..\subset I_{n}..$ of ideals such that $\dim I_{n}<\infty$
and $\overline{\sum{}_{n=1}^{\infty}I_{n}}=A.$ Each $A\neq J\in\emph{\ Id}%
_{A}$ is contained in $I\in\emph{\ Id}_{A}$ such that $0<\dim(I/J)<\infty.$
\end{corollary}

Since ideals of a Banach algebra $A$ are precisely invariant subspaces of the
algebra $LR(A),$ it is natural to call superinvariant subspaces of $A$ by
\textit{superideals. }Thus $I\in$ Id$_{A}$ is a superideal of $A$, if (see
$(\ref{3.1}))$ it is invariant for all operators $S$ from the Lie subalgebra
Nor $LR\mathcal{(}A)$ of the algebra $B(A)$:%
\begin{equation}
\text{Nor }LR\mathcal{(}A)=\{S\in B(A)\text{: }ST-TS\in LR\mathcal{(}A)\text{
for all }T\in LR\mathcal{(}A)\}. \label{4.31}%
\end{equation}

A bounded operator $\delta\in B(A)$ is a derivation on $A$ if $\delta
(ab)=\delta(a)b+a\delta(b)$ for $a,b\in A.$

\begin{theorem}
\label{T9.6}\emph{(i) Nor} $LR\mathcal{(}A)$ contains all \textit{derivations}
of $A.\smallskip$

\emph{(ii)\ \ }If $I\in$ \emph{Id}$_{A}$ then each operator $S\in$ \emph{Nor}
$LR\mathcal{(}A)$ maps $\overline{I^{2}}$ in $I$\emph{:} $Sx\in I$ for
$x\in\overline{I^{2}}.\smallskip$

\emph{(iii) }If $\overline{I^{2}}=I$ $($in particular$,$ if $I$ has a left
$($right$)$ a. i.$)$ then $I$ is a superideal.
\end{theorem}

\begin{proof}
(i) Let $L_{a},$ $R_{a}$ be the operators of left and right multiplication by
$a\in A$ on $A$. Then%
\[
\lbrack\delta,L_{a}]x=\delta L_{a}x-L_{a}\delta(x)=\delta(ax)-a\delta
(x)=\delta(a)x=L_{\delta(a)}x\text{ for all }x\in A.
\]
Hence $[\delta,L_{a}]=L_{\delta(a)}\in LR\mathcal{(}A).$ Similarly,
$[\delta,R_{a}]=\delta R_{a}-R_{a}\delta=R_{\delta(a)}\in LR\mathcal{(}A).$
Since the algebra $LR\mathcal{(}A)$ is generated by the sums of products of
the operators $L_{a}$ and $R_{b}$ and since $[\delta,TS]=[\delta
,T]S+T[\delta,S]$ for all $T,S\in B(A),$ we have $\delta\in$ Nor
$LR\mathcal{(}A).$

(ii) Let $S\in$ Nor $LR\mathcal{(}A)$ and $a,b\in I.$ Then
\[
S(ab)=SR_{b}a=R_{b}Sa+[S,R_{b}]a=(Sa)b+[S,R_{b}]a.
\]
As $Sa\in A$ and $b\in I,$ we have $(Sa)b\in I.$ As $[S,R_{b}]\in
LR\mathcal{(}A)$ by $(\ref{4.31}),$ we have $[S,R_{b}]a\in I.$ Hence $S(ab)\in
I.$ Since $S$ is linear and bounded, it maps $\overline{I^{2}}$ into $I.$ Part
(iii) follows from (ii). \bigskip
\end{proof}

To see an example of a Banach algebra that has no non-trivial superideal,
consider a Banach space $A$ with trivial multiplication: $ab=0$ for $a,b\in
A.$ Then $A$ is a Banach algebra, $LR(A)=\{0\}$ and Nor $LR\mathcal{(}%
A)=B(A).$ As $B(A)\ $has no non-trivial invariant subspaces, $A$ has no
superideals. On the other hand, all closed subspaces of $A$ are ideals of $A.$

Denote by $\Sigma_{A}$ the set of all closed subalgebras of $A$ of finite codimension.

\begin{proposition}
\label{C-Laf}For each Banach algebra $A,$ $\sigma_{_{\ll_{_{\infty}}}}%
(A)=\cap_{_{S\in\Sigma_{A}}}S.$
\end{proposition}

\begin{proof}
As $\{J\in$ Id$_{A}$: $J\ll_{_{\infty}}A\}\subseteq\Sigma_{A},$ we have
$\cap_{_{S\in\Sigma_{A}}}S\subseteq\sigma_{_{\ll_{_{\infty}}}}(A)$ by
(\ref{9.7}). In \cite{L} it was proved that each $S\in\Sigma_{A}$ contains an
ideal of finite codimension. Thus $\sigma_{_{\ll_{_{\infty}}}}(A)\subseteq
\cap_{_{S\in\Sigma_{A}}}S.$ So $\sigma_{_{\ll_{_{\infty}}}}(A)=\cap
_{_{S\in\Sigma_{A}}}S.$
\end{proof}

\subsection{$\mathbf{H}$-relations in the lattices of all ideals of
C*-algebras.}

Let $A$ be a C*-algebra. Since ideals of $A$ have bounded a.i., $I+J\in$
Id$_{A}$ for all $I,J\in$ Id$_{A},$ by Proposition \ref{P3.3} (see \cite{D}).
Thus Id$_{A}$ is a sublattice (not a complete sublattice) of Ln($A).$ So
Id$_{A}$ is a modular lattice and $<_{\mathfrak{g}}$ is an $\mathbf{HH}%
$-relation in Id$_{A}$ by Proposition \ref{P9.1}. It also follows from
Corollary \ref{C3.4} that the relations $\prec,$ $\sqsubset$ and $\subseteq$
coincide in Id$_{A}$ and $<_{\mathfrak{g}}$ $=$ $\prec_{\mathfrak{g}}$ $=$
$\sqsubset_{\mathfrak{g}}.$

The complete lattice Id$_{A}$ has many $\mathbf{H}$-relations (some of them
stronger than $<_{\mathfrak{g}})$ that depend on the nature of the quotient
ideals $J/I.$ A large variety of such relations was investigated in
\cite{KST4}. In this section we briefly consider some of them.

Let $\mathfrak{A}$ be the set of all C*-algebras. We say that a subclass $P$
of $\mathfrak{A}$ is a \textit{property}$\mathfrak{,}$ if%
\begin{equation}
\{0\}\in P\text{ and }A\in P\text{ implies }B\in P\text{ for all }B\approx
A\text{.} \label{1.7}%
\end{equation}
If $A\in P,$ we say that $A$ has property $P,$ or $A$ is a $P$%
-\textit{algebra}. For example, the classes $CCR$ of all CCR-algebras and
$GCR$ of all GCR-algebras are properties.

A standard way to define a relation in Id$_{A}$ is to consider a property $P$
and to write%
\begin{equation}
I\ll_{_{P}}J\text{ if }I\subseteq J\text{ in Id}_{A}\text{ and }J/I\text{ is a
}P\text{-algebra},\text{ i.e., }J/I\in P. \label{4.1}%
\end{equation}
A property $P$ on $\mathfrak{A}$ is \textbf{lower stable} if $A\in P$ implies
that all $I\in$ Id$_{A}$ belong to $P;$\emph{ }$P$ is \textbf{upper stable }if
$A\in P$ implies that the quotients $A/I\in P$ for all $I\in$ Id$_{A}.$

\begin{theorem}
\label{T9.4}\emph{(i) }A property\emph{\ }$P$ is upper stable if and only if
$\ll_{_{P}}$ is an $\mathbf{H}$-relation$.$

\qquad Then $\ll_{_{P}}^{\triangleright}$ is an $\mathbf{R}$-order and
$\emph{Id}_{A}$ has the $\ll_{_{P}}^{\triangleright}$-radical $\mathfrak{r}%
_{_{P}}\emph{:}$ $\{0\}\ll_{_{P}}^{\triangleright}\mathfrak{r}_{_{P}}$
$\overleftarrow{\ll_{_{P}}^{\triangleright}}$ $A.\smallskip$

\emph{(ii)\ \ }A property $P$ is lower stable if and only if\emph{ }$\ll
_{_{P}}$ is a dual $\mathbf{H}$-relation\emph{. }

\qquad Then $\ll_{_{P}}^{\triangleleft}$ is a dual $\mathbf{R}$-order and
$\emph{Id}_{A}$ has the dual $\ll_{_{P}}^{\triangleleft}$-radical
$\mathfrak{p}_{_{P}}\emph{:}$ $\{0\}$ $\overrightarrow{\ll_{_{P}%
}^{\triangleleft}}$ $\mathfrak{p}_{_{P}}\ll_{_{P}}^{\triangleleft}A.$
\end{theorem}

\begin{proof}
(i) Let $I\ll_{_{P}}J$ in Id$_{A},$ $I\neq J,$ and $I\subseteq K\in$ Id$_{A}.$
Then $(K\cap J)/I$ is an ideal of $J/I.$ As $J/I\in P$ and $P$ is upper
stable, the quotient $(J/I)/((K\cap J)/I)\in P.$ As $(J/I)/((K\cap
J)/I)\approx J/(K\cap J),$ we have $J/(K\cap J)\in P.$ By Corollary 1.8.4
\cite{D}, $J/(K\cap J)\approx(J+K)/K,$ so that $(J+K)/K\in P.$ Thus
$K\ll_{_{P}}(J+K)=J\vee K.$ So, by Lemma \ref{le1}, $\ll_{_{P}}$ is an
$\mathbf{H}$-relation.

Conversely, if $\ll_{_{P}}$ is an $\mathbf{H}$-relation and $A\in P$ then
$\{0\}\ll_{_{P}}A$ by (\ref{4.1}). Hence, for $I\in$ Id$_{A},$ $I\ll_{_{P}%
}(A+I)=A$ by Lemma \ref{le1}. Thus $A/I\in P$ by (\ref{4.1}). So $P$ is upper stable.

The rest of (i) follows from Theorem \ref{inf}. Part (ii) can be proved
similarly.\bigskip
\end{proof}

Now we can treat Id$_{A}$ as a lattice with two relations: $\subseteq$ and
$\ll_{_{P}}.$ With this approach we can construct a large variety of
$\mathbf{H}$- and dual $\mathbf{H}$-relations and, using methods of the
lattice theory, obtain many results in the theory of C*-algebras. Thus
Theorems \ref{T4.6} and \ref{T9.4} yield

\begin{theorem}
\label{T9.5}\emph{I. }Let $P$ be an upper stable property\emph{, }so that
$\ll_{_{P}}$ is an $\mathbf{H}$-relation$.$ Let $A\in\mathfrak{A}.\smallskip$

\emph{(i) \ \ }The radical $\mathfrak{r}_{_{P}}$ is the \textbf{largest} ideal
in $A$ such that there is an ascending $\ll_{_{P}}$-series of ideals from
$\{0\}$ to $\mathfrak{r}_{_{P}}$.\smallskip

\emph{(ii) \ }If $\mathfrak{r}_{_{P}}\nsubseteq I\neq A,$ then there is $J\in$
Id$_{A}$ such that $J/I$ is a $P$-algebra.

\emph{(iii)}$\ \mathfrak{r}_{_{P}}$ is the \textbf{smallest} out of all ideals
$J$ such that $A/J$ has no $P$-ideals.\smallskip

\emph{II. }Let $P$ be a lower stable property\emph{, }so that $\ll_{_{P}}$ is
a dual $\mathbf{H}$-relation$.$ Let $A\in\mathfrak{A}.\smallskip$

\emph{(i) \ \ }The dual radical $\mathfrak{p}_{_{P}}$ is the \textbf{smallest}
ideal in $A$ such that there is an descending $\ll_{_{P}}$-series of ideals
from $A$ to $\mathfrak{p}_{_{P}}$.\smallskip

\emph{(ii) \ }If $\{0\}\neq I\nsubseteq\mathfrak{p}_{_{P}}$ then there is
$J\in$ Id$_{A}$ such that $I/J$ is a $P$-algebra.\smallskip

\emph{(iii)}$\ \mathfrak{p}_{_{P}}$ is the \textbf{largest} ideal such that
all its quotients are not $P$-algebras.
\end{theorem}

\begin{proof}
I. Parts (i) and (ii) follow from Theorem \ref{T4.6} (ii).

(iii). By Theorem \ref{T4.6} (ii), $\mathfrak{r}_{_{P}}$ has no $\ll_{_{P}}%
$-successor. So $A/\mathfrak{r}_{_{P}}$ has no $P$-ideals.

Let $A/J$ have no $P$-ideals for $J\in$ Id$_{A}.$ If $\mathfrak{r}_{_{P}%
}\nsubseteq J$ then, by Theorem \ref{T4.6} (ii) 2), $A/J$ has $P$-ideals. So
$\mathfrak{r}_{_{P}}\subseteq J$ which proves (iii). The proof of part II is
similar.\bigskip
\end{proof}

Each automorphism $\phi$ of a C*-algebra $A$ generates a lattice automorphism
$\widetilde{\phi}$ of Id$_{A}$. For a property $P\mathfrak{,}$ let $I\ll
_{_{P}}J$ in Id$_{A},$ i.e., $J/I\in P.$ Then $\widetilde{\phi}(I)\subseteq
\widetilde{\phi}(J)$ and the map $\widehat{\phi}$: $J/I\rightarrow
\widetilde{\phi}(J)/\widetilde{\phi}(I)$ defined by $\widehat{\phi}%
(x+I)=\phi(x)+\widetilde{\phi}(I)$ for $x\in J,$ is an isomorphism, i.e.,
$J/I\approx\widetilde{\phi}(J)/\widetilde{\phi}(I),$ so that $\widetilde{\phi
}(J)/\widetilde{\phi}(I)\in P,$ i.e., $\widetilde{\phi}(I)\ll_{_{P}%
}\widetilde{\phi}(J).$ As $\phi^{-1}$ is also an automorphism$,$
$\widetilde{\phi}(I)\ll_{_{P}}\widetilde{\phi}(J)$ implies $I\ll_{_{P}}J.$
Thus $\widetilde{\phi}$ preserves $\ll_{_{P}}.$ So Proposition \ref{T3.11} yields

\begin{corollary}
\label{C9.8}Let $\ll_{_{P}}$ be an $\mathbf{H}$-relation with $\ll_{_{P}%
}^{\triangleright}$-radical $\mathfrak{r}_{_{P}}$ and $\ll_{_{S}}$ be a dual
$\mathbf{H}$-relation with dual $\ll_{_{S}}^{\triangleleft}$-radical
$\mathfrak{p}_{_{S}}.$ Then $\widetilde{\phi}(\mathfrak{r}_{_{P}%
})=\mathfrak{r}_{_{P}}$ and $\widetilde{\phi}(\mathfrak{p}_{_{S}%
})=\mathfrak{p}_{_{S}}$ for all automorphisms $\phi$ of $A.$
\end{corollary}

By Propositions 4.2.4 and 4.3.5 \cite{D}, the properties $CCR$ and $GCR$ of
all CCR- and GCR-algebras are lower and upper stable. For the property $CCR$
this gives a well-known result that the radical $\mathfrak{r}_{_{CCR}}$ is the
largest GCR-ideal in $A$ and $A/\mathfrak{r}_{_{CCR}}$ has no CCR-ideals.
Moreover, if $\mathfrak{r}_{_{CCR}}\nsubseteq I\neq A$ then there is $J\in$
Id$_{A}$ such that $J/I$ is a CCR-algebra. However, the property of all
NGCR-algebras is lower, but not upper stable (see 4.7.4 b) and c) \cite{D}).

Consider now some more lower and upper stable properties. A unital
$A\in\mathfrak{A}$ has real rank zero\textbf{ }\cite{BP}\textbf{ }if its
invertible selfadjoint elements are dense in the set of all selfadjoint
elements of $A.$ A non-unital algebra is real rank zero if its unitization is
real rank zero. Denote by $RZ$ the class of all real rank zero algebras. Then
$RZ$ is a property.

Let $A\in RZ$. By Corollary 2.8 \cite{BP}, each hereditary C*-subalgebra of
$A$ also belongs to $RZ.$ As each ideal $I$ of $A$ is hereditary (Theorem
I.5.3 \cite{Da1}), $I\in RZ$. Let $p$: $A\rightarrow A/I.$ As $\left\Vert
p(x)\right\Vert _{A/I}\leq\left\Vert x\right\Vert _{A},$ it follows that $A/I$
also belongs to $RZ$. This yields

\begin{corollary}
\label{C9.11}The property RZ is lower and upper stable.
\end{corollary}

Recall that $A\in\mathfrak{A}$ is \textit{approximately finite-dimensional
}(AF algebra) if it is the closure of an increasing union of
finite-dimensional *-subalgebras. It is \textit{nuclear }if, for each
C*-algebra $B$, $A\odot B$ has only one C*-norm. Finite-dimensional and all
commutative C*-algebras, $C(\mathcal{H}),$ all AF and all C*-algebras of type
I are nuclear. Then (see Section III.4 \cite{Da1} and Corollary XV.3.4
\cite{T})

\begin{theorem}
\label{NE}\emph{(i) }A C$^{\ast}$-algebra $A$ is nuclear if and only if each
ideal $I$ and $A/I$ are nuclear$.\smallskip$

\emph{(ii) }A C$^{\ast}$-algebra $A$ is an AF algebra if and only if each
ideal $I$ and $A/I$ are AF algebras$.$
\end{theorem}

Denote by $AF$ and $NU$ the classes of all AF and of all nuclear C*-algebras.
By Theorem \ref{NE}, they are lower and upper stable properties. Hence Theorem
\ref{T9.4} yields

\begin{corollary}
\label{C9.7}For each $A\in\mathfrak{A},$ $\ll_{_{CCR}},$ $\ll_{_{GCR}},$
$\ll_{_{RZ}},$ $\ll_{_{AF}},$ $\ll_{_{NU}}$ are $\mathbf{HH}$-relations in
\emph{Id}$_{A}.$
\end{corollary}

We will now discuss transitivity of the relation $\ll_{_{P}}$ for a property
$P.$

\begin{lemma}
\label{L9.2}The relation $\ll_{_{P}}$ is transitive in all lattices
\emph{Id}$_{A},$ $A\in\mathfrak{A},$ if and only if\emph{, }for all
$B\in\mathfrak{A}$ and all $S\in$ \emph{Id}$_{B},$ the condition $S,B/S\in P$
implies $B\in P.$
\end{lemma}

\begin{proof}
Let $B\in\mathfrak{A}$ and $S\in$ Id$_{B}.$ Let $S,B/S\in P.$ Then
$\{0\}\ll_{_{P}}S\ll_{_{P}}B.$ If $\ll_{_{P}}$ is transitive in all lattices
Id$_{A},$ $A\in\mathfrak{A},$ it is transitive in Id$_{B}.$ So $\{0\}\ll
_{_{P}}B.$ Thus $B\in P.$

Conversely, let $A\in\mathfrak{A}$ and $K\ll_{_{P}}I\ll_{_{P}}J$ in Id$_{A}.$
Then $I/K,J/I\in P$ and $J/I\approx(J/K)/(I/K).$ Let,\emph{ }for all
$B\in\mathfrak{A}$ and $S\in$ Id$_{B},$ $S,B/S\in P$ implies $B\in P.$ Set
$B=J/K$ and $S=I/K.$ As $S,B/S\approx J/I\in P,$ we have $J/K=B\in P.$ So
$K\ll_{_{P}}J.$ Thus $\ll_{_{P}}$ is transitive in all Id$_{A},$
$A\in\mathfrak{A}.\bigskip$
\end{proof}

The relation $\ll_{_{CCR}}$ is not transitive. Indeed, let $B=C(H)+\mathbb{C}%
\mathbf{1}_{H}.$ Then $C(H)$ and $B/C(H)\approx\mathbb{C}\mathbf{1}$ are
CCR-algebras, but $B$ is not a CCR-algebra. So, by Lemma \ref{L9.2},
$\ll_{_{CCR}}$ is not transitive.

Similarly, $\ll_{_{RZ}}$ is not transitive, since $A\in RZ$ if and only if all
$I\in$ Id$_{A}$ and $A/I$ are $RZ$-algebras and \textit{all projections in}
$A/I$ \textit{lift to projections in} $A$ (Theorem 3.14 \cite{BP}). On the
other hand, $\ll_{_{AF}}$ and $\ll_{_{NU}}$ are transitive by Theorem
\ref{NE}. It is also easy to show that $\ll_{_{GCR}}$ is transitive.

If an $\mathbf{HH}$-relation $\ll_{_{P}}$ is not transitive then, generally,
the $\ll_{_{P}}^{\triangleright}$-radical $\mathfrak{r}_{_{P}}$ is not a
$P$-algebra. For example, if $B=C(H)+\mathbb{C}\mathbf{1}_{H},$ we have
$\{0\}\ll_{_{CCR}}C(H)\ll_{_{CCR}}B$ and $\mathfrak{r}_{_{CCR}}=B$ is a GCR-
and not a CCR-algebra. Even if $\ll_{_{P}}$ is transitive, $\mathfrak{r}%
_{_{P}}$ may still be not a $P$-algebra. It is a $P$-algebra if and only if
the inductive limits of ascending series of $P$-ideals of $A$ are $P$-algebras
(\cite{KST4}). For example, AF and nuclear algebras have this property (see
\cite[p. 17]{W}).

\begin{corollary}
\label{C8.4}For $A\in\mathfrak{A}$\emph{, }let $\mathfrak{r}_{_{AF}}$ be the
$\ll_{_{AF}}^{\triangleright}$-radical and $\mathfrak{r}_{_{NU}}$ be the
$\ll_{_{NU}}^{\triangleright}$-radical.\smallskip

\emph{(i) (\cite{KST4})} The relations $\ll_{_{AF}},$ $\ll_{_{NU}}$ are
$\mathbf{R}$-orders in \emph{Id}$_{A}$\emph{:} $\ll_{_{AF}}=$ $\ll_{_{AF}%
}^{\triangleright}$ and $\ll_{_{NU}}=$ $\ll_{_{NU}}^{\triangleright
}.\smallskip$

\emph{(ii) (\cite{ST}, \cite{KST4})} $\mathfrak{r}_{_{AF}}$ is the largest
AF-algebra and $\mathfrak{r}_{_{NU}}$ is the largest nuclear algebra in $A.$
\end{corollary}

\subsection{Relation $\ll_{_{\infty}}$ in the lattices of Lie ideals of Banach
Lie algebras.}

A complex Lie algebra $A$ with Lie multiplication $[\cdot,\cdot]$ is a Banach
Lie algebra, if it is a Banach space in some norm $\left\Vert \cdot\right\Vert
$ and $\left\Vert [a,b]\right\Vert \leq C\left\Vert a\right\Vert \left\Vert
b\right\Vert $ for some $C>0$ and for all $a,b\in A.$

A subspace $I$ of $A$ is a \textit{Lie subalgebra }if $[a,b]\in I$ for all
$a,b\in I$; it is a \textit{Lie ideal} if $[a,b]\in I$ for all $a\in I,$ $b\in
A.$ Banach algebras are Banach Lie algebras with Lie multiplication
$[a,b]=ab-ba$.

The set Li$_{A}$ of all closed Lie ideals of $A$ is a complete lattice of
invariant subspaces of the subalgebra of $B(A\mathcal{)}$ generated by all
operators ad($a)$: $x\rightarrow\lbrack a,x]$ on $A.$ As in (\ref{9.2}), the
$\mathbf{HH}$-relation $\ll_{_{\infty}}$ in Li$_{A}$ is defined by:
$I\ll_{_{\infty}}J$ if $\dim(J/I)<\infty.$ Then Corollaries \ref{C4.2n} and
\ref{C.0} hold in this setting with two-sided ideals replaced by Lie ideals.

A linear bounded operator $\delta$ on $A$ is a \textit{Lie derivation} if%
\begin{equation}
\delta([a,b])=[\delta(a),b]+[a,\delta(b)]\text{ for }a,b\in A. \label{1.1}%
\end{equation}
The set $\mathfrak{D}\left(  A\right)  $ of all Lie derivations on $A$ is a
closed Lie subalgebra of the algebra $B\left(  A\right)  $ of all bounded
operators on $A.$ Each $a\in A$ defines a Lie derivation ad$\left(  a\right)
$ on $A.$

An ideal $I\in$ Li$_{A}$ is called \textit{characteristic} if $\delta
(I)\subseteq I$ for all $\delta\in\mathfrak{D}\left(  A\right)  .$ By
(\ref{1.1}), the centre of $A$ is a characteristic Lie ideal. If $A$ is
commutative then $\{0\}$ and $A$ are the only characteristic Lie ideals of
$A$, as $\mathfrak{D}\left(  A\right)  =B(A\mathcal{)}$ and only $\{0\}$ and
$A$ are invariant for $B(A).$

If $\delta\in\mathfrak{D}\left(  A\right)  $ then $e^{t\delta}=\sum
_{n=0}^{\infty}t^{n}\delta^{n}/n!$ for $t\in\mathbb{R},$ is a one-parameter
group of bounded Lie automorphism of $A$: $e^{t\delta}([a,b])=[e^{t\delta
}(a),e^{t\delta}(b)]$ for $a,b\in A.$ Hence
\begin{equation}
I\in\text{ Li}_{_{A}}\text{ }\Rightarrow e^{t\delta}(I)\in\text{ Li}_{_{A}%
}\text{ for all }t\in\mathbb{R}\text{ and }\delta\in\mathfrak{D}\left(
A\right)  . \label{10.1}%
\end{equation}
Moreover, since $\delta(a)=\underset{t\rightarrow0}{\lim}(e^{t\delta
}(a)-a)/t\in I$ for all $a\in I,$ it follows that%
\begin{equation}
I\text{ is a characteristic Lie ideal if and only if }e^{t\delta}(I)=I\text{
for all }t\in\mathbb{R}\text{ and }\delta\in\mathfrak{D}\left(  A\right)  .
\label{10.2}%
\end{equation}

Denote by Li$_{_{A}}^{\text{ch}}$ the subset of Li$_{_{A}}$ of all closed
characteristic Lie ideals of $A.$ We have that%
\begin{equation}
\text{if }J\in\text{ Li}_{_{A}}^{\text{ch}}\text{ and }I\text{ is a
characteristic Lie ideal of }J,\text{ then }I\in\text{ Li}_{_{A}}^{\text{ch}}.
\label{4.41}%
\end{equation}
Indeed, $\delta|_{J}\in\mathfrak{D}\left(  J\right)  $ for all $\delta
\in\mathfrak{D}\left(  A\right)  $, as $J\in$ Li$_{_{A}}^{\text{ch}}.$ Then
$\delta(I)=\delta|_{J}(I)\subseteq I,$ as $I\in$ Li$_{_{J}}^{\text{ch}}$.
Thus, as ad($a)\in\mathfrak{D}\left(  A\right)  ,$ for each $a\in A$, $[a,I]=$
ad($a)(I)\subseteq I.$ Hence $I\in$ Li$_{_{A}}^{\text{ch}}.$

The existence of Lie and characteristic Lie ideals of finite codimension in
Banach Lie algebras was studied in \cite{KST2}. We will use below the
following result obtained there.

\begin{theorem}
\label{KST1}If a Banach non-commutative Lie algebra has a closed proper Lie
subalgebra of finite codimension$,$ then it has a proper closed characteristic
Lie ideal of finite codimension.
\end{theorem}

As the intersection and the closed sum of any family of characteristic Lie
ideals is a characteristic Lie ideal, Li$_{_{A}}^{\text{ch}}$ is a complete
sublattice of Li$_{_{A}}$. Denote by $\mathfrak{p}_{_{\infty}}^{\text{ch}}(A)$
the dual $\ll_{_{\infty}}^{\triangleleft}$-radical in Li$_{_{A}}^{\text{ch}}$
and by $\Sigma_{_{A}}$ the set of all closed Lie subalgebras of finite
codimension in $A$.

We now prove the following version of Proposition \ref{C-Laf} for Banach Lie algebras.

\begin{theorem}
\label{comdiff}Let \emph{Li}$_{_{A}}^{\text{\emph{ch}}}$ have no commutative
infinite-dimensional Lie ideals and $\cap_{_{L\in\Sigma_{A}}}L=\{0\}.$
Then\emph{,}\ for each $\{0\}\neq K\in$ \emph{Li}$_{_{A}}^{\text{\emph{ch}}},$
there is $I\in$ \emph{Li}$_{_{A}}^{\text{\emph{ch}}}$ such that $K\neq
I\ll_{_{\infty}}I,$ and $\mathfrak{p}_{_{\infty}}^{\text{\emph{ch}}%
}(A)=\{0\}.$
\end{theorem}

\begin{proof}
Set $\mathfrak{p}=\mathfrak{p}_{_{\infty}}^{\text{ch}}(A).$ Let $\{0\}\neq
K\in$ Li$_{_{A}}^{\text{ch}}.$ If $\dim K<\infty$ then $\{0\}\ll_{_{\infty}}K$
and $\{0\}\in$ Li$_{_{A}}^{\text{ch}}.$

Let $\dim K=\infty.$ As $\cap_{_{L\in\Sigma_{_{A}}}}L=\{0\},$ there is
$L\in\Sigma_{_{A}}$ that does not contain $K.$ Then $K\cap L$ is a proper
closed Lie subalgebra of $A.$ Replacing $M$ by $A$ in Lemma \ref{Clos}(i), we
have that $K\cap L$ has finite codimension in $K$. By our assumption, $K$ is
non-commutative. So, by Theorem \ref{KST1}, there is $I\in$ Li$_{_{K}%
}^{\text{ch}}$ of finite codimension, i.e., $\dim(K/I)<\infty.$ By
(\ref{4.41}), $I\in$ Li$_{_{A}}^{\text{ch}},$ so that $I\ll_{_{\infty}}K$ in
Li$_{_{A}}^{\text{ch}}.$

If $\mathfrak{p}\neq\{0\}$ then, by above, $I\ll_{_{\infty}}\mathfrak{p}$ for
some $\mathfrak{p}\neq I\in$ Li$_{_{A}}^{\text{ch}}$ which contradicts Theorem
\ref{T4.6}(i).$\bigskip$
\end{proof}

The condition that Li$_{_{A}}^{\text{ch}}$ has no commutative
infinite-dimensional Lie ideals in Theorem \ref{comdiff} is essential. To show
this, consider the following construction of Banach Lie algebras \cite[Sec
1.8]{Bo}.

Let $X$ be a Banach space and\emph{ }$\mathcal{L}$ a closed Lie subalgebra of
$B(X).$ The direct sum $A=\mathfrak{L}\oplus^{\text{id}}X$ (\textit{semidirect
product of }$\mathfrak{L}$ and $X$) endowed with Lie multiplication and norm%
\[
\left[  (a;x),(b;y)\right]  =\left(  [a,b];ay-bx\right)  \text{ and
}\left\Vert (a;x)\right\Vert =\max\left\{  \left\Vert a\right\Vert ,\left\Vert
x\right\Vert \right\}  \text{ for }a,b\in\mathfrak{L},\text{ }x,y\in X,
\]
is a Banach Lie algebra. Clearly, $J=\{0\}\oplus^{\text{id}}X$ is a closed
commutative Lie ideal of $A$.

\begin{example}
\label{E3}Let $\dim X=\infty,$ $\dim\mathfrak{L}<\infty$ and let
$\mathfrak{L}$ have no invariant subspaces in $X.$ Then $\cap_{_{L\in
\Sigma_{_{A}}}}L=\{0\}.$ However$,$\emph{ }$\mathfrak{p}_{_{\infty}%
}^{\text{\emph{ch}}}(A)=J\neq\{0\},$ as \emph{Li}$_{_{A}}^{\text{\emph{ch}}}$
has a commutative Lie ideal $J$.
\end{example}

\begin{proof}
If $Y$ is a closed subspace of codimension $1$ in $X$ then $\{0\}\oplus
^{\text{id}}Y\in\Sigma_{A}$, as $\dim\mathfrak{L}<\infty.$ Since the
intersection of all such Lie subalgebras is\emph{ }$\{0\},$ we have
$\cap_{_{L\in\Sigma_{A}}}L=\{0\}.$

As $\mathfrak{L}$ is irreducible in $X$, each $K\in$ Li$_{A}$ has form
$K=I\oplus^{\text{id}}X,$ where $I\in$ Li$_{\mathfrak{L}};$ and $J$ is the
unique minimal non-zero closed Lie ideal of $A$. By (\ref{10.1}), $e^{t\delta
}(J)$ is a minimal Lie ideal for each $t\in\mathbb{R}$ and $\delta
\in\mathfrak{D}\left(  A\right)  ,$ so that $e^{t\delta}(J)=J.$ By
(\ref{10.2}), $J\in$ Li$_{_{A}}^{\text{ch}}.$ Thus the condition in Theorem
\ref{comdiff} fails and the theorem does not hold since $\{0\}\neq
\mathfrak{p}_{_{\infty}}^{\text{ch}}(A)=J\ll_{_{\infty}}A$.\bigskip
\end{proof}

Let $\mathfrak{p}_{_{\infty}}(\mathfrak{L})$ be the dual $\ll_{_{\infty}%
}^{\triangleleft}$-radicals in Li$_{_{\mathfrak{L}}}$ and $\mathfrak{p}%
_{_{\infty}}(X)$ be the dual $\ll_{_{\infty}}^{\triangleleft}$-radicals in Lat
$\mathfrak{L}$ -- the lattice of all $\mathfrak{L}$-invariant subspaces in
$X$. Let $\mathfrak{p}_{_{\infty}}(A)$ be the dual $\ll_{_{\infty}%
}^{\triangleleft}$-radical in Li$_{_{A}}$.

\begin{proposition}
\label{P8.2n}Let $A=\mathfrak{L}\oplus^{\text{\emph{id}}}X.$ If $\mathfrak{p}%
_{_{\infty}}(\mathfrak{L})=\{0\}$ then $\mathfrak{p}_{_{\infty}}%
(A)=\{0\}\oplus^{\text{\emph{id}}}\mathfrak{p}_{_{\infty}}(X).$
\end{proposition}

\begin{proof}
As $\mathfrak{p}_{_{\infty}}(\mathfrak{L})=\{0\}$, there is a descending
$\ll_{_{\infty}}$\textit{-}series of Lie ideals $\left(  I_{\lambda}\right)
_{1\leq\lambda\leq\gamma}$ of $\mathfrak{L}$ from $\mathfrak{L}$ to $\{0\}.$
Then $C_{_{\mathfrak{L}}}=\{I_{\lambda}\oplus^{\text{id}}X\}_{1\leq\lambda
\leq\gamma}$ is a descending $\ll_{_{\infty}}$\textit{-}series of Lie ideals
of $A$ from $A$ to $\{0\}\oplus^{\text{id}}X.$

Let $C_{_{X}}=\{Y_{\omega}\}_{1\leq\omega\leq\alpha}$ be a descending
$\ll_{_{\infty}}$\textit{-}series of subspaces in Lat $\mathfrak{L}$ from $X$
to $\mathfrak{p}_{_{\infty}}(X).$ As $\mathfrak{p}_{_{\infty}}(X)$ has no
invariant subspaces of finite codimension, the Lie ideal $J=\{0\}\oplus
^{\text{id}}\mathfrak{p}_{_{\infty}}(X)$ contains no Lie ideals of $A$ of
finite codimension and $C=C_{_{\mathfrak{L}}}\cup C_{_{X}}$ is a descending
$\ll_{_{\infty}}$\textit{-}series of Lie ideals from $A$ to $J.$ So, by
Corollary \ref{C4.2n}, $\mathfrak{p}_{_{\infty}}(A)=J.$
\end{proof}

E. Kissin: STORM, London Metropolitan University, 166-220 Holloway Road,
London N7 8DB, Great Britain; e-mail: e.kissin@londonmet.ac.uk\medskip

V. S. Shulman: Department of Mathematics, Vologda State University, Vologda,
Russia; e-mail: shulman.victor80@gmail.com\medskip

Yu. V. Turovskii: Konakovo, Russia, e-mail: yuri.turovskii@gmail.com

\end{document}